\documentclass[11pt,reqno]{amsart}
\usepackage{amssymb,amsthm,amsfonts,amstext}
\usepackage{amsmath}
\usepackage{zito}

\usepackage{helvet}         % selects Helvetica as sans-serif font
\usepackage{courier}        % selects Courier as typewriter font
\usepackage{fourier}

\makeatletter
\def\resetMathstrut@{%
  \setbox\z@\hbox{%
    \mathchardef\@tempa\mathcode`\(\relax
    \def\@tempb##1"##2##3{\the\textfont"##3\char"}%
    \expandafter\@tempb\meaning\@tempa \relax
  }%
  \ht\Mathstrutbox@1.2\ht\z@ \dp\Mathstrutbox@1.2\dp\z@
}
\makeatother

\usepackage{dsfont}

\usepackage{color}

\usepackage{hyperref}

\makeatletter \@addtoreset{equation}{section} \makeatother

\renewcommand\thefigure{\thesection.\@arabic\c@figure}
\renewcommand\thetable{\thesection.\@arabic\c@table}
\newtheorem{theorem}{Theorem}[section]
\newtheorem{lemma}[theorem]{Lemma}
\newtheorem{proposition}[theorem]{Proposition}
\newtheorem{corollary}[theorem]{Corollary}

\theoremstyle{remark}
\newtheorem{remark}{Remark}[section]

\newcommand{\mc}[1]{{\mathcal #1}}

\newcommand{\bb}[1]{{\mathbb #1}}

\newcommand{\<}{\langle}
\renewcommand{\>}{\rangle}

\newcommand{\cadlag}{{c\`adl\`ag~}}

\DeclareMathOperator{\Bern}{Bern}

\newcommand{\tdn}{{\bb T_n^d}}
\newcommand{\td}{\bb T^d}
\newcommand{\eps}{\varepsilon}
\newcommand{\subind}{\substack{x \in \bb T_n^d \\ b \in \mc B}}
\newcommand{\subindd}{\substack{x \in \bb T_n^d \\ b' \in \mc B}}
\newcommand{\sqrtt}[1]{\sqrt{\smash{#1} \vphantom{b}}}
\newcommand{\wt}{\widetilde}
\newcommand{\inT}{\in [0,T]}

\newcommand{\varH}{\hspace{2pt}
\overline{\hspace{-2pt}\smash{H}\vphantom{t}}}

\title{Non-equilibrium Fluctuations of Interacting Particle Systems}

\author{Milton Jara}
\address{Instituto de Matem\'atica Pura e Aplicada, Estrada Dona Castorina 110, 22460-320  Rio de Janeiro, Brazil.}  \email{mjara@impa.br}

\author{Ot\'avio Menezes}
\address{\noindent Centro de An\'alise Matem\'atica, Geometria e Sistemas Din\^amicos\\
	Instituto Superior T\'ecnico\\
	Av. Rovisco Pais, 1049-001 Lisboa, Portugal.}
\email{otavio.menezes@tecnico.ulisboa.pt}

\begin{document}

\begin{abstract}
We obtain the large scale limit of the fluctuations around its hydrodynamic limit of the density of particles of a weakly asymmetric exclusion process in dimension $d \leq 3$. The proof is based upon a sharp estimate on the relative entropy of the law of the process with respect to product reference measures associated to the hydrodynamic limit profile, which holds in any dimension and is of independent interest. As a corollary of this entropy estimate, we obtain some quantitative bounds on the speed of convergence of the aforementioned hydrodynamic limit.
\end{abstract}

\maketitle

\section{Introduction}

One of the main problems in non-equilibrium statistical mechanics is the derivation of macroscopic equations, like Euler, heat or Navier-Stokes equations, as large-scale limits of microscopic models of interacting particle systems. For deterministic microscopic systems, this derivation is acknowledged to be very difficult, and a bottle of a very good wine has been offered for the derivation of Fourier's law from microscopic deterministic dynamics \cite{BonLebR-B}. For stochastic microscopic dynamics, the situation is much better understood. These stochastic systems are known in the literature as {\em interacting particle systems}. Derivation of partial differential equations as limits of properly rescaled observables of these systems are known as {\em hydrodynamic limits}, see \cite{D-MPre} and \cite{Spo} for reviews up to the early 90's. These hydrodynamic limits can be understood as a law of large numbers in functional spaces. One of the advantages of the derivation of hydrodynamic limits from stochastic models, is that one can try to go a step further and look at the {\em fluctuation theory} of these models, namely the central limit theorems and large deviations principles associated to these law of large numbers. These limits contain non-trivial physical information about the underlying physics, which is not available solely looking at the hydrodynamic limit. 

In \cite{D-MPre}, the authors give an exhaustive review of the method of $v$-functions, which in modern terminology allows to derive hydrodynamic limits of particle systems which are perturbations of stochastically integrable, symmetric systems, like the symmetric exclusion process or systems of independent particles. This method is good enough to also obtain the corresponding central limit theorem, but it is not suitable for looking at large deviations. 

A major breakthrough in the theory of hydrodynamic limits was obtained by Guo, Papanicolaou and Varadhan in \cite{GuoPapVar}, where the authors introduced the {\em entropy method} to derive hydrodynamic limits. This method is very robust, it does not rely on integrability and it can also be used to derive large deviations principles \cite{KipOllVar} and {\em equilibrium} central limit theorems \cite{Cha}. The main drawback of the entropy method is that it is based on the explicit knowledge of the invariant measures of the stochastic dynamics. This point was subsequently addressed by the {\em relative entropy method} of Yau \cite{Yau}, which only requires the knowledge of a good approximation of the invariant measure. These methods are extensively reviewed in \cite{KipLan}. 

Our main interest is the derivation of {\em non-equilibrium fluctuations} results for {\em diffusive systems}, that is, the derivation of the central limit theorems associated to hydrodynamic limits on which the limiting equation is of parabolic type. Early results used the concept of {\em duality} in order to obtain precise correlation estimates, which subsequently can be used to prove laws of large numbers and central limit theorems, see \cite{D-MFerLeb}, \cite{FerPreVar}. We specially recommend \cite{PreSpo}, where both the hydrodynamic limit and the fluctuation result are discussed together. A remarkable integrability property of the nearest-neighbour, weakly asymmetric exclusion process in dimension $d=1$ was discovered in \cite{Gar}, which used a discrete version of the Cole-Hopf to transform the system into a stochastic system with working dualities. The correlation estimates obtained in this way were used in \cite{DitGar} to derive non-equilibrium fluctuations for the WASEP. The same result was proved in \cite{D-MPreSca} using $v$-functions. The approach of \cite{Gar} had a recent revival as a fundamental tool in order to derive convergence to the KPZ for models presenting {\em stochastic integrability}, see \cite{BerGia}, \cite{AmiCorQua}, \cite{DemTsa} \cite{CorSheTsa}, \cite{Lab}, \cite{CarGiaRedSas}. 

For {\em conservative, diffusive, stationary systems}, that is, symmetric stochastic systems with a conserved quantity and starting from an explicitly known invariant measure, the density fluctuations are well understood. In \cite{BroRos}, the authors introduced the so-called {\em Boltzmann-Gibbs principle}, which roughly states that any space-time average of a local observable of the system can be well approximated by a linear function of the conserved quantity. The original proof of \cite{BroRos} of the Boltzmann-Gibbs principle requires translation invariance and strict scale invariance, and therefore it is not suitable to derive fluctuations of systems with mixed scaling, as required to derive viscous Burgers, Navier-Stokes or reaction-diffusion equations from microscopic systems. In \cite{Cha}, the entropy method of \cite{GuoPapVar} was adapted to give a more general and flexible proof of the Boltzmann-Gibbs principle, and it was used to derive the equilibrium fluctuations of a spatially inhomogeneous, weakly asymmetric exclusion process.

Apart from stochastic systems on which some degree of stochastic integrability is present, up to our knowledge the only work dealing with non-equilibrium fluctuations is \cite{ChaYau}, where the authors consider the one-dimensional Ginzburg-Landau model in dimension $d=1$. A general derivation of non-equilibrium fluctuations of conservative systems has remained largely open since then, and it is mentioned as Conjecture II.3.6 in \cite{Spo}, as a relevant open problem in Lecture 7 of \cite{JenYau} and ``no progress has been made in the last 20 years'' according to \cite{GonLanMil}. For more historical references, see Chapter 11 of \cite{KipLan}.

In this article we develop a general strategy to derive non-equilibrium fluctuations of interacting particle systems. In order to present the main ideas in a more transparent way, instead of searching for generality, we decided to focus on one particular example. In a wink to the seminal articles \cite{Cha}, \cite{KipOllVar} where the general strategy of proof for the equilibrium fluctuations and the large deviations principle were presented, we have chosen to work with a spatially inhomogeneous, weakly asymmetric exclusion process (WASEP), which is a generalization to higher dimensions of the models presented in \cite{Cha}, \cite{KipOllVar}. This model {\em does not} have known explicit invariant measures, and therefore not even in the equilibrium case this model is tractable by previous methods. We will prove convergence of the density fluctuations around its hydrodynamic limit for the WASEP in dimension $d <4$ to the solution of a space-time inhomogeneous, linear stochastic heat equation.

The starting point of our proof is Yau's relative entropy inequality. The hydrodynamic limit of the WASEP considered here has not been stated before in the literature, but the methods of \cite{BerD-SGabJ-LLan} can be used to prove this hydrodynamic limit. The corresponding hydrodynamic equation is 
\[
\partial_t u = \Delta u - 2 \nabla \cdot \big(u(1-u) F\big),
\]
where $F$ is the driving vector field of the weak asymmetry of the model.
As a zeroth-order approximation for the law of the stochastic particle system, we take product measures with a density given by a solution of this hydrodynamic equation. A proof of the convergence of the density of particles to solutions of the hydrodynamic equation is not needed in our proof, and it is actually a corollary of our main results. 
Yau's relative entropy inequality, see Lemma \ref{lA1}, states that the derivative of the relative entropy with respect to given reference measures can be bounded by a {\em dissipative term} which reduces to the Fisher information in the case of diffusions, and an {\em entropy production term}. 
Since the density satisfies the hydrodynamic equation, the entropy production term is of degree $2$ in the sense of the second-order Boltzmann-Gibbs principle introduced in \cite{GonJar}. In order to take advantage of this we use what we call the {\em main lemma}, stated in two versions in Lemmas \ref{l1} and \ref{main_v2}, which allows to estimate the entropy production term by the dissipative term. This lemma can be seen as a version of the two-blocks estimate at the level of fluctuations. In dimension $d=1$, the size of the larger block is macroscopic, while in $d \geq 2$ the size of the larger block is only mesoscopic. This is in line with known computations, which suggest that the relevance of noise overcomes the relevance of quadratic terms on dimensions $d \geq 2$, see \cite{LanOllVar}. After the use of the main lemma, the bound on the derivative of the relative entropy follows easily from the entropy inequality and a Gaussian estimate for the density with respect to the reference measures. 
This estimate, stated in Theorem \ref{t2}, is good enough to give a first new result: in Corollary \ref{c1} we obtain an estimate on the speed of convergence of the hydrodynamic limit of our model. The method of $v$-functions allows to obtain quantitative bounds on the rate of convergence of the hydrodynamic limit, but it has the restrictions already discussed. As far as we know, quantitative hydrodynamic limits are only discussed in \cite{ChaYau} and in the recent references \cite{DizMenOttWu1}, \cite{DizMenOttWu2}. 
In these references, the estimate holds in a stronger topology than ours, but for the moment the approach is restricted to dimension $d=1$ and it requires precise knowledge of the log-Sobolev constant of the model, which is a notoriously difficult problem for discrete systems and also for non-convex, continuous systems. Our quantitative hydrodynamics estimate holds for any dimension $d \geq 1$ and in particular it is not restricted to $d <4$.

Once the entropy estimate is derived, the next target is the Boltzmann-Gibbs principle. By means of a new variational estimate for exponential moments of observables of Markov processes, see Lemma \ref{lA2}, we reduce the proof of the Boltzmann-Gibbs principle to an application of version 2 of the main lemma, as stated in Lemma \ref{main_v2}. We point out that the entropy estimate is needed in the proof of the Boltzmann-Gibbs principle in order to be able to use {\em a priori} bounds to get rid of error terms.

With both the entropy estimate of Theorem \ref{t2} and the Boltzmann-Gibbs principle of Theorem \ref{BG} already proved, the derivation of the large-scale limit of the density fluctuations stated in Theorem \ref{t3} is not difficult to establish. The limiting stochastic heat equation
\[
\partial_t X_t = \nabla \cdot \Big( \nabla X_t - 2 X_t (1-2u_t) F + \sqrt{2u_t(1-u_t)} \dot{\mc W}_t\Big)
\]
can be explicitly solved in terms of the semigroup generated by the operator $\Delta + 2 (1-2u_t) F \cdot \nabla$ and the noise in the equation above. Therefore, convergence is reduced to show that an approximated version of this solution holds at microscopic level. The nonlinear part of the dynamics is controlled by the Boltzmann-Gibbs principle, while the linear part of the dynamics is handled by martingale methods, as introduced in \cite{HolStr}. This is enough to obtain convergence of finte-dimensional laws, which is the statement of Theorem \ref{t3}. Tightness is more difficult to prove, and the arguments deteriorate with dimension. The one-dimensional case is easy due to the nice $O(1)$ estimate on the relative entropy provided by Theorem \ref{t2}. In dimension $d=2$, although $\log n$ seems not so far from constant, a completely different approach is needed in order to prove tightness, and non-optimal results are given. Dimension $d=3$ is even worse, and we had to give up continuity of trajectories in order to obtain tightness. Nevertheless, the convergence of the martingale part holds at optimal topologies.

Now we describe the organization of this article. Since we believe that the ideas exposed in this article could be useful in other problems of interacting particle systems, proofs are very detailed. We hope that the interested reader would find such detail useful, and that the experts can jump some details without much effort. 
From the technical point of view, there are various novelties, which can be of independent interest. In Section \ref{s1} we define the model we consider in this article and we state our main results. In Section \ref{s2} we prove the main lemma, which is the main technical result of this article. In Section \ref{s4} we prove the entropy inequality, relying on the main lemma proved in Section \ref{s2}. In Section \ref{s5} we prove the Boltzmann-Gibbs principle, using version 2 of the main lemma and also Theorem \ref{t2} as input. In Section \ref{s6} we prove Theorem \ref{t3}. In Section \ref{s7} we prove tightness of the density fluctuation field, improving Theorem \ref{t3} to a functional central limit theorem.

In Appendix \ref{sA} we prove Yau's relative entropy inequality and an variational estimate for exponential moments of observables of Markov chains. Since these inequalities could be of independent interest, we present the proof for general Markov chains, and then we specify them for the model considered here. In Appendix \ref{sB} we collect some classical results for solutions of parabolic equations. In Appendix \ref{sD} we collect all results and definitions about topologies and functional spaces we need along the article. In Appendix \ref{sF1} we perform some elementary, although quite tedious computations involving the generator of the processes considered here. In Appendix \ref{sG} we derive a bound known in the literature as the {\em integration-by-parts formula}. In our particular setting, this formula is not as clean as in other situations, since the reference measures are not spatially uniform. This integration-by-parts formula is used to estimate expectations of some functions written in gradient form in terms of the Fisher information. In Appendix \ref{sH} we collect various entropy and concentration inequalities we need along the article. In particular, we provide a version of the so-called {\em Hanson-Wright inequality} for dependent random variables. Since we did not find a reference working in our particular context, we included a full proof. Aiming for clarity, we also included proofs of various other results that could be otherwise omitted. These estimates are key in order to exploit the fact that we can compare averages of local functions with quadratic functions of the density of particles. 

Finally, in Appendix \ref{sI} we give a geometric proof of what we call the {\em flow lemma}, which is basically the construction of an explicit solution for an optimal transport problem. Although the abstract theory of optimal transport could have been used to construct such flows, we found it difficult to extract the properties we need for these flows from these abstract results. This flow lemma is used in order to prove a two-blocks estimate during the proof of the main lemma. Up to our knowledge, there is no previous proof of the two-blocks estimate using flows in the literature.

\subsection*{About notations}

Since this article is quite long, some definitions are repeated along the article for the reader's convenience. Sometimes, these definitions are slightly modified to match the context, but they always coincide in the relevant cases. We use the denomination Proposition for results proved somewhere else, we use Theorem only for original work, and we use Lemma for auxiliary results. Some technical lemmas which could be of relevance in other applications are given own names, like main lemma or flow lemma.

\section{Definitions and results}
\label{s1}

\subsection{The exclusion process}
\label{s1.1}
Let $n \in \bb N$ be a scaling parameter. Let $\tdn := \bb Z^d/n \bb Z^d$ be the discrete, $d$-dimensional torus of size $n$. We will think about $\tdn$ as a discrete approximation of the continuous torus $\td := \bb R^d/\bb Z^d$. We say that $x,y \in \tdn$ are {\em neighbors} if $|y_1-x_1|+\dots+|y_d-x_d|=1$. In this case we write $x \sim y$. This definition induces a graph structure in $\tdn$. From now on, we will always think about $\tdn$ as the graph induced by this neighbors' structure.

Let $\Omega_n := \{0,1\}^{\tdn}$. We denote the elements of $\Omega_n$ by $\eta =\{\eta_x; x \in \tdn\}$ and we call them {\em particle configurations}. If $\eta_x =1$, we say that configuration $\eta$ has a particle at site $x$. Otherwise, we say that the site $x$ is empty. The variables $\eta_x$ are called the {\em occupation variables}.

For $x,y \in \tdn$ and $\eta \in \Omega_n$, let $\eta^{x,y} \in \Omega_n$ be given by
\[
\eta_z^{x,y} =
\left\{
\begin{array}{r@{\;;\;}l}
\eta_y & z=x\\
\eta_x & z =y\\
\eta_z & z \neq x,y.
\end{array}
\right.
\]
In other words, the configuration $\eta^{x,y}$ is obtained from $\eta$ by exchanging the values of the occupation variables at $x$ and $y$. 

For $f: \Omega_n \to \bb R$ and $x \sim y \in \tdn$, let $\nabla_{x,y} f: \Omega_n \to \bb R$ be given by
\[
\nabla_{x,y} f(\eta) = f(\eta^{x,y})-f(\eta) \text{ for any } \eta \in \Omega_n.
\]
We say that a function $r_n: \tdn \times \tdn \to [0,\infty)$ is a {\em jump rate} if $r_n(x,y) =0$ unless $x \sim y$. Notice that $r_n$ depends on $n$ through its domain. Later on we will make a more specific choice for $r_n$. For $f: \Omega_n \to \bb R$, let $L_n f:\Omega_n \to \bb R$ be given by
\[
L_n f(\eta) = \sum_{x,y \in \tdn } r_n(x,y) \eta_x(1-\eta_y) \nabla_{x,y} f(\eta)
\] 
for any $\eta \in \Omega_n$. Notice that the sum can be restricted to the set $\{x,y \in \tdn; x \sim y\}$. This relation defines a linear operator $L_n$, which turns out to be the generator of a continuous-time Markov chain $\{\eta_x^n(t); t \geq 0, x \in \tdn\}$ with state space $\Omega_n$. This chain is known in the literature as the {\em exclusion process} with jump rate $r_n$. We will also use the notation $\eta^n(\cdot)$ in order to refer to this chain.

The dynamics of the process $\eta^n(\cdot)$ is not difficult to describe. The value of $\eta^n_x(t)$ represents the presence or absence of a particle at site $x$ at time $t$. The denomination {\em exclusion} comes from the fact that there is at most one particle at any site at any given time. If the site $x$ is occupied bya  particle and the site $y$ is empty, the particle jumps from $x$ to $y$ at instantaneous rate $r_n(x,y)$. This happens independently for each particle-hole couple in the system.

\subsection{Invariant measures}
\label{s1.2}
We say that the jump rate $r_n$ is {\em irreducible} if for any two sites $x,y \in \tdn$ there exists a sequence $\{x_0=x,x_1,\dots,x_\ell=y\}$ in $\tdn$ such that $r_n(x_{i-1},x_i)>0$ for any $i=1,\dots,\ell$. In that case, the process $\eta^n(\cdot)$ is irreducible on each of the sets
\[
\Omega_{n,k} := \Big\{ \eta \in \Omega_n; \sum_{x \in \tdn} \eta_x =k\Big\}, k=0,1,\dots,n^d.
\]
Equivalently, $\eta^n(\cdot)$ has a unique invariant measure\footnote{Along this article, measure always means {\em probability} measure.} on each of the sets $\Omega_{n,k}$, $k=0,1,\dots,n^d$. Apart from the trivial cases $k=0,n^d$, these measure can not be described explicitly without further assumptions. In the case on which the jump rate $r_n$ is symmetric, that is, $r_n(x,y) = r_n(y,x)$ for any $x,y \in \tdn$, the process $\eta^n(\cdot)$ is reversible with respect to each of the uniform measures in $\Omega_{n,k}$. Equivalently, the product measures $\nu_\rho^n$ defined as
\[
\nu_\rho^n := \bigotimes_{x \in \tdn} \Bern(\rho)
\]
are invariant under the evolution of $\eta^n(\cdot)$ for any $\rho \in [0,1]$. However, they are not ergodic, except for the trivial cases $\rho=0$ or $1$.

It can be verified that the measures $\nu_\rho^n$ are invariant if and only if
\[
\sum_{y \in \tdn} \big( r_n(x,y) -r_n(y,x)\big) =0 
\]
for any $x \in \tdn$. This is a very restrictive condition in $d >1$, which implies that the jump rate $r_n$ is {\em divergence-free}. Necessary and sufficient conditions for the existence of product invariant measures can be derived without too much effort; we just point out that they imply that $r_n$ is either divergence-free or of gradient form.

\subsection{The weakly asymmetric exclusion process}
\label{s1.3}
Let $F: \td \to \bb R^d$ be a vector field. For simplicity, we assume that $F$ is of class $\mc C^\infty$. Let $\mc B=\{e_1,\dots,e_d\}$ be the canonical basis of $\td$ (and also of $\tdn$). Notice that for $x,y \in \tdn$, $x \sim y$ if and only if there exists $b \in \mc B$ such that $x = y+b$ or $y =x+b$. Define $\widetilde{r}_n$ as
\[
\widetilde{r}_n(x,x+b) = n^2\Big( 1+\frac{1}{n} F_b^n(x)\Big), \quad \widetilde{r}_n(x+b,x) = n^2\Big( 1-\frac{1}{n} F_b(x)\Big),
\]
for any $x \in \tdn$ and any $b \in \mc B$, where
\[
F_b^n(x):= F\big(\tfrac{x}{n} +\tfrac{b}{2n}\big) \cdot b.
\]
The $n^2$ factor in the definition of $\widetilde{r}_n$ fixes a {\em diffusive} space-time scaling for the model. In order to $\widetilde{r}_n$ be well defined, we need to assume that $n$ is large enough. To simplify the notation we will define $r_n$ as
\begin{equation}
\label{ancud}
r_n(x,x+b) = n^2\max\Big\{ \frac{1}{2},1+\frac{1}{n} F_b^n(x) \Big\}, \quad r_n(x+b,x) = n^2\max\Big\{ \frac{1}{2},1-\frac{1}{n} F_b^n(x) \Big\}.
\end{equation}
The rates $r_n$ and $\widetilde{r}_n$ coincide if $n \geq 2 \|F\|_\infty$.

The family of exclusion processes $\{\eta^n(\cdot)\}_{n \in \bb N}$ with rate $r_n$ is known in the literature as {\em weakly asymmetric exclusion process} (WASEP). Our choice of rates corresponds to a non-reversible, multidimensional generalization of the model studied in \cite{Cha}. In \cite{KipOllVar} the authors pointed out that understanding the scaling limits of inhomogeneous WASEP was fundamental in order to understand the large deviations principle for the symmetric exclusion process. For $F$ scalar, the asymmetric version of this model was studied in \cite{Bah}, \cite{CovRez}. Notice that the rates $r_n$ are defined evaluating $F$ in a dual lattice. This makes the model a better approximation of the continuous eqeuation.

We will denote by $\bb P_n$ the law of $\eta^n(\cdot)$ in the space $\mc D([0,\infty); \Omega_n)$ of \cadlag paths, and by $\bb E_n$ the expectation with respect to $\bb P_n$. Whenever we need to specify the initial law $\mu$ of $\eta^n(\cdot)$, we will use $\bb P_n^\mu$ and $\bb E_n^\mu$. 
Our main objective will be to prove a central limit theorem for the density of particles with respect to $\bb P_n$, for carefully chosen initial laws for $\eta^n(0)$.

It is possible to examine under which conditions the WASEP has invariant measures of product form. If the density profile $u$ of these measures is not constant, then the field $F$ must satisfy a discrete version of the equation
\[
F = \nabla \log \sqrt{\frac{u}{1-u}}.
\]
Therefore, $F$ must be of gradient form if one is willing to define the rates $r_n$ in such a way that they approximate $F$ and give rise to a process with invariant measures of product form. In that case, the corresponding invariant measure would be reversible.

If $u$ is constant, $F$ must satisfy $\nabla \cdot F =0$, that is, $F$ must be a divergence-free field. Therefore, for generic vector fields $F$, $\eta^n(\cdot)$ will not have invariant measures of product form.

\subsection{Hydrodynamic limit and relative entropy}
\label{s1.4}
In this section we state what we understand by the {\em hydrodynamic limit} of the WASEP. In order to do that, we need to introduce the hydrodynamic equation
\begin{equation}
\label{echid}
\partial_t u = \Delta u - 2 \nabla \cdot \big(u(1-u) F\big).
\end{equation}
Let $u_0: \td \to [0,1]$ an initial density profile and let $\{u(t,x); t \geq 0, x \in \td\}$ be the solution of \eqref{echid} with initial condition $u_0$. The following result is known in the literature as the hydrodynamic limit of the WASEP:

\begin{proposition}
\label{p1}
Let $\{\bar{\nu}_0^n\}_{n \in \bb N}$ be a sequence of measures in $\Omega_n$. Let $\eta^n(\cdot)$ be the WASEP with initial law $\bar{\nu}_0^n$. Assume that
\[
\lim_{n \to \infty} \frac{1}{n^d} \sum_{x \in \tdn} \eta_x^n f \big(\tfrac{x}{n}\big) = \int f(x) u_0(x) dx
\]
for any $f \in \mc C^\infty(\td)$, in probability with respect to $\bar{\nu}_0^n$. Then, for any $f \in \mc C^\infty(\td)$ and any $t \geq 0$,
\[
\lim_{n \to \infty} \frac{1}{n^d} \sum_{x \in \tdn} \eta_x^n(t) f \big( \tfrac{x}{n} \big) = \int f(x) u(t,x) dx,
\]
in probability with respect to $\bb P_n$.
\end{proposition}

In the case on which $F = \nabla V$ for some potential $V$, this theorem corresponds to Theorem 3.1 of \cite{KipOllVar}. Although in that reference the theorem is stated only in $d=1$, the method of proof can be adapted to any dimension $d$. If $F$ is not the gradient of a potential, the method of \cite{KipOllVar} does not apply directly. In that case, Proposition \ref{p1} can be proved using the method of \cite{BerD-SGabJ-LLan}.

In this article we will prove a quantitative version of this theorem, under more restrictive conditions on the initial measures $\bar{\nu}_0^n$. This quantitative estimate will be a consequence of our first main result, which we proceed to describe.

Let $u^n: [0,\infty) \times \tdn \to [0,1]$ be defined as $u_x^n(t) = u\big(t,\tfrac{x}{n} \big)$, where $u$ is a solution of \eqref{echid}.
Let $\mu_t^n$ be the product Bernoulli measure in $\Omega_n$ associated to the profile $u^n_{\cdot}(t)$, that is,
\begin{equation}
\label{pichidangui}
\mu_t^n := \bigotimes_{x \in \tdn} \Bern\big(u_x^n(t)\big).
\end{equation}
Let $f_t^n$ be the density of the law of $\eta^n(t)$ with respect to $\mu_t^n$, and let $H_n(t)$ be the relative entropy of the law of $\eta^n(t)$ with respect to $\mu_t^n$, that is,
\[
H_n(t) = \int f_t^n \log f_t^n d \mu_t^n.
\]
Yau's relative entropy method \cite{Yau} consists of proving that $H_n(t) = o(n^d)$ if initially $H_n(0) =o(n^d)$. Notice that as soon as $u^n_x(t) \neq 0,1$ for any $x$, the relative entropy of $\overline{\nu}$ with respect to $\mu_t^n$ is bounded by $C(u_0) n^d$ for any measure $\overline \nu$ in $\Omega_n$.
Although it may seems as a very modest improvement on the order of magnitude of the relative entropy, Yau's bound is good enough to imply the thesis of Proposition \ref{p1}. Our first main result is a sizeable improvement over the asymptotic behaviour of $H_n(t)$:

\begin{theorem}
\label{t2}
Let $\eps_0, \kappa >0$ be such that
\begin{itemize}

\item $u_x^n(t) \in [\eps_0,1-\eps_0]$ for any $x \in \bb T_n^d$,

\item $n\big|u_{x+b}^n(t) -u_x^n(t)\big|\leq \kappa$ for any $x \in \bb T_n^d$ and any $b \in \mc B$.

\end{itemize}
\
There exists a finite constant $C = C(\eps_0,\kappa)$ such that
\begin{equation}
\label{entropy1}
H_n'(t) \leq C \big( H_n(t) + n^{d-2} g_d(n) \big),
\end{equation}
where
\[
g_d(n) := 
\left\{
\begin{array}{c@{\;;\;}l}
n & d=1\\
\log n & d=2\\
1 & d \geq 3.
\end{array}
\right.
\]
In particular, if $u_0(x) \in (0,1)$ for any $x \in \bb T^d$, then for any $T >0$ there exists a finite constant $C = C(u_0,F,T)$ such that
\begin{equation}
\label{entropy2}
H_n(t) \leq C \big( H_n(0) + n^{d-2}g_d(n)\big) 
\end{equation}
for any $t \in [0,T]$ and any $n \in \bb N$.
\end{theorem}

\begin{remark}
This theorem improves by almost two orders of magnitude the previous bounds obtained by Yau's method in $d \geq2$, while in $d=1$ it gives a bound uniform in $t$, which is the best possible.
\end{remark}

A simple consequence of this theorem is the following corollary, which gives an estimate on the speed of convergence in the hydrodynamic limit stated in Proposition \ref{p1}:

\begin{corollary}
\label{c1}
Assume that there exists a finite constant $C_0$ such that $H_n(0) \leq C_0n^{d-2} g_d(n)$ for any $n \in \bb N$. Under the conditions of Theorem \ref{t2}, for any $p \in [1,2)$ there exists a finite constant $C_1=C_1(p,F,u_0,T,C_0)$ such that
\[
\bb E_n \Big[ \Big| \frac{1}{n^d} \sum_{x \in \tdn} \big( \eta^n_x(t) - u_x^n(t) \big) f \big( \tfrac{x}{n} \big) \Big|^p \Big] \leq \frac{C_1 g_d(n)^{p/2} \|f\|_\infty^{p/2}}{n^p}
\]
for any $t \leq T$ and any $f \in \mc C(\bb T^d)$.
\end{corollary}

\begin{remark}
If $H_n(0) \leq a_n n^d$ with $n^{d-2}g_d(n) \ll a_n \ll n^d$, it can be proved that then the speed of convergence in this corollary is $(a_n n^{-d})^{1/2}$.
\end{remark}

\begin{remark}
If one is more careful about the constants appearing in the proof of this Corollary, it is possible to replace the norm $\|f\|_\infty$ by the weaker norm
\[
\|f\|_{\ell_n^2}:= \Big( \frac{1}{n^d} \sum_{x \in \tdn} f\big(\tfrac{x}{n}\big)^2\Big)^{1/2}.
\]
\end{remark}

\subsection{Density fluctuations} 
\label{s1.5}
In this section we state our next main result, which is the derivation of the central limit theorem associated to the law of large numbers stated in Proposition \ref{p1} and Corollary \ref{c1}. 

For $f \in \mc C^\infty(\bb T^d)$ and $t \geq 0$, let us define
\begin{equation}
\label{losandes}
X_t^n(f) := \frac{1}{n^{d/2}} \sum_{x \in \tdn} \big(\eta_x^n(t) -u_x^n(t)\big)f\big(\tfrac{x}{n}\big).
\end{equation}
In order to simplify the notation we will write $\overline{\eta}_x:= \eta_x^n(t) -u_x^n(t)$. Most of the time we will omit the dependence on $n$ and/or $t$ of various expressions whenever this dependence can be understood from the context.

By duality, \eqref{losandes} defines a process $\{X_t^n; t \geq 0\}$ with values in $H_{-k}(\bb T^d)$ for any $k>d/2$, see the comments after Proposition \ref{pD4.1}. The process $\{X_t^n; t \geq 0\}$ defined in this way is known in the literature as the {\em density fluctuation field} associated to the process $\eta^n(\cdot)$. We will prove the following result:

\begin{theorem}
\label{t3}
Under the setting of Theorem \ref{t2}, assume that $u_0(x) \in [\eps_0,1-\eps_0]$ for any $x \in \bb T^d$ and that $H_n(0) \leq C g_d(n) n^{d-2}$ for some constants $\eps_0$ positive and $C$ finite. In addition, assume that there exists $k\in \bb N$ and random variable $X_0$ with values in $H_{-k}(\bb T^d)$ such that $X_0^n$ converges to $X_0$ in law with respect to the topology of $H_{-k}(\bb T^d)$. Then, in dimension $d < 4$, the finite-dimensional laws of $\{X_t^n; t \geq 0\}$ converge to the finite-dimensional laws of the process $\{X_t; t \geq 0\}$, solution of the space-time inhomogeneous stochastic heat equation
\begin{equation}
\label{SHE}
\partial_t X_t = \nabla \cdot \Big( \nabla X_t - 2 X_t (1-2u_t) F + \sqrt{2u_t(1-u_t)} \dot{\mc W}_t\Big)
\end{equation}
with initial condition $X_0$, where $\dot{\mc W}$ is a vectorial space-time white noise of dimension $d$ and $u_t$ is the solution of the hydrodynamic equation \eqref{echid}.
\end{theorem}

\begin{remark}
This theorem can be improved in two directions. First we can assume a growth with $n$ of $H_n(0)$ which is faster than $g_d(n)n^{d-2}$, and second we can obtain a functional CLT for $\{X_t^n; t \geq 0\}_{n \in \bb N}$ by means of a tightness proof for suitable topologies. Since these improvements are dimension-dependent and quite technical, we decided to state them in a precise way only after the corresponding proofs.
\end{remark}

\section{The main lemma}
\label{s2}

In this section we state and prove a technical lemma which is the main tool to prove both Theorems \ref{t2} and \ref{t3}. In a first reading, the proof of this lemma can be skipped, on which case the reader can pass directly to Section \ref{s4}, although it constitutes the heart of the proof of Theorems \ref{t2} and \ref{t3}.

Let $u: \bb T^d \to (0,1)$ and let $\eps_0>0$, $\kappa$ finite be such that $\eps_0 \leq u_x \leq 1-\eps_0$ for any $x \in \tdn$ and 
\[
n|u_{x+b}-u_x| \leq \kappa \text{ for any } x \in \tdn \text{ and any } b \in \mc B.
\]
Let $\mu$ be the measure in $\Omega_n$ given by
\[
\mu := \bigotimes_{x \in \tdn} \Bern(u_x).
\]
Notice that $\mu = \mu_t^n$ when $u=u^n(t)$. Let $\bb O^-:=\{x \in \bb Z^d;  z_i \leq \text{ for } i=1,\dots,d\}$ be the negative orthant and let $A \subseteq \bb O^-$ be finite. For $n$ large enough, $A$ is projected into $\tdn$ in a canonical way. For $x \in \bb T^d_n$ we define $\omega_x: \Omega_n \to \bb R$ as
\[
\omega_x : = \frac{\eta_x -u_x}{u_x(1-u_x)}.
\]
Since $\eps_0>0$, $\omega_x$ is well defined for any $x \in \tdn$. Now we define $\omega_{x+A}: \Omega_n \to \bb R$ as
\[
\omega_{x+A} := \prod_{y \in A} \omega_{x+y}.
\]
For $b \in \mc B$ and $G: \tdn \to \bb R$ we define
\begin{equation}
\label{vina}
V(G) = V_b(G;A) := \sum_{x \in \tdn} \omega_{x+A} \omega_{x+b} G_x.
\end{equation}
Let $f: \Omega_n \to [0,\infty)$ be a density with respect to $\mu$, that is, $\int f d \mu =1$. Define
\begin{equation}
\label{valparaiso}
\mc D\big( \sqrtt{f}; \mu\big) := \sum_{\subind} \int \big( \nabla_{x,x+b} \sqrtt{f}\big)^2 d \mu,
\end{equation}
\[
H(f;\mu) :=  \int f \log f d \mu.
\]
We have the following result:

\begin{lemma}[Main lemma, v1]
\label{l1}
There exists a finite constant $C = C(\eps_0,A)$ such that for any $G :\tdn \to \bb R$, any density $f$ with respect to $\mu$ and any $\delta >0$,
\[
\int V(G) f d \mu \leq \delta n^2 \mc D\big( \sqrtt{f}; \mu\big) + 
		\frac{C(1+\kappa^2)}{\delta} \big( \|G\|_\infty+\|G\|_\infty^2\big)
			\big(H(f;\mu) + n^{d-2}g_d(n)\big).
\]
\end{lemma}

We will dedicate the rest of this section to the proof of this lemma. The idea of the proof has its roots in the second-order Boltzmann-Gibbs principle introduced in \cite{GonJar}. We will replace the functions $\omega_{x+A}$ and $\omega_{x+b}$ in $V(G)$ by spatial averages over boxes of mesoscopic size $\ell$. We will see that the cost of this replacement can be estimated in terms of $\mc D\big( \sqrtt{f}; \mu\big)$. The main difference with respect to previous works is that we will replace the product $\omega_{x+A} \omega_{x+b}$ by the product of {\em two} local averages, instead of the single local average that appears in the original Boltzmann-Gibbs principle. The proof we describe below incorporates some ideas from \cite{GonJarSim}. 

For $\ell \in \bb N$, let $\Lambda_\ell:= \{z \in \bb Z^d; 0 \leq z_i \leq \ell-1, i=1,\dots,d\}$ be the cube of size $\ell$ and vertex $0$. For $\ell \leq n$, we can identify $\Lambda_\ell$ with a subset of $\tdn$. The same remark remains valid for various functions in $\bb Z^d$ of finite support  that we will define below. Let $p_\ell: \bb Z^d \to [0,1]$ be the uniform measure in $\Lambda_\ell$, that is, $p_\ell(z) = \ell^{-d} \mathds{1}(z \in \Lambda_\ell)$ for any $z \in \bb Z^d$.

Let $q_\ell: \bb Z^d \to [0,1]$ be the measure given by
\[
q_\ell(z) := \sum_{y \in \bb Z^d} p_\ell(y) p_\ell(z-y)
\]
for any $z \in \bb Z^d$. In other words, $q_\ell = p_\ell \ast p_\ell$, the convolution of $p_\ell$ with itself. Notice that $q_\ell$ is supported in $\Lambda_{2\ell-1}$ and that $q_\ell(z) \leq \ell^{-d}$ for any $z \in \bb Z^d$. Let $\ell < n/2$ and define $\omega_x^\ell: \Omega_n \to \bb R$ as
\begin{equation}
\label{sanantonio}
\omega_x^\ell := \sum_{y \in \bb Z^d} \omega_{x+y} q_\ell(y).
\end{equation}
Now we define $V^\ell(G): \Omega_n \to \bb R$ as
\[
V^\ell(G) := \sum_{x\in \tdn} \omega_{x+A} \omega_{x+b}^\ell G_x.
\]
Thanks to the choice $q_\ell = p_\ell \ast p_\ell$, $V^\ell(G)$ can be written as a sum over a product of two averages. In fact, rearranging terms to pass the convolution from $p_\ell$ to $\omega_{x+A}$, we see that
\begin{equation}
\label{curacavi}
V^\ell(G) = \sum_{x \in \tdn} \Big( \sum_{y \in \bb Z^d} \omega_{x-y+A} G_{x-y} p_\ell(y) \Big) \Big( \sum_{z \in \bb Z^d} \omega_{x+z+b} p_\ell(z)\Big).
\end{equation}

Now the idea is to compare $\int V(G) f d\mu$ with $\int V^\ell(G) f d\mu$ using Lemma \ref{lG3}. In order to do it in an efficient way, we will introduce the concept of {\em flow}. 

A flow in $\bb Z^d$ is a function $\phi: \bb Z^d \times \mc B \to \bb R$. We say that the support of $\phi$ is contained in a set $\Lambda \in \bb Z^d$ if for any $(x,b) \in \bb Z^d \times \mc B$ such that $\phi(x;b) \neq 0$, $\{x,x+b\} \subseteq \Lambda$. Let $p,q$ be two measures in $\bb Z^d$. We say that the flow $\phi$ {\em connects} $p$ to $q$ if
\[
p(z) -q(z) = \sum_{b \in \mc B} \big( \phi(z;b) - \phi(z-b;b)\big) \text{ for any } z \in \bb Z^d.
\]
A flow connecting $p$ to $q$ with support contained in a finite set satisfies the following {\em divergence formula}:
\begin{equation}
\label{rancagua}
\sum_{z \in \bb Z^d} f(z) \big( p(z) -q(z)\big) = \sum_{\substack{z \in \bb Z^d\\b \in \mc B}} \phi(z;b) \big( f(z+b)-f(z)\big).
\end{equation}
Let us recall the definition of $g_d(n)$ given in Theorem \ref{t2}:
\[
g_d(n) = 
\left\{
\begin{array}{c@{\;;\;}l}
n & d=1\\
\log n & d=2\\
1 & d \geq 3.
\end{array} 
\right.
\]
We have the following result:

\begin{lemma}[Flow lemma]
\label{flow}
There exists a finite constant $C=C(d)$ such that for any $\ell \in \bb N$ there exists a flow $\phi_\ell$ connecting the point mass at $0$ to $q_\ell$ with support contained in $\Lambda_{2\ell-1}$ such that
\[
\sum_{\substack{z \in \bb Z^d \\ b \in \mc B}} \phi_\ell(z;b)^2 \leq C g_d(\ell);\quad \quad 
\sum_{\substack{z \in \bb Z^d\\b \in \mc B}} \big| \phi_\ell(x;b)\big| \leq C \ell.
\]
\end{lemma}

The proof of this lemma can be found in Appendix \ref{sI}. The dependence in $\ell$ in this lemma is optimal, and it is exactly due to this lemma that the constant $g_d(n)$ appears in Theorem \ref{t2}.

 Using the flow $\phi_\ell$ given by Lemma \ref{flow}, we can compare $V(G)$ and $V^\ell(G)$: using \eqref{rancagua} with $p = \delta_x$, $q = q_\ell(\cdot+x)$ and $f = \omega_x$, we see that 
\[
\omega_x - \omega_{x}^\ell  = \sum_{\substack{z \in \bb Z^d\\b' \in \mc B}} \phi_\ell(z;b') (\omega_{x+z+b'}-\omega_{x+z}).
\]
Therefore,
\begin{equation}
\label{codegua}
\begin{split}
V(G) - V^\ell(G) 
		&= \sum_{x \in \tdn} \omega_{x+A} \sum_{\substack{z \in \bb Z^d\\b' \in \mc B}} \phi_\ell(z; b') (\omega_{x+z+b+b'}-\omega_{x+z+b}) G_x\\
		&=\sum_{x \in \tdn}  \sum_{\substack{z \in \bb Z^d\\b' \in \mc B}} \omega_{x-z-b+A}\phi_\ell(z;b') G_{x-z-b} (\omega_{x+b'}-\omega_x).
\end{split}
\end{equation}
For each $x \in \tdn$ and each $b' \in \mc B$, let $h_x^{\ell,b'}(G): \Omega_n \to \bb R$ be defined as
\begin{equation}
\label{teno}
h_x^{\ell,b'}(G):= \sum_{z \in \bb Z^d} \phi_\ell(z;b') \omega_{x-z+A} G_{x-z}.
\end{equation}
We have that \eqref{codegua} can be rewritten as
\[
V(G)- V^\ell(G) = \sum_{\subindd} h_{x-b}^{\ell,b'}(G) (\omega_{x+b'}-\omega_x).
\]
Our definitions have been carefully chosen in such a way that $\nabla_{x,x+b'} h_{x-b}^{\ell,b'}(G) =0$. For $x \in \tdn$ and $b' \in \mc B$, let 
\[
\mc D_{x,x+b'}\big(\sqrtt{f}; \mu\big) := \int \Big( \nabla_{x,x+b'} \sqrtt{f} \Big)^2 d\mu
\]
for any density $f$. Notice that
\[
\mc D\big( \sqrtt{f}; \mu \big) = \sum_{\subindd} \mc D_{x,x+b'} \big( \sqrtt{f}; \mu \big).
\]
Using Lemma \ref{lG3} for $x, y =x+b'$, $h = h_{x-b}^{\ell,b'}(G)$ and $\frac{\delta}{2}$, we see that
\begin{equation}
\label{curico}
\begin{split}
\int \big( V(G) -V^\ell(G) \big) f d \mu 
		&\leq \frac{\delta n^2}{2} \mc D\big( \sqrtt{f} ; \mu \big) 
		+ \frac{8}{\delta \eps_0^2 n^2} \int \sum_{\subindd} h_{x}^{\ell,b'}(G)^2 f d\mu\\
		&\quad \quad - \sum_{\subindd} \int (u_{x+b'}-u_x) h_{x-b}^{\ell,b'}(G) \omega_x \omega_{x+b'} f d \mu.
\end{split}
\end{equation}
Let us introduce the definitions
\[
W^\ell(G) = W^\ell_b(G;A) := \sum_{\subindd} h_x^{\ell,b'}(G)^2,
\]
\[
Z^\ell(G) = Z^\ell_b(G;A) := \sum_{\subindd} n (u_{x+b'}-u_x) h_{x-b}^{\ell,b'}(G) \omega_x \omega_{x+b'}.
\]
With these definitions, we see that \eqref{curico} ca be rewritten as
\begin{equation}
\label{cumpeo}
\int V(G)f d\mu 
		\leq \frac{\delta n^2}{2} \mc D\big( \sqrtt{f}; \mu\big) + \int \Big\{ V^\ell(G) + \frac{8}{\delta \eps_0^2 n^2} W^\ell(G) + \frac{1}{n} Z^\ell(G) \Big\} f d \mu.
\end{equation}
Therefore, if the integral on the right-hand side of this inequality could be estimated by $C(H(f; \mu) + g_d(n) n^{d-2})$, the lemma would be proved. As one can guess from the factor $\frac{\delta}{2}$ in front of the quadratic form, this is not yet possible. However, this is the case for the two terms involving $V^\ell(G)$ and $W^\ell(G)$, as now we will see.

Recall that by \eqref{curacavi},
\[
V^\ell(G) = \sum_{x \in \tdn} \overleftarrow{\omega}_{x+A}^\ell \overrightarrow{\omega}_{x+b}^\ell,
\]
where
\[
\overleftarrow{\omega}_{x+A}^\ell := \sum_{y \in \bb Z^d} \omega_{x-y+A} G_{x-y} p_\ell(y),
\]
\[
\overrightarrow{\omega}_x^\ell := \sum_{z \in \bb Z^d} \omega_{x+z} p_\ell(z).
\]
The fact that there is a product of two averages appearing in this expression for $V^\ell(G)$ is crucial in the proof of Lemma \ref{l1}. In order to simplify the computations, we need to introduce some definitions. We say that a set $B \subseteq \bb T^d_n$ is $\ell$-sparse if $\|y-x\|_\infty \geq \ell$ for any $x \neq y \in B$. We say that a family of random variables $\{\xi_x; x \in \bb T^d_n\}$ is $\ell$-dependent if the random variables $\{\xi_x; x \in B\}$ are independent for any $\ell$-sparse set $B \subseteq \tdn$. Notice that with this convention, independent random variables are $1$-dependent. Notice as well that the variables $\{\omega_{x+A}; x \in \tdn\}$ are $\ell_0$-dependent, where $\ell_0$ is the size of the smallest cube containing $A$. In a similar way, $\{\overleftarrow{\omega}_{x+A}^\ell; x \in \tdn\}$ is $(\ell+\ell_0)$-dependent and $\{\overrightarrow{\omega}_x^\ell; x \in \tdn\}$ is $\ell$-dependent.

We say that a random variable $\xi$ is {\em subgaussian} of order $\sigma^2$ if
\[
\log E\big[e^{\theta \xi} \big] \leq \tfrac{1}{2} \sigma^2 \theta^2 \text{ for any } \theta \in \bb R.
\]
By Lemma \ref{lH2.4}, the random variables $\omega_{x+A}$ are subgaussian of order $C(A,\eps_0) = (2/\eps_0)^{2 \# A}$. Therefore, by Lemma \ref{lH2.5} $\overleftarrow{\omega}^\ell_{x+A}$ is subgaussian of order
\[
C(A,\eps_0) \|G\|_\infty^2 \ell^{-d},
\]
where now $C(A,\eps_0) = (2/\eps_0)^{2\# A} (d+1) \ell_0^d$. Although it would be possible to keep track of the dependence in $A$ and $\eps_0$ of the constants $C(A,\eps_0)$, from now on we will not do it. For notational convenience, the value of $C(A,\eps_0)$ may change from line to line. In a similar way, $\overrightarrow{\omega}_{x+b}^\ell$ is subgaussian of order $(2/\eps_0)^2 \ell^{-d}$. Notice that the variables $\{ \overleftarrow{\omega}_{x+A}^\ell \overrightarrow{\omega}_{x+b}^\ell; x \in \tdn\}$ are $(2\ell+\ell_1-1)$-dependent, where $\ell_1$ is the smallest $\ell$ such that $-A \subseteq \Lambda_\ell$. Notice that $\ell_0 \leq \ell_1$, with identity if and only if $0 \in A$. Since $A$ is fixed and $\ell$ is going to grow with $n$, we will assume that $\ell \geq \ell_1$, on which case $2\ell+\ell_1-1 \leq 3\ell$. By Lemma \ref{lH1.4}, for any $\gamma >0$
\[
\int V^\ell(G) f d \mu \leq \frac{d+1}{\gamma} \Big( H(f;\mu) + \frac{1}{(3\ell)^d} \sum_{x \in \tdn} \log \int e^{\gamma (3\ell)^d  \overleftarrow{\omega}_{x+A}^\ell \overrightarrow{\omega}_{x+b}^\ell}d\mu\Big).
\]
By Lemma \ref{lH2.2}, the integral on the right-hand side of this inequality is bounded by $\log 3$ for $
\gamma^{-1} = C(A,\eps_0) \|G\|_\infty$, which gives the bound
\begin{equation}
\label{molina}
\int V^\ell(G) f d \mu \leq C(A,\eps_0) \|G\|_\infty \Big(H(f;\mu) + \frac{n^d}{\ell^d}\Big).
\end{equation}
Notice that the faster $\ell$ grows with $n$, the better this bound is. As we will see, the opposite happens for $W^\ell(G)$. The interplay between these two terms will determine the optimal choice for $\ell$. The integral $\int W^\ell(G) f d \mu$ is estimated in a similar way. $W^\ell(G)$ is the sum of $d$ terms of the form
\begin{equation}
\label{chimbarongo}
W^{\ell,b'}(G) :=\sum_{x \in \tdn} h_x^{\ell,b'}(G)^2, \text{ with } b' \in \mc B.
\end{equation}
We will estimate each of these terms separately. The family $\{h_x^{\ell,b'}(G)^2; x \in \bb T^d_n\}$ is $(2\ell+\ell_0-1)$-dependent. Therefore, by Lemma \ref{lH1.4},
\[
\int W^{\ell,b'}(G) f d \mu \leq \frac{d+1}{\gamma} \Big(H(f; \mu) + \frac{1}{(3\ell)^d} \sum_{x \in \tdn} \log \int e^{\gamma (3\ell)^d h_x^{\ell,b'}(G)^2} d \mu \Big).
\]
Looking at equation \eqref{teno}, we see that by Lemma \ref{lH2.5} $h_x^{\ell,b'}(G)$ is subgaussian of order $C(A,\eps_0) \|G\|_\infty^2 g_d(\ell)$. It is exactly at this point that the function $g_d$ appears in the estimate of Lemma \ref{l1}. By Proposition \ref{lH2.1}, the integral above is bounded by $\log 3$ for $\gamma^{-1} = C(A,\eps_0) \ell^d g_d(\ell) \|G\|_\infty^2$, from where we obtain the bound
\begin{equation}
\label{cauquenes}
\frac{8}{\delta \eps_0^2 n^2} \int W^{\ell}(G) f d \mu \leq \frac{C(A,\eps_0) \|G\|_\infty^2 \ell^d g_d(\ell)}{\delta n^2}\Big(H(f; \mu) + \frac{n^d}{\ell^d}\Big).
\end{equation}
Due to the leading term $\ell^d g_d(\ell)$, this estimate gets worse as $\ell$ grows. 
Notice the similarity of this estimate with \eqref{molina}. 
Apart from the dependence on $\|G\|_\infty$,  both estimates coincide if we choose $\ell$ in such a way that the ratio $\frac{\ell^d g_d(\ell)}{n^2}$ is constant in $n$. This leads to the choice
\begin{equation}
\label{arauco}
\ell =\ell(n):=
\left\{
\begin{array}{c@{\;;\;}l}
\frac{1}{8} n & d=1\\[5pt]
\dfrac{n^2}{\sqrt{\log n}} & d =2 \\[10pt]
n^{2/d} & d \geq 3.
\end{array}
\right.
\end{equation}

The factor $\frac{1}{8}$ for $d=1$ is there to guarantee that the supports of $\overleftarrow{\omega}_{x+A}^\ell$ and $\overrightarrow{\omega}_{x+b}^\ell$ do not overlap. 

The term $\frac{1}{n} \int Z^\ell(G) f d \mu$ can be estimated in a similar way, but the estimate than one obtains is
\begin{equation}
\label{mala}
\frac{1}{n} \int Z^\ell(G) f d \mu \leq \frac{C}{\gamma} \Big( H(f; \mu) + \gamma^2 n^{d-2} \ell^d g_d(\ell)\Big). 
\end{equation}
For $\ell$ given by \eqref{arauco}, this estimate becomes
\[
\frac{1}{n} \int Z^\ell(G) f d \mu \leq \frac{C}{\gamma} \Big( H(f; \mu) + \gamma^2 n^{d}\Big),
\]
which is not good. If one takes $\ell$ of smaller order than the choice given by \eqref{arauco}, the estimate involving $W^\ell(G)$ is of smaller order than the estimate involving $V^\ell(G)$. Balancing \eqref{mala} with \eqref{molina}, we see that the optimal choice of $\ell$ would be
\[
\ell =
\left\{
\begin{array}{c@{\;;\;}l}
n^{2/3} & d=1\\
\dfrac{n^{1/2}}{(\log n)^{1/4}} & d=2\\
n^{1/d} & d \geq 3,
\end{array}
\right.
\]
and that would lead to prove a bound of the form
\[
\int V(G) f d \mu \leq \delta n^2 \mc D\big( \sqrtt{f}; \mu \big) + C H(f;\mu) + C 
\left\{
\begin{array}{c@{\;;\;}l}
n^{1/3} & d=1\\
n \sqrt{\log n} & d=2\\
n^{d-1} & d \geq 3
\end{array}
\right.
\]
for some constant $C = C(A,\eps_0,\kappa,\|G\|_\infty)$. This bound would be enough to prove a form of Theorem \ref{t2} that would imply a quantitative hydrodynamic limit as the one stated in Corollary \ref{c1}, but it would be enough to prove Theorem \ref{t3} only in dimension $d <2$. 

In order to improve the bound on $\frac{1}{n}\int Z^\ell(G) f d \mu$, notice that $Z^\ell(G)$ is a renormalized version of $V(G)$: the local function $\omega_{x+A}$ has been replaced by the function $h_{x-b}^{\ell,b'}(G)\omega_x$, which includes a spatial average in its definition. Thanks to the term $\omega_{x+b'}$, $Z^\ell(G)$ and $V(G)$ have the same structure, and the ideas used to bound $V(G)$ can be iterated. 

The definitions of $V^\ell(G)$, $W^\ell(G)$ and $Z^\ell(G)$ have been chosen in such a way that only one iteration will be enough to prove Lemma \ref{l1}; if we were used the renormalization schemes of \cite{GonJar} or \cite{GonJarSim}, multiple iterations would have been needed.

Now the idea is to define objects analogous to $V^\ell(G)$, $W^\ell(G)$ and $Z^\ell(G)$, but using $Z^\ell(G)$ instead of $V(G)$ as basic object.

For $b' \in \mc B$, let us define
\[
\wt{V}^{\ell,b'}(G) : = \sum_{x \in \tdn} n (u_{x+b'}-u_x) h_{x-b}^{\ell,b'}(G) \omega_x \omega_{x+b'}^\ell
\]
and let us define
\[
\wt{V}^\ell(G) := \sum_{b' \in \mc B} \wt{V}^{\ell,b'}(G).
\]
For $b',b'' \in \mc B$, define
\begin{equation}
\label{cunco}
h_x^{\ell,b',b''}(G) : = \sum_{z \in \bb Z^d} \phi_\ell(z;b'') n (u_{x-z+b'}-u_{x-z}) h_{x-z-b}^{\ell,b'}(G) \omega_{x-z}.
\end{equation}
Using \eqref{rancagua} we have the relation
\begin{equation}
\label{pucon}
Z^\ell(G) - \wt{V}^\ell(G) = \sum_{\substack{x \in \tdn \\b',b'' \in \mc B}} h_{x-b'}^{\ell,b',b''}(G) (\omega_{x+b''}-\omega_x).
\end{equation}
Let us define now
\[
\wt{W}^{\ell,b',b''} (G) := \sum_{x \in \tdn} h_x^{\ell,b',b''}(G)^2,
\]
\[
\wt{W}^\ell(G) := \sum_{b',b'' \in \mc B} W^{\ell,b',b''}(G).
\]
And finally define
\[
\wt{Z}^{\ell,b',b''}(G) := \sum_{x \in \tdn} h_{x-b'}^{\ell,b',b''}(G) n(u_{x+b''}-u_x) \omega_x \omega_{x+b''},
\]
\[
\wt{Z}^\ell(G) := \sum_{b',b'' \in \mc B} \wt{Z}^{\ell,b',b''}(G).
\]
These identities define the terms that appear as right-hand side when we estimate \eqref{pucon} using lemma \ref{lG3}: we have that
\begin{equation}
\label{linares}
\frac{1}{n} \int \big( Z^\ell(G) - \wt{V}^\ell(G) \big) f d \mu 
		\leq \frac{\delta n^2}{2} \mc D\big( \sqrtt{f}; \mu \big)
			+\frac{8}{\delta \eps_0^2 n^4} \int \wt{W}^\ell(G) f d \mu
				+\frac{1}{n^2} \wt{Z}^\ell(G) f d\mu.
\end{equation}
At this point, the quadratic form $\mc D\big(\sqrtt{f}; \mu\big)$ finishes playing its part. From now on, only entropy estimates will be needed to finish the proof of Lemma \ref{l1}. 

In principle, in this second stage we could have introduced a second mesoscopic scale $\wt{\ell} \gg \ell$, exchanging $\omega_{x+b'}$ by $\omega_{x+b'}^{\wt \ell}$. Fortunately, this is not necessary; we will see that for the choice of $\ell$ described in \eqref{arauco}, the factor $\frac{1}{n}$ in front of $Z^\ell(G)$ balances precisely the wider support of the function $h_x^{\ell,b'}(G)$. 

Let us estimate $\int \wt{V}^{\ell,b'}(G) f d \mu$. Since $h_x^{\ell,b'}(G)$ is already an average over a box of size $\ell$, it will not be profitable to pass one of the convoluted probabilities in the expression for $\omega_{x+b'}^\ell$ to $h_x^{\ell,b'}(G)$.

The variables $\{h_{x-b}^{\ell,b'}(G)\omega_x \omega_{x+b'}^\ell; x \in \tdn\}$ are $(3\ell+\ell_1)$-dependent. Therefore, by Lemma \ref{lH1.4}, 
\begin{multline}
\label{pangue}
\frac{1}{n} \int \wt{V}^{\ell,b'}(G) f d \mu 
	\leq \frac{d+1}{\gamma n} \Big( H(f; \mu) +\\
		+ \frac{1}{(4\ell)^d}\sum_{x \in \tdn} \log \int e^{\gamma (4\ell)^d n(u_{x+b'}-u_x) h_{x-b}^{\ell,b'}(G)\omega_x \omega_{x+b'}^\ell } d \mu \Big).
\end{multline}
We have already seen that $h_{x-b}^{\ell,b'}(G)$ is subgaussian of order $C(A,\eps_0) \|G\|_\infty^2 g_d(\ell)$. Since $|\omega_x| \leq \eps_0^{-1}$, $h_{x-b}^{\ell,b'}(G) \omega_x$ is also subgaussian of order $C(A,\eps_0) \|G\|_\infty^2 g_d(\ell)$. Recall that $q_\ell(z) \leq \ell^{-d}$ for any $z$. By Lemma \ref{lH2.5}, $\omega_x^\ell$ is subgaussian of order $C(\eps_0) \ell^{-d}$. By Lemma \ref{lH2.2}, the integral on the right-hand side of \eqref{pangue} is bounded by $\log 3$ for $\gamma^{-1} = C(A,\eps_0) \kappa \|G\|_\infty\sqrt{\ell^d g_d(\ell)}$. Putting this estimate into \eqref{pangue}, we obtain the bound
\begin{equation}
\label{curepto}
\begin{split}
\frac{1}{n} \int \wt{V}^{\ell,b'}(G) f d \mu 
	&\leq \frac{C(A,\eps_0) \kappa \|G\|_\infty \sqrt{\ell^d g_d(\ell)}}{n} \Big( H(f;\mu) + \frac{n^d}{\ell^d}\Big)\\
	&\leq C(A,\eps_0) \kappa \|G\|_\infty \big( H(f;\mu) + n^{d-2} g_d(n)\big),
\end{split}
\end{equation}
if $\ell$ is chosen according to \eqref{arauco}. 

Now let us estimate $\int \wt{W}^{\ell,b'b''}(G) f d\mu$. If we want to proceed as we did with $W^\ell(G)$, we need to estimate the exponential moments of $h_x^{\ell,b',b''}(G)^2$. Looking back at \eqref{cunco}, we see that $h_x^{\ell,b',b''}(G)$ is bilinear in the variables $\omega_{x-z+A}$, $\omega_{x-z}$. Therefore, our subgaussian bounds will not be effective for its square. Putting \eqref{teno} into \eqref{cunco}, we see that
\begin{multline}
\label{chillan}
h_x^{\ell,b',b''}(G) = \sum_{z,z' \in \bb Z^d} \phi_\ell(z;b') \phi_\ell(z';b'') n(u_{x-z'+b'}-u_{x-z'}) \times \\
	\times\omega_{x-z'} \omega_{x-z-z'-b+A} G_{x-z-z'-b}.
\end{multline}
In particular, by Lemma \ref{flow}, $\|h_x^{\ell,b',b''}(G)\|_\infty \leq C(A,\eps_0) \kappa \|G\|_\infty \ell^2$.
We conclude that for any $\gamma>0$, 
\[
\begin{split}
\log \int e^{\gamma h_x^{\ell,b',b''}(G)^2} d \mu
		&\leq \log \int e^{\gamma C(A,\eps_0) \kappa \|G\|_\infty \ell^2 |h_x^{\ell,b',b''}(G)|} d \mu\\
		&\leq \max_{\pm} \Big\{\log   \int e^{\pm \gamma C(A,\eps_0) \kappa \|G\|_\infty \ell^2 h_x^{\ell,b',b''}(G)} d \mu \Big\} + \log 2,
\end{split}
\]
where we used the inequalities $e^{|x|} \leq e^x+ e^{-x}$ and $\log(a+b) \leq \max\{\log a, \log b\} + \log 2$. 

To estimate the exponential moments of $h_x^{\ell,b',b''}(G)$ we will use the Hanson-Wright inequality, as stated in Lemma \ref{lH2.6}. Let $ \vec{\ell_1} = (\ell_1,\dots,\ell_1)$. Recall that $-A \subseteq \Lambda_{\ell_1}$. The variables $\xi_x := \omega_{x+\vec{\ell_1} +A}$, $\wt{\xi}_x$ and the sum \eqref{chillan} satisfy the hypothesis of Lemma \ref{lH2.6} and therefore 
\begin{equation}
\label{parral}
\log   \int e^{\pm \gamma  h_x^{\ell,b',b''}(G)} d \mu \leq \log 3
\end{equation}
for 
$\gamma^{-1} = C(A,\eps_0) \kappa \|G\|_\infty g_d(\ell)$. The variables $h_x^{\ell,b',b''}(G)$ are $(3\ell+\ell_1)$-dependent. Therefore, by Lemma \ref{lH1.4} we conclude that
\[
\begin{split}
		\int \wt{W}^{\ell,b',b''}(G) f d \mu
		&\leq \frac{d+1}{\gamma} \Big( H(f;\mu) + \frac{1}{(4\ell)^d} \sum_{x \in \tdn} \log \int e^{\gamma (4\ell)^d h_x^{\ell,b',b''}(G)^2} d \mu\Big)\\
		&\leq \frac{d+1}{\gamma} \bigg( H(f;\mu) + \frac{1}{(4\ell)^d} \times\\
		&\hspace{14pt} \times\sum_{x \in \tdn} \Big(\max_{\pm} \Big\{\log \int e^{\pm \gamma C(A,\eps_0)\kappa\|G\|_\infty \ell^{d+2} h_x^{\ell,b',b''}(G)} d \mu\Big\} + \log 2\Big)\bigg).\\
\end{split}
\]
Taking $\gamma^{-1} = C(A,\eps_0) \kappa^2 \|G\|_\infty^2 g_d(\ell) \ell^{d+2}$, we conclude that
\[
\frac{1}{\delta \eps_0^2 n^4} 
		\int \wt{W}^{\ell,b',b''}(G) f d \mu
			\leq \frac{C(A,\eps_0) \kappa^2 \|G\|_\infty^2 g_d(\ell) \ell^{d+2}}{\delta n^4}\Big( H(f;\mu) + \frac{n^d}{\ell^d}\Big).
\]
Replacing the value of $\ell$ chosen in \eqref{arauco}, we conclude that
\begin{equation}
\label{constitucion}
\frac{1}{\delta \eps_0^2 n^4} 
		\int \wt{W}^{\ell,b',b''}(G) f d \mu
			\leq \frac{C(A,\eps_0) \kappa^2 \|G\|_\infty^2 \ell^2}{\delta n^2}\Big( H(f;\mu) + n^{d-2}g_d(n)\Big),
\end{equation}
which is of smaller order than \eqref{curepto}, except for the quadratic dependence on $\kappa$ and $\|G\|_\infty$. Now we are only left to estimate $\int \widetilde{Z}^\ell(G) f d \mu$. By \eqref{parral}, 
\[
\int \exp\big\{ \gamma h^{\ell,b',b''}_{x-b'}(G) n(u_{x+b''}-u_x)\omega_x\omega_{x+b''} \big\} d \mu \leq \log 3
\]
for $\gamma^{-1} = C(A,\eps_0)\kappa^2 \|G\|_\infty g_d(\ell)$. Therefore, by Lemma \ref{lH1.4} we have that
\begin{equation}
\label{lircay}
\frac{1}{n^2} \int \widetilde{Z}^{\ell,b',b''}(G) f d\mu \leq \frac{C(A,\eps_0)\kappa^2 \|G\|_\infty \ell^d g_d(\ell)}{n^2} \Big(H(f;\mu) + \frac{n^d}{\ell^d}\Big).
\end{equation}

Putting estimates \eqref{molina}, \eqref{cauquenes}, \eqref{curepto}, \eqref{constitucion} and \eqref{lircay}, Lemma \ref{l1} is proved.

If instead of using the entropy estimate we just collect estimates \eqref{cumpeo} and \eqref{linares}, we obtain the following version of Lemma \ref{l1}:

\begin{lemma}[Main lemma, v2]
\label{main_v2}
There exists constant $C=C(A,\eps_0)$ such that for any $G: \bb T^d_n \to \bb R$, any density $f$ with respect to $\mu$ and any $\delta>0$,
\[
\begin{split}
\int V(G) f d \mu \leq
		\delta n^2\mc D\big( \sqrtt{f}; \mu\big) 
		&+ \int \Big( V^\ell(G) +\frac{C(A,\eps_0)}{\delta n^2} W^\ell(G) +\frac{1}{n} \widetilde{V}^\ell(G)\\
		&\quad \quad \quad + \frac{C(A,\eps_0)}{\delta n^4} \widetilde{W}^\ell(G) 
		+\frac{1}{n^2} \widetilde{Z}^\ell(G)\Big) f d\mu.
\end{split}
\]
\end{lemma}

This version of Lemma \ref{l1} will be needed in the proof of our non-equilibrium version of the Boltzmann-Gibbs principle.

\begin{remark}
If needed, the dependence of $C(A,\eps_0)$ can be tracked back; since this dependence is not very intuitive and we do not need it here, we opted to not make it explicit.
\end{remark}

\begin{remark}
The fact that $A \in \bb O^-$ is not very important, but it considerably simplifies the notation. The interested reader will not have trouble working out the corresponding modifications.
\end{remark}

\section{The entropy inequality}
\label{s4}
In this section we prove Theorem \ref{t2} and Corollary \ref{c1}. As we will see below, most of the work has been accomplished in the derivation of Lemma \ref{l1}. 

\subsection{Proof of Theorem \ref{t2}}
\label{s4.1}
Let us recall that the process $\eta^n(\cdot)$ is generated by the operator $L_n$ given by
\[
\begin{split}
L_n h(\eta) = n^2 \sum_{\subind} \Big\{ 
		&\max\big\{\tfrac{1}{2}, 1+ \tfrac{1}{n} F_b^n(x)\big\} \eta_x (1-\eta_{x+b}) \\
		&+ \max\big\{\tfrac{1}{2}, 1- \tfrac{1}{n} F_b^n(x)\big\}\eta_{x+b}(1-\eta_x) \Big\} \nabla_{x,x+b} h(\eta).
\end{split}
\]
The {\em carr\'e du champ} associated to $L_n$ is the quadratic operator $\Gamma_n$ given by
\[
\begin{split}
\Gamma_n h(\eta) = n^2 \sum_{\subind} \Big\{ 
		&\max\big\{\tfrac{1}{2}, 1+ \tfrac{1}{n} F_b^n(x)\big\} \eta_x (-\eta_{x+b}) \\ \vspace{-10pt}
		&+ \max\big\{\tfrac{1}{2}, 1- \tfrac{1}{n} F_b^n(x)\big\}\eta_{x+b}(1-\eta_x) \Big\} \big(\nabla_{x,x+b} h(\eta)\big)^2.
\end{split}
\]
Recall definition \eqref{valparaiso}. Notice that this definition makes sense for any measure $\mu$ in $\Omega_n$. We have that for any measure $\mu$ and any density $f$ with respect to $\mu$,
\begin{equation}
\label{pitrufquen}
\int \Gamma_n \sqrtt{f} d \mu \geq \frac{n^2}{2} \mc D\big( \sqrtt{f}; \mu \big).
\end{equation}
By Yau's inequality stated in Lemma \ref{lA1}, 
\begin{equation}
\label{teodoro}
H'_n(t) \leq - \int\Gamma_n \sqrtt{f_t^n} d \mu_t^n + \int \big( L_{n,t}^\ast \mathds{1} - \tfrac{d}{dt} \log \psi_t^n \big) f_t^n d \mu_t^n,
\end{equation}
where $L_{n,t}^\ast$ is the adjoint of $L_n$ with respect to $\mu_t^n$ and where $\psi_t^n$ is the Radon-Nikodym derivative of $\mu_t^n$ with respect to $\nu_{1/2}^n$. Thanks to \eqref{pitrufquen}, \eqref{teodoro} implies that
\[
H'_n(t) \leq -\frac{n^2}{2} \mc D\big( \sqrtt{f_t^n} ; \mu \big) + \int J_t^n f_t^n d \mu_t^n,
\] 
where $J_t^n =  L_{n,t}^\ast \mathds{1} - \tfrac{d}{dt} \log \psi_t^n$. By \eqref{funG}, $J_t^n$ is of the form
\[
J_t^n = \sum_{\subind} G_{x,b}^n(t) \omega_{x} \omega_{x+b} + \frac{1}{n^2} \sum_{x \in \tdn} \omega_x R_x^n(t),
\]
where $G_{x,b}^n(t)$ satisfies
\[
|G_{x,b}^n(t)| \leq \|\nabla u_t \|_\infty \big( \| \nabla u_t \|_\infty + \|F\|_\infty \big)
\]
and $R_x^n(t)$ satisfies $|R_x^n(t)| \leq \|u_t\|_{\mc C^4}$. The term involving $R_x^n(t)$ is very easy to estimate: by the entropy estimate \eqref{temuco}, 
\[
\begin{split}
\int \frac{1}{n^2} \sum_{x \in \tdn} \omega_x R_x^n(t) f_t^n d \mu_t^n 
		&\leq \gamma^{-1} \Big( H_n(t) + \log \int \exp\Big\{ \frac{\gamma}{n^2} \sum_{x \in \tdn} \omega_x R_x^n(t)\Big\} d \mu_t^n \Big)\\
		&\leq \gamma^{-1} \big( H_n(t) + C(\eps_0) \|R^n(t)\|_\infty^2 \gamma^2 n^{d-4}\big).
\end{split}
\]
Then it is enough to take $\gamma=1$ and to observe that $n^{d-4} \ll g_d(n) n^{d-2}$.
By Lemma \ref{lC3}, for any $T >0$ there exists $\eps_1 = \eps_1(T,\eps_0)>0$ such that $\eps_1 \leq u_x^n(t) \leq 1-\eps_1$ for any $n \in \bb N$, any $x \in \tdn$ and any $t \in [0,T]$. Therefore, Lemma \ref{l1} can be used with $\mu = \mu_t^n$, $f = f_t^n$, $A =\{0\}$, $G_x = G_{x,b}^n(t)$ and $\delta = \frac{1}{2d}$ to conclude that
\[
\int J_t^n f_t^n d \mu_t^n \leq \frac{n^2}{2} \mc D \big( \sqrtt{f_t^n} ; \mu_t^n \big) + C(u_0,F) \big(H_n(t) + n^{d-2} g_d(n)\big),
\]
which proves \eqref{entropy1}. The second part of Theorem \ref{t2} follows from \eqref{entropy1} and Gronwald's inequality.

\subsection{Proof of Corollary \ref{c1}} 
\label{s4.2}
Corollary \ref{c1} is a particular case of the following result, which is an example of known in the literature as {\em conservation of local equilibrium}. 

\begin{corollary}
\label{c2}
For any $p \in [1,2)$, any $t \in [0,T]$ and any $A \subseteq \bb Z^d$ there exists finite constant $C = C(\eps_0,T,p,A)$ such that for any $n \in \bb N$ and any $H: \tdn \to \bb R$,
\[
\bb E_n \Big[\Big| \frac{1}{n^d} \sum_{x \in \tdn} \omega_{x+A} H_x \Big|^p\Big] \leq \frac{C\big(1+H_n(t)\big)^{p/2} \|H\|_\infty^p}{n^{pd/2}}.
\]
In particular, if $H_n(0) \leq C n^{d-2} g_d(n)$, then
\[
\bb E_n \Big[\Big| \frac{1}{n^d} \sum_{x \in \tdn} \omega_{x+A} H_x \Big|^p\Big] \leq \frac{C g_d(n)^{p/2} \|H\|_\infty^p}{n^p}.
\]
\end{corollary}
\begin{proof}
By \eqref{temuco2},
\[
\bb P_n\Big( \frac{1}{n^d} \sum_{x \in \tdn} \omega_{x+A} H_x > \lambda\Big) \leq \frac{H_n(t) + \log 2}{\log \mu_t^n\Big(\frac{1}{n^d} \sum_{x \in \tdn} \omega_{x+A} H_x > \lambda \Big)^{-1}}
\]
By Lemmas \ref{lH2.4} and \ref{lH2.5}, 
\[
\mu_t^n\Big(\Big|\frac{1}{n^d} \sum_{x \in \tdn} \omega_{x+A} H_x \Big| > \lambda \Big)
		\leq 2\exp\Big\{ -\frac{\lambda^2 n^d}{C(A,\eps_0) \|H\|_\infty^2}\Big\},
\]
and therefore
\[
\bb P_n\Big( \frac{1}{n^d} \sum_{x \in \tdn} \omega_{x+A} H_x > \lambda\Big) 
		\leq \frac{H_n(t) + \log 2}{n^d} \cdot \frac{C\|H\|_\infty^2}{\lambda^2}.
\]
By Lemma \ref{lH1.5}, we conclude that
\[
\bb E_n \Big[\Big| \frac{1}{n^d} \sum_{x \in \tdn} \omega_{x+A} H_x \Big|^p\Big] \leq \frac{C\big(1+H_n(t)\big)^{p/2} \|H\|_\infty^p}{n^{pd/2}},
\]
as we wanted to show. The second estimate follows from the bound
\[
H_n(t) \leq C(\eps_0,F,T)\big( H_n(0) + g_d(n) n^{d-2}\big),
\]
which was obtained in Theorem \ref{t2}.
\end{proof}

An immediate consequence of Holder's inequality and Corollary \ref{c2} is the following estimate, which we state for further reference: under the conditions of Corollary \ref{c2}, for any $0 \leq s <t \leq T$ and any $H: [s,t] \times \tdn \to \bb R$,
\begin{equation}
\label{c2'}
\bb E_n \Big[ \Big| \int_s^t \frac{1}{n^{d/2}} \sum_{x \in \tdn} \omega_{x+A} H_{s',x} ds \Big|^p\Big] \leq 
	C \|H\|_\infty^p |t-s|^p\sup_{s \leq s'\leq t} \big(1+H_n(s')\big)^{p/2}.
\end{equation}

\section{The Boltzmann-Gibbs Principle}
\label{s5}

In this section we prove what is known in the literature as the {\em Boltzmann-Gibbs principle}, which roughly states that space-time averages of local observables of conservative dynamics can be approximated by functions of the conserved quantities. A general proof of this principle only exists in {\em equilibrium}, that is, when the stochastic system in consideration starts from one of its invariant measures. The main novelty in this section is the derivation of a general strategy to prove the Boltzmann-Gibbs principle {\em out of equilibrium}, in dimensions $d < 4$. Existing proofs of this principle out of equilibrium are either based on the concept of duality (see \cite{D-MPre} or \cite{AyaCarRed} for a more recent reference), or require $d=1$, reversibility and the availability of a sharp estimate of the log-Sobolev constant of the system with respect to its invariant measure \cite{ChaYau}, \cite{Lan}.

We will state the Boltzmann-Gibbs principle in a less conventional way, which is more convenient for the purposes of this article. After the proof is finished, we will explain how to relate our formulation with the formulation commonly found in the literature.

\begin{theorem}[Boltzmann-Gibbs principle]
\label{BG}
Let $A \subseteq \{z \in \bb Z^d; z_i \leq 0, i=1,\dots,d\}$ be fixed. Let $b \in \mc B$ and let $H^n: [0,T] \times \tdn \to \bb R$ be uniformly bounded in $n$. Assume that
\begin{equation}
\label{lota}
\lim_{n \to \infty} \frac{H_n(0)}{n^{d/2}} =0.
\end{equation}
For $d <4$, any $s<t \in [0,T]$, any $b \in \mc B$ and any $\lambda >0$,
\[
\lim_{n \to \infty} \bb P_n \Big( \Big| \int_s^t \frac{1}{n^{d/2}} \sum_{x \in \tdn} \omega_{x+A} \omega_{x+b} H_x ds' \Big| > \lambda \Big) =0.
\]
Moreover, for any $K>0$ this convergence is uniform in the set $\{\sup_{n \in \bb N} \|H^n\|_\infty \leq K\}$. 
\end{theorem}

Recall that we are omitting the dependence on $s'$ and $n$ of $H_x=H_x^n(s)$ and $\omega_x$. 

\begin{remark}
We wrote $d<4$ instead of $d \leq 3$ to emphasize that the critical dimension for our method is $d=4$. This condition comes from the condition $2 > d/2$. Here $d/2$ is the size of the fluctuations, and $2$ is the spectral exponent of the process. It is not difficult to build models (for example using long-range dynamics) on which the spectral exponent is $\alpha \in (0,2)$ and for which one can verify that the critical dimension is $2 \alpha$.  
\end{remark}

\begin{proof}
Considering $H$ and $-H$ and using the union bound, it is enough to prove that
\[
\lim_{n \to \infty} \bb P_n \Big( \int_s^t \frac{1}{n^{d/2}} \sum_{x \in \tdn} \omega_{x+A} \omega_{x+b} H_x ds' > \lambda \Big) =0.
\]
for each $b \in \mc B$. Notice that 
\[
\sum_{x \in \tdn} \omega_{x+A} \omega_{x+b} H_x
\]
coincides with the function $V(H)$ defined in \eqref{vina}. Therefore, our aim will be to estimate the probability
\[
\bb P_n \Big( \int_s^t V(H) ds' > \lambda n^{d/2}\Big).
\]
For reasons that will become apparent in a few lines, it will be convenient to introduce an auxiliary function $U_\gamma(H)$ which will depend on $H$ and on an additional parameter $\gamma>0$, and to estimate
\[
\bb P_n \Big( \int_s^t \big( V(H) - U_\gamma(H)\big)ds' > \lambda n^{d/2} \Big).
\]
The idea is to use the exponential Tchebyshev inequality and Lemma \ref{lA2} to estimate this probability. However, Lemma \ref{lA2} requires $\eta^n(s)$ to have law $\mu_s^n$. Let $\widetilde{\bb P}_n$ be the law of $\eta^n(\cdot+s)$. By the Markov property,
\[
H\bigg( \frac{d \widetilde{\bb P}_n}{d \bb P_n^{\mu_s^n}} ; \bb P_n^{\mu_s^n} \bigg) = H \big( f_s^n ; \mu_s^n\big) = H_n(s).
\]
We conclude that by \eqref{temuco2},
\[
\bb P_n \Big( \int_s^t \hspace{-4pt}\big( V(H) - U_\gamma(H)\big)ds' \hspace{-2pt}>\hspace{-2pt} \lambda n^{d/2} \Big)
		\leq\frac{H_n(s)+\log 2}{\log \bb P_n^{\mu_s^n} \Big( \int_0^{t-s} \big( V(H) - U_\gamma(H)\big)ds' > \lambda n^{d/2} \Big)^{-1}}.
\]
Therefore, it is enough to estimate 
\[
\log \bb P_n^{\mu_s^n} \Big( \int_0^{t-s} \big( V(H) - U_\gamma(H)\big)ds > \lambda n^{d/2} \Big).
\]
Since the initial law is now $\mu_s^n$, we can use Lemma \eqref{lA2} (with $\mu_{s+t}^n$ in place of $\mu_t$). By the exponential Tchebyshev inequality, the last expression is bounded by
\begin{multline}
\label{calafquen}
-\gamma \lambda n^{d/2} + \int_0^{t-s} \sup_{f} \Big\{ -\int \Gamma_n \sqrtt{f} d \mu_{s+s'}^n +\\
+ \int \gamma\big(V(H) -U_\gamma(H)\big) f d\mu_{s+s'}^n +\frac{1}{2} \int J_{s+s'}^n f d \mu_{s+s'}^n\Big\} ds',
\end{multline}
where the supremum runs over all densities $f$ with respect to $\mu_{s+s'}^n$. We have already seen that the main lemma can be used to bound $J_{s+s'}^n$, and it can also be used to bound $V(H)$. However, since here we are taking the supremum over all densities, we need to use the version of the main lemma stated in Lemma \ref{main_v2}. Using Lemma \ref{main_v2} for both $\int J_{s+s'}^n f d\mu_{s+s'}^n$ and $\int V(H) f d \mu_{s+s'}^n$, we see that \eqref{calafquen} is bounded by
\[
\begin{split}
-\gamma \lambda n^{d/2}
		&+\int_0^{t-s} \sup_f \Big\{ - \int \gamma U_\gamma(H) f d \mu_{s+s'}^n \\
		&\hspace{1.5pt}+\Big( \gamma V^\ell \hspace{-1.5pt} (H) +\frac{C\gamma^2}{n^2} W^\ell(H) +\frac{\gamma}{n} \widetilde{V}^\ell\hspace{-1.5pt}(H) + \frac{C \gamma^2}{n^4} \widetilde{W}^\ell(H) + \frac{\gamma}{n^2} \widetilde{Z}^\ell(H) +\frac{1}{2} V^\ell\hspace{-1.5pt}(G)\\
		&\quad + \frac{C}{n^2} W^\ell(G) +\frac{1}{2n} \widetilde{V}^\ell(G) +\frac{C}{n^4} \widetilde{W}^\ell(G) + \frac{1}{2n^2} \widetilde{Z}^\ell(G) \Big) f d \mu_{s+s'}^n\Big\} ds',
\end{split}
\]
where $C$ is a constant that depend only on $u_0,T$ and $F$. We have introduced the function $U_\gamma(H)$ in order to cancel the other 10 terms on this variational expression. More precisely, if we define
\begin{equation}
\label{petrohue}
\begin{split}
U_\gamma(H) 
		&:= \gamma V^\ell(H) +\frac{C\gamma^2}{n^2} W^\ell(H) +\frac{\gamma}{n} \widetilde{V}^\ell(H) + \frac{C \gamma^2}{n^4} \widetilde{W}^\ell(H) + \frac{\gamma}{n^2} \widetilde{Z}^\ell(H) \\
		&\quad \quad \quad+\frac{1}{2} V^\ell(G) + \frac{C}{n^2} W^\ell(G) +\frac{1}{2n} \widetilde{V}^\ell(G) +\frac{C}{n^4} \widetilde{W}^\ell(G) + \frac{1}{2n^2} \widetilde{Z}^\ell(G),
\end{split}
\end{equation}
we conclude that
\begin{equation}
\label{panguipulli} 
\log \bb P_n^{\mu_s^n} \Big( \int_0^{t-s} \big( V(H) - U_\gamma(H)\big)ds' > \lambda n^{d/2} \Big)
		\leq -\gamma \lambda n^{d/2},
\end{equation}
and therefore
\begin{equation}
\label{choshuenco}
\bb P_n \Big( \int_s^t \big( V(H) - U_\gamma(H)\big)ds' > \lambda n^{d/2} \Big)
		\leq \frac{H_n(0) + \log 2}{\gamma \lambda n^{d/2}} \xrightarrow{n \to \infty} 0
\end{equation}
by \eqref{lota}, since $\gamma$ is fixed.

In order to finish the proof of the theorem, we need to estimate 
\begin{equation}
\label{rinihue}
\bb P_n\Big( \int_s^t U_\gamma(H) ds' > \lambda n^{d/2} \Big).
\end{equation}
For this term, the exponential bound of Lemma \ref{lA2} will not be useful, so we need another argument. The idea is to use the entropy bound as an {\em a priori} bound to deal with $U_\gamma(H)$. Since $U_\gamma(H)$ is defined in terms of averages over boxes of size $\ell$, the entropy inequality will be effective. We have that \eqref{rinihue} is bounded by
\[
\int_s^t \bb E_n\Big[ \frac{1}{n^{d/2}} \big|U_\gamma(H) \big| \Big] ds'.
\]
Looking back at \eqref{petrohue}, we see that $U_\gamma(H)$ is the sum of 10 terms that can be grouped into 5 pairs, whose expectations can be estimated in the same way we obtained \eqref{molina}, \eqref{cauquenes}, \eqref{curepto}, \eqref{constitucion} and \eqref{lircay}. The only difference comes from the absolute value inside the probability. On each of these estimates, if we use the second estimate in Lemma \ref{lH1.4}, we obtain the same bounds, since we always estimated the logarithm by something positive. We obtain the bound
\begin{equation}
\label{futrono}
\begin{split}
\int_s^t \bb E_n\Big[ \frac{1}{n^{d/2}} \big|U_\gamma(H) \big| \Big] ds'
		&\leq \frac{C (t-s)}{\lambda n^{d/2}} \Big( \frac{1}{\gamma} + \|H\|_\infty + \gamma \|H\|_\infty^2\Big)\times\\
		&\hspace{65pt}\times \Big( \sup_{s \leq s' \leq t}\big(1+ H_n(s')\big) + n^{d-2} g_d(n) \Big)\\
		&\leq \frac{C(t-s)}{\lambda}  \Big( \frac{1}{\gamma} +  \gamma \|H\|_\infty^2\Big) \Big(  \frac{H_n(0)}{n^{d/2}} + n^{d/2-2} g_d(n) \Big),
\end{split}
\end{equation}
where $C$ is a constant depending only on $u_0, F$ and $T$. We see that the right-hand side of this estimate goes to 0 as $n \to \infty$ exactly under the condition $d <4$. Putting \eqref{choshuenco} and \eqref{futrono} together, we conclude that
\begin{multline}
\label{lagranja}
\bb P_n \Big( \int_s^t V(H) ds' > \lambda n^{d/2} \Big) 
		\leq \frac{C\big( H_n(0)+ \log 2\big)}{\gamma \lambda n^{d/2}}+\\
		 + \frac{C(t-s)}{\lambda} \Big( \frac{1}{\gamma}+ \gamma \|H\|_\infty^2\Big) \Big( \frac{H_n(0)}{n^{d/2}} + n^{d/2-2}g_d(n)\Big),
\end{multline}
which proves the theorem.
\end{proof}

The classical formulation of the Boltzmann-Gibbs Principle is the following. Let $h_0: \Omega \to \bb R$ be a local function, that is $h_0$ depends on a finite number of variables $\eta_x$. Let $h_x$ be the translation by $x$ of $h_0$. Assume that $H$ is smooth. Then there exist functions $\{a_x^n(t), b_x^n(t); x \in \tdn, t \geq 0\}$ such that
\[
\lim_{n \to \infty} \int_0^T \frac{1}{n^{d/2}} \sum_{x \in \tdn} \big( H_x h_x -a_x -b_x(\eta_x -u_x)\big) ds =0
\]
in probability. Since $h_0$ is local, it is a finite combination of monomials of the form $\omega_{x+B}$. Theorem \ref{BG} requires the cardinality of $B$ to be equal to 2 or higher. It is enough to choose $a_x$ in such a way that it cancels the constant term of $h_x$ and $b_x$ in such a way that it cancels the linear term of $h_x$. This last step requires a summation by parts which takes advantage of the smoothness of $H$.

\section{Proof of Theorem \ref{t3}} 
\label{s6}
In this section we prove Theorem \ref{t3}. The proof will follow the martingale method introduced in \cite{HolStr}, see Chapter 11 of \cite{KipLan} for a review. In Section \ref{s6.1} we will write the process $\{X_t^n; t \geq 0\}$ as the sum of a martingale process $\{M_t^n; t \geq 0\}$ and an integral process. In Section \ref{s6.2} we show that the process $\{M_t^n; t \geq 0\}$ converges to a Gaussian noise,  which corresponds to the noise appearing in \eqref{SHE}. We will show that the convergence takes place with respect to an almost optimal topology, in a sense to be discussed afterwards. Finally, in Section \ref{s6.3} we use Theorem \ref{BG} and the results of Section \ref{s6.2} in order to prove Theorem \ref{t3}.

In order to avoid non-relevant topological considerations, in this section we will fix a finite time window $[0,T]$ and we will consider all processes as defined for $t \in [0,T]$.

\subsection{The associated martingales}
\label{s6.1}
A simple consequence of Dynkin's formula is that for any $g: [0,T] \times \Omega_n \to \bb R$ smooth on the time variable, the process
\[
g(t,\eta^n(t)) -g(0,\eta^n(0)) -\int_0^t (\partial_s+L_n) g(s,\eta^n(s)) ds
\]
is a martingale. The quadratic variation of this martingale is given by
\[
\int_0^t \Gamma_n g(s,\eta^n(s)) ds.
\]
Now let $H:[0,T] \times \bb T^d \to \bb R$ be smooth. Applying these formulas to the function $X_t^n(H_t)$, we see that the process $\{M_t^n(H); t \in [0,T]\}$ given by
\begin{equation}
\label{santiago}
M_t^n(H):= X_t^n(H_t) - X_0^n(H_0) - \int_0^t (\partial_s + L_n) X_s^n(H_s) ds
\end{equation}
is a martingale of quadratic variation
\[
\<M_t^n(H)\> = \int_0^t \Gamma_n X_s^n(H_s) ds.
\]
By duality, these relations define a martingale process $\{M_t^n; t \in [0,T]\}$ with values in $H_{-k}(\bb T^d)$ and \cadlag trajectories for $k$ large enough. Later on we will see that $k>d/2$ is enough. 

Both the integral term in \eqref{santiago} and the quadratic variation $\<M_t^n(H)\>$ can be computed explicitly. For $f: \tdn \to \bb R$, define $\Lambda_{t}^n f : \tdn \to \bb R$ as
\begin{equation}
\label{cerrillos}
\begin{split}
\Lambda_{x,t}^n f 
	&:= \sum_{b \in \mc B} n^2\big( f_{x+b} +f_{x-b} -2 f_x\big)\\
	&\quad \quad	+ \sum_{b \in \mc B} \Big( (1-2u_{x+b}) F_b^n(x) n \big(f_{x+b}-f_x\big) \\
	&\quad \quad \quad \quad		+ (1-2u_{x-b}) F_b^n(x-b) n \big( f_x- f_{x-b}\big)\Big).
\end{split}
\end{equation}
Then,
\begin{equation}
\label{santiago3}
X_t^n(H_t) = X_0^n(H_0^{\vphantom{n}}) + \mc R_t^n(H) + \mc A_t^n(H) + \mc Q_t^n(H) + M_t^n(H), 
\end{equation}
where
\begin{equation}
\label{santiago_rest}
\mc R_t^n(H) = \int_0^t \frac{1}{n^{d/2}}\sum_{x \in \tdn} H_s\big(\tfrac{x}{n}\big) \big(\mc L^n - \tfrac{d}{ds}\big) u_x ds,
\end{equation}
$\mc L^n$ is the discrete operator defined in \eqref{conguillio},
\begin{equation}
\label{santiago4}
\mc A_t^n(H) := \int_0^t X_s^n\big( (\tfrac{d}{ds} + \Lambda_s^n)H_s\big) ds
\end{equation}
and
\begin{equation}
\label{santiago5}
\mc Q_t^n(H) := \int_0^t \frac{1}{n^{d/2}} \sum_{\subind} 2 n \big( H_s\big(\tfrac{x+b}{n}\big) -H_s \big(\tfrac{x}{n}\big)\big) F_b^n(x) (\eta_x -u_x ) (\eta_{x+b} -u_{x+b}) ds.
\end{equation}
\
The quadratic variation of $M_t^n(G)$ is equal to
\begin{equation}
\label{santiago2}
\begin{split}
\<M_t^n(H)\> 
		&= \int_0^t \frac{1}{n^d} \sum_{\subind}\big( r_n(x,x+b) \eta_x (1-\eta_{x+b})\\
			&\quad \quad \quad + r_n(x+b,x) \eta_{x+b} (1-\eta_x)\big) n^2 \big( H_s\big(\tfrac{x+b}{n}\big)-H_s\big(\tfrac{x}{n}\big)\big)^2 ds
\end{split}
\end{equation}

\subsection{Convergence of the martingale process}
\label{s6.2}
In this section we prove the convergence of the martingale process $\{M_t^n; t \in [0,T]\}$. We start proving tightness. We will use Aldous' criterion, see Proposition \ref{pD3.1}. Let $\tau$ be a stopping time and let $\beta>0$. Recall the definition \eqref{ecA1.1} of the Sobolev norm $\|\cdot\|_{-k}$. We have that
\[
\begin{split}
\bb P_n\big( \big\| M_{\tau+\beta}^n-M_\tau\big\|_{-k} \geq \eps \big)
		&\leq \eps^{-2} \bb E_n \big[ \big\| M_{\tau+\beta}^n - M_\tau^n\big\|^2_{-k}\big]\\
		&\leq \eps^{-2} \sum_{m \in \bb Z^d} (1+|m|^2)^{-k} \bb E_n\big[ \big|M_{\tau+\beta}^n(\phi_m)-M_{\tau}^n(\phi_m)\big|^2\big].
\end{split}
\]
Here we use the convention $M_t^n(\phi_m) = M_t^n(\Re \phi_m) + i M_t^n(\Im \phi_m)$. Therefore,
\begin{equation}
\label{renca}
\begin{split}
\bb P_n\big( \big\| M_{\tau+\beta}^n-M_\tau\big\|_{-k} \geq \eps \big)
		&\leq \eps^{-2} \sum_{m \in \bb Z^d}  (1+|m|^2)^{-k}\bb E_n\big[ \<M_{\tau+\beta}^n(\Re \phi_m)\> - \<M_{\tau}^n(\Re \phi_m)\>\\
		&\quad \quad \quad \quad \quad \quad \quad \quad  +  \<M_{\tau+\beta}^n(\Im \phi_m)\> - \<M_{\tau}^n(\Im \phi_m)\>\big].
\end{split}
\end{equation}
Since $\|r_n\|_\infty \leq 1 +\|F\|_\infty = C(F)$ and 
\[
n^2 \big( H_s \big(\tfrac{x+b}{n}\big)- H_s \big(\tfrac{x}{n}\big)\big)^2 \leq \min\big\{\|\nabla H\|_\infty^2, n^2 \|H\|_\infty\big\},
\]
we have that
\[
\tfrac{d}{dt} \<M_t^n(H)\> \leq C(F)  \min\big\{\|\nabla H\|_\infty^2, n^2 \|H\|_\infty\big\}
\]
for any test function $H$ and therefore
the right-hand side of \eqref{renca} is bounded by
\[
C(F) \beta \eps^{-2} \sum_{m \in \bb Z^d} (1+|m|^2)^{-k} \min\{n^2,m^2\}.
\]
This expression is bounded by $C \beta$ for a constant $C =C(F,\eps,d)$ independent of $n$ if $2k-d>d$, that is, if $k > 1+d/2$, which proves item ii) of Proposition \ref{pD3.1}. Observe that the right-hand side of \eqref{renca} is finite for $k >d/2$, which shows that $\{M_t^n; t \in [0,T]\}$ has trajectories in $\mc D([0,T]; H_{-k}(\bb T^d)$ for $k >d/2$.
In order to prove item i) of Proposition \ref{pD3.1}, we will use the characterization of compact sets of Proposition \ref{pD1.1}. 

Let $\lambda_m = \frac{M}{|m|^{d+\eps}}$. Notice that $\lambda_m$ is summable, so by Proposition \ref{pD1.1} the set 
\[
K:=\big\{\big|\widehat{f}(m)\big|^2(1+|m|^2)^{-k} \leq \lambda_m\big\}
\]
is compact in $H_{-k}(\bb T^d)$. We have that
\[
\bb P_n\big(M_t^n \notin K\big) \leq \sum_{m \in \bb Z^d} \bb P_n\big( \big|M_t^n(\phi_m)\big|^2 (1+|m|^2)^{-k} \geq \lambda_m\big)
\]
Notice that the jumps of $\{M_t^n(\phi_m); t \in [0,T]\}$ are exactly equal to the jumps of $\{X_t^n(\phi_m); t \in [0,T]\}$. In particular, since two particle never jump at the same time, $\{M_t^n(\phi_m); t \in [0,T]\}$ has jumps of size at most $2n^{-d/2}$. By Burkholder-Davis-Gundy inequality (see Lemma C.1 in \cite{MouWeb} for the exact form used here), for any $p\geq1$ there is a constant $C_p$ such that
\[
\bb E_n\big[ \big|M_t^n(\phi_m)\big|^2\big] \leq C_p \Big( \bb E_n \big[ \<M_t^n(\phi_m)\>^p\big] + \frac{2^{2p}}{n^{pd}} \Big) \leq C_p(1+|m|^2)^p.
\]
Therefore,
\[
\bb P_n\big( M_t^n \notin K\big) \leq \sum_{m \in \bb Z^d} \frac{C_p (1+|m|^2)^{p(1-k)}}{\lambda_m^p} \leq \frac{C(p,d,\eps)}{M^p}
\]
as soon as $2p(1-k) + p(d+\eps) <-d$, that is, if
\[
k > \frac{d+\eps}{2} + \frac{d}{2p} +1.
\]
We conclude that the sequence $\{M_t^n; t \in [0,T]\}_{n \in \bb N}$ is tight with respect to the $J_1$-Skorohod topology of $\mc D([0,T]; H_{-k}(\bb T^d))$ for any $k > \frac{d+\eps}{2} +\frac{d}{2p}+1$. Since $\eps$ and $p$ are arbitrary, we obtain the restriction $k < 1+d/2$. 

Now that we know that the sequence $\{M_t^n; t \in [0,T]\}_{n \in \bb N}$ is tight we will show that all its limit points are continuous. Let
\[
\Delta_T^n := \sup_{0 \leq t \leq T} \big\|M_t^n-M_{t-}^n\big\|_{-k},
\]
the size in $H_{-k}(\bb T^d)$ of the largest jump of $\{M_t^n; t \in [0,T]\}$. By Proposition \ref{pD3.2}, it is enough to show that $\Delta^n_T \to 0$ in probability as $n \to \infty$.
We already observed that the jumps of $\{M_t^n;t \in [0,T]\}$ are the same of $\{X_t^n; t \in [0,T]\}$. All jumps of $\{X_t^n; t \in [0,T]\}$ are of the form $\pm n^{-d/2} \big( \delta_{\frac{x+b}{n}}-\delta_{\frac{x}{n}}\big)$. Since $\delta_y \in H_{-k}(\bb T^d)$ for any $y$ and any $k>d/2$, by translation invariance we conclude that 
\[
\Delta_T^n \leq \sup_{\subind} n^{-d/2} \big\|\delta_{\frac{x+b}{n}} -\delta_{\frac{x}{n}}\big\|_{-k} \leq C(k) n^{-d/2},
\]
and any limit point of $\{M_t^n; t \in [0,T]\}$ has continuous trajectories.

Let $\{M_t; t \in [0,T]\}$ be a limit point of $\{M_t^n; t \in [0,T]\}_{n \in \bb N}$. We have just proved that the trajectories of $\{M_t; t \in [0,T]\}$ are continuous. Our objective is to prove that $\{M_t: t \in [0,T]\}$ is a martingale admitting the representation
\begin{equation}
\label{maipu}
M_t(H) = \sum_{i=1}^d \int_0^t d \mc W_s^i\big( \sqrtt{2u_s(1-u_s)} \partial_i H\big)
\end{equation}
for any test function $H \in \mc C^\infty(\bb T^d)$, where the processes $\{\mc W^i_t; t \in [0,T], i=1,\dots,d\}$ are independent, cylindrical Wiener processes. By the Cram\'er-Wold device and L\'evy's characterization theorem, it is enough to prove that for any test function $H \in \mc C^\infty(\bb T^d)$, $\{M_t(H); t \in [0,T]\}$ is a continuous martingale of quadratic variation 
\[
\<M_t(H)\> = \int_0^t \int 2 u(s,x)(1-u(s,x)) \|\nabla H(x)\|^2 dx ds.
\]
Recall that $\<M_t^n(H)\>$ is given by \eqref{santiago2}. Therefore, the convergence of  $\<M_t^n(H)\>$ to $\<M_t(H)\>$ follows at once from Corollary \ref{t3}. By Theorem VIII.3.11 of \cite{JacShi}, the limiting process $\{M_t(H); t \in [0,T]\}$ is a martingale of quadratic variation $\<M_t(H)\>$, which proves that $\{M_t; t \geq 0\}$ admits the representation \eqref{maipu}. Since this representation characterizes the law of $\{M_t; t \in [0,T]\}$, the sequence $\{M_t^n; t \in [0,T]\}$ has a unique limit, and therefore it is convergent. 

Summarizing, we have proved the following result:

\begin{theorem}
\label{t5}
The sequence of martingales $\{M_t^n; t \in [0,T]\}_{n \in \bb N}$ converges in law with respect to the $J_1$-Skorohod topology of $\mc D([0,T]; H_{-k}(\bb T^d))$ for any $k>1+d/2$ to the martingale $\{M_t; t \in [0,T]\}$ given by
\[
M_t(H) = \int_0^t \sum_{i=1}^d d\mc W_t^i\big(\sqrtt{2u_s(1-u_s)} \partial_i H\big)
\]
for any test function $H \in \mc C^\infty(\bb T^d)$, where $\{\mc W_t^i; t \in [0,T], i=1,\dots,d\}$ is an independent family of cylindrical Wiener processes in $L^2(\bb T^d)$.
\end{theorem}

\begin{remark}
Looking carefully at Corollary \ref{c1}, we see that this Theorem holds for any dimension $d$ and under the condition
\[
\lim_{n \to \infty} \frac{H_n(0)}{n^d} =0.
\]
\end{remark}

\begin{remark}
This theorem is optimal with respect to the topology in the following sense. By Proposition \ref{pD4.1}, $\{\mc W_t^i; t \in [0,T]\}$ has values in $H_{-k}(\bb T^d)$ if and only if $k > d/2$. Since $dM_t = \nabla \cdot \sqrtt{2u_t(1-u_t)} d\mc W_t$, the additional derivative tells us that $\{M_t; t \in [0,T]\}$ has values in $H_{-k}(\bb T^d)$ if and only if $k>1+d/2$.
\end{remark}

\subsection{Convergence of finite-dimensional laws}

\label{s6.3}

In this section we prove Theorem \ref{t3}. We will see that Theorem \ref{t3} is a simple consequence of the convergence of the martingales proved in Theorem \ref{BG} and the Boltzmann-Gibbs principle proved in Theorem \ref{t5}. For each $t \in [0,T]$, let $\bb L_t$ be the operator given by
\begin{equation}
\label{concepcion}
\bb L_t f(x) := \Delta f(x) + 2(1-2u(t,x)) F(x) \cdot \nabla f(x)
\end{equation}
for any $f \in \mc C^\infty(\bb T^d)$ and any $x \in \bb T^d$. For each $f \in \mc C^\infty(\bb T^d)$, let $\{P_{s,t}f ; 0 \leq s \leq t \}$ be the solution of the backwards Fokker-Planck equation
\begin{equation}
\label{semigrupo}
\left\{
\begin{array}{r@{\;=\;}l}
\partial_s v + \bb L_s v & 0 \text{ for } s \leq t\\
v_t & f.
\end{array}
\right.
\end{equation}
By Proposition \ref{pC2}, $s \mapsto P_{s,t} f$ is a smooth function. Notice that $\Lambda_s^n$ is a discrete approximation of order $\mc O(n^{-2})$ of $\bb L_s$, in the sense that there is a finite constant $C = C(u_0,F,\alpha)$ such that for any function $f \in \mc C^{2+\alpha}$, 
\begin{equation}
\label{matucana}
\sup_{s \in [0,T]} \sup_{x \in \tdn} \big| \Lambda_{s,x}^n f - \bb L_s f \big(\tfrac{x}{n}\big) \big| \leq \frac{C \|f\|_{\mc C^{2+\alpha}}}{n^{ \min\{2,\alpha\}}}
\end{equation}
Using $P_{\cdot,t} f$ as a test function in \eqref{santiago3}, we see that
\begin{equation}
\label{cerronavia}
X_t^n(f) = X_0^n(P_{0,t} f) + \mc R_t^n(P_{\cdot,t} f) +\mc A_t^n\big( P_{\cdot,t} f) + \mc Q_t^n( P_{\cdot,t} f) + M_t^n(P_{\cdot,t} f).
\end{equation}
By hypothesis, $X_0^n(P_{0,t}f)$ converges to $X_0( P_{0,t} f)$. By Theorem \ref{t5},  $M_t^n(P_{\cdot,t} f)$ converges to $M_t(P_{\cdot,t}f)$. 

From \eqref{matucana} and the definition of $P_{s,t} f$, we see that 
\[
\big|\mc A_t^n(P_{\cdot,t} f) \big| \leq C n^{d/2-2} \int_0^t \|P_{s,t} f\|_{\mc C^4} ds,
\]
which goes to $0$ for $d<4$.
Since $u$ is a smooth solution of the hydrodynamic equation \eqref{echid}, 
\begin{equation}
\label{cartagena}
\|\mc R_t^n(P_{\cdot,t} f)\|_\infty \leq Ct n^{d/2-2} \|f\|_\infty \|u\|_{\mc C^4}.
\end{equation}
The integral $\mc Q_t^n(P_{\cdot,t} f)$ is the sum of $d$ terms, each one them satisfying the conditions of Theorem \ref{BG}. Therefore $\mc Q_t^n(P_{\cdot,t} f) \to 0$ as $n \to \infty$.

We conclude that
\begin{equation}
\label{pudahuel}
\lim_{n \to \infty} X_t^n(f) = X_0(P_{0,t} f) + M_t(P_{\cdot,t} f).
\end{equation}
Notice that this relation does not identify the limit law, since we do not know the relation between $X_0$ and $M_t$. Let $ \{\mc F_t^n; t \in [0,T]\}$ the filtration generated by the process $\eta^n(\cdot)$. Since $\{M_t^n; t \in [0,T]\}$ is a martingale with respect to the filtration $\{\mc F_t^n; t \in [0,T]\}$ and $X_0^n$ is $\mc F_0^n$-measurable, we can assume that $\{M_t; t \in [0,T]\}$ and $X_0$ are defined in the same probability space, on which there is a filtration $\{\mc F_t; t \inT\}$ with respect to which $\{M_t; t \inT\}$ is a martingale and $X_0$ is $\mc F_0$-measurable. For any test function $H \in \mc C^\infty(\bb T^d)$, $\{M_t(H); t \inT\}$ is a real, continuous martingale of deterministic quadratic variation, starting at zero. Therefore, by L\'evy's characterization theorem, $\{M_t; t \inT \}$ is independent of $\mc F_0$ and in particular of $X_0$. We conclude that \eqref{pudahuel} characterizes the law of $X_t(f):= \lim_{n \to \infty} X_t^n(f)$. Notice that the relation
\[
X_t(f) = X_0(P_{0,t} f) + M_t(P_{\cdot,t} f) 
\]
defines a process $\{X_t; t \in [0,T]\}$ with values in $H_{-k}(\bb T^d)$ for $k > d/2+1$. According to Proposition \ref{pD4.2} this process is the solution of \eqref{SHE} with initial condition $X_0$. In order to complete the proof of Theorem \ref{t3}, we need to prove that for any $\ell \in \bb N$, any $0 \leq t_1 \leq \dots \leq t_\ell \leq T$ and any $f_1,\dots,f_\ell \in \mc C^\infty(\bb T^d)$,
\[
\lim_{n \to \infty} \big( X_{t_1}^n(f_1),\dots,X_{t_\ell}^n(f_\ell)\big) = \big(X_{t_1}(f_1),\dots,X_{t_\ell}(f_\ell)\big)
\]
in law. We just proved the case $\ell=1$. We proceed by induction. Assume that the limit above holds for $\ell$. By Cram\'er-Wold's device, the convergence of the couple $\big( X_{t_1}^n(f_1),\dots,X_{t_\ell}^n(f)\big)$ for any $f \in \mc C^\infty(\bb T^d)$ implies the convergence of the couple $\big( X_{t_1}^n(f_1),\dots,X_{t_\ell}^n(f_\ell), X_{t_\ell}^n(f_{\ell+1})\big)$ to its corresponding limit. Therefore, we can assume that $t_\ell< t_{\ell+1}$. In that case, the same proof of above shows that
\[
\lim_{n \to \infty} X_{t_{\ell+1}}^n(f_{\ell+1}) = X_{t_\ell}(P_{t_\ell,t_{\ell+1}} f) + M_{t_\ell,t_{\ell+1}}(P_{\cdot,t_{\ell+1}}f ),
\]
where $\{M_{t_\ell,t}; t \in [t_\ell,t_{\ell+1}]\}$ is a continuous martingale, independent of $\mc F_{t_\ell}$, defined in a common probability space with $\big( X_{t_1}^n(f_1),\dots,X_{t_\ell}^n(f_\ell)\big)$, in such a way that $\big( X_{t_1}^n(f_1),\dots,X_{t_\ell}^n(f_\ell)\big)$ is $\mc F_{t_\ell}$-measurable for any $f_1,\dots,f_\ell \in \mc C^\infty(\bb T^d)$. The quadratic variation of $\{M_{t_\ell, t} ; t \in [t_\ell,t_{\ell+1}]\}$ is equal to 
\[
\int_{t_\ell}^t \int 2 u_s (1-u_s) \big\|\nabla P_{s,t_{\ell+1}} f_{\ell+1} \big\|^2 dx ds.
\]
Since the process $\{X_t; t \in [0,T]\}$ satisfies the identity
\[
X_t(f) = X_s(P_{s,t} f) + M_t(P_{\cdot,t}f) - M_s(P_{\cdot,t} f)
\]
for any $0 \leq s \leq t \leq T$ and any test function $f \in \mc C^\infty(\bb T^d)$, we conclude that the sequence $\big( X_{t_1}^n(f_1),\dots,X_{t_{\ell+1}}^n(f_{\ell+1})\big)$ is convergent and its limit is equal to the vector $\big( X_{t_1}(f_1),\dots,X_{t_{\ell+1}}(f_{\ell+1})\big)$, which ends the proof of Theorem \ref{t3}.

\section{Tightness of the density fluctuation fields}
\label{s7}
In this section we prove tightness of the density fluctuation fields $\{X_t^n; t \inT\}_{n \in \bb N}$. Since the proofs depend on the dimension and also on the hypotheses over $H_n(0)$, we will treat each dimension separately. The reader satisfied with convergence of finite-dimensional laws can skip this section.

\subsection{The case \texorpdfstring{$d=1$}{d=1}} 
\label{s7.1}
In this section we assume that $d=1$ and that $\sup_n H_n(0) <+\infty$. In this case, Theorem \ref{t2} says that there is a finite constant $C = C(T,\eps_0,F)$ such that $H_n(t) \leq C$ for any $t \inT$ and any $n \in \bb N$. In decomposition \eqref{santiago3}, the terms $\{X_0^n\}_{n \in \bb N}$ and $\{M_t^n; t \in [0,T]\}_{n \in \bb N}$ are already tight, so it is enough to show tightness of the integral processes $\{\mc R_t^n; t \in [0,T]\}$, $\{\mc A_t^n; t \inT\}_{n \in \bb N}$, $\{\mc Q_t^n; t \inT\}_{n \in \bb N}$. The idea is to estimate the probabilities
\[
\bb P_n \big( \big\|\mc R_t^n -\mc R_s^n\big\|_{-k} > \lambda \big), \quad 
\bb P_n \big( \big\|\mc A_t^n -\mc A_s^n\big\|_{-k} > \lambda \big), 
\quad \bb P_n \big( \big\|\mc Q_t^n -\mc Q_s^n\big\|_{-k} > \lambda \big)
\]
and to invoke Kolmogorov-Centsov's criterion, stated in Proposition \ref{pD3.3}. Since $\big\|\mc A_t^n -\mc A_s^n\big\|_{-k}$ is a sum of squares and Corollary \ref{c2} does not hold for $p=2$, our estimates will be somehow indirect. For each $\delta>0$, define $a_m = c_\delta (1+m^2)^{-(1/2+\delta)}$, choosing $c_\delta$ in such a way that $\sum_m a_m =1$. Since $\delta>0$, $c_\delta$ is well defined. Then, for any $p \in [0,2)$,
\[
\begin{split}
\bb P_n\big( \big\| \mc A_t^n- \mc A_s^n \big\|_{-k} > \lambda \big)
		&\leq \sum_{m \in \bb Z} \bb P_n \big( \big|\mc A_t^n(\phi_m) -\mc A_s^n(\phi_m)\big|^2 > \lambda^2 a_m(1+m^2)^k\big) \\
		&\leq \frac{1}{\lambda^p} \sum_{m \in \bb Z} \frac{\bb E_n  \big[ \big|\mc A_t^n(\phi_m) -\mc A_s^n(\phi_m)\big|^p\big]}{(1+m^2)^{p(k-1/2-\delta)/2} }
\end{split}
\]
On the other hand, by Corollary \ref{c2} and since $\Lambda_s^n$ is an approximation of order $\mc O(n^{-2})$ of $\bb L_s$, $\|\Lambda_s^n \phi_m\|_\infty \leq C m^2$ and
\[
\bb E_n\big[ \big|X_s^n(\Lambda_s^n \phi_m)\big|^p\big] \leq C(\eps_0,F,T,p) |m|^{2p}.
\]
Recall definition \ref{santiago4}. By \eqref{c2'}, the previous estimate gives that
\[
\bb E_n \big[ \big| \mc A_t^n(\phi_m) -\mc A_s^n(\phi_m)\big|^p \big] \leq C |m|^{2p} |t-s|^p,
\]
from where
\[
\bb P_n\big( \big\| \mc A_t^n- \mc A_s^n \big\|_{-k} > \lambda \big)
		\leq \frac{C}{\lambda^p} \sum_{m \in \bb Z} \frac{|m|^{2p} |t-s|^p}{(1+m^2)^{p(k-1/2-\delta)/2}} \leq \frac{C|t-s|^p}{\lambda^p}
\]
as soon as $(k-\frac{1}{2} -\delta)p -2p >1$. Since $\delta$ and $p$ can be taken arbitrarily close to $0$ and $2$ respectively, we have proved that for any $k>3$ there exists $p \in (1,2)$ and $C=C(\eps_0,F,T,p)$ finite such that
\[
\bb P_n\big( \big\| \mc A_t^n- \mc A_s^n \big\|_{-k} > \lambda \big) \leq \frac{C|t-s|^p}{\lambda^p},
\]
which proves tightness of the sequence $\{\mc A_t^n; t \inT\}_{n \in \bb N}$ in $\mc C^{\alpha}([0,T]; H_{-k}(\bb T))$ for $k>3$ and $\alpha <1/2$. 

We have already seen in \eqref{cartagena} that $\mc R_t^n$ goes uniformly to 0. The computation will be same in any dimension, so we will do it for general $d <4$. Since $|\tfrac{d}{dt} \mc R_t^n(f))| \leq C n^{d/2-2} \|f\|_\infty$, 
\begin{equation}
\label{concon}
\big\|\mc R_t^n -\mc R_s^n \big\|_{-k}^2 \leq C|t-s|^2 n^{d-4} \sum_{m \in \bb Z} \big(1+|m|^2\big)^{-k},
\end{equation}
and $\{\mc R_t^n; t \in [0,T]\}_{n \in \bb N}$ goes to $0$ as $n \to \infty$ uniformly in $\mc C^{\alpha}(H_{-k}(\bb T^d))$ for any $\alpha <1$ and any $k > d/2$.

Notice that Corollary \ref{c2} and estimate \eqref{c2'} also apply to $\{\mc Q_t^n; t \inT\}$, this time taking as a test function $H_x =2 n \big( \phi_m \big(\tfrac{x+b}{n}\big) -\phi_m\big( \tfrac{x}{n} \big) \big) F_b^n(x)$. In this case, $\|H\|_\infty \leq 2 \|F\|_\infty |m|$. Repeating the proof above, we see that
\[
\bb P_n\big( \big\| \mc Q_t^n- \mc Q_s^n \big\|_{-k} > \lambda \big)
		\leq \frac{C}{\lambda^p} \sum_{m \in \bb Z} \frac{|m|^p |t-s|^p}{(1+m^2)^{p(k-1/2-\delta)/2}},
\]
which proves tightness of the sequence $\{\mc Q_t^n; t \inT\}_{n \in \bb N}$ with respect to the topology of $\mc C^\alpha([0,T]; H_{-k}(\bb T^d))$ for $k >2$ and $\alpha <1/2$. For the sequence of martingales $\{M_t^n; t \inT\}_{n \in \bb N}$, the restriction is $k >3/2$, see Theorem \ref{t5}, while the natural restriction for the convergence of $\{X_0^n\}_{n \in \bb N}$ is $k >1/2$. We have proved the following result:

\begin{theorem}
\label{t6}
Assume that $d=1$. Let $k>3$ and let $\{\eta^n(0); n \in \bb N\}$ be a sequence of initial conditions such that:
\begin{itemize}
\item[i)] $\sup_{n \in \bb N} H_n(0) <+\infty$,

\item[ii)] $X_0^n \to X_0$ in law with respect to the strong topology of $H_{-k}(\bb T)$. 
\end{itemize}

Then, for any $T >0$, the sequence $\{X_t^n; t \in [0,T]\}_{n \in \bb N}$ converges in law to the solution of \eqref{SHE} with initial condition $X_0$ with respect to the $J_1$-Skorohod topology of $\mc D([0,T]; H_{-k}(\bb T))$.
\end{theorem} 

\begin{remark}
The method of proof of this section uses in a fundamental way that $H_n(t)$ is uniformly bounded in $t$ and $n$. Therefore, we will need a different method of proof for dimensions $d=2,3$. As we can see from the proof of Theorem \ref{t6}, the most problematic term is $\mc A_t^n$.
\end{remark}

\begin{remark}
As observed in pg.~297 of \cite{GonJarSet}, the condition $H_n(0) \leq K$ already allows to create some non-trivial initial conditions. For example, one can take $\eta^n(0)$ with law
\[
\bigotimes_{x \in \bb T_n} \Bern\big(u\big(\tfrac{x}{n}\big)+ \tfrac{1}{\sqrt n} \kappa_x^n\big),
\]
where $\kappa_x^n:=\kappa\big(\tfrac{x}{n}\big)$ and $\kappa$ is continuous. In that case, $H_n(0) \leq C(u_0) \|\kappa\|_\infty^2$ and $X_0^n$ converges to $\sqrtt{u_0(1-u_0)}\xi +\kappa$, where $\xi$ is a  white noise. The function $\kappa$ could even be random, as long as it is continuous with probability $1$ and independent of the dynamics. An example that could be interesting in some situations is the case on which $\kappa$ is a Brownian bridge in $\bb T$.
\end{remark}

\subsection{The case \texorpdfstring{$d=2$}{d=2}}
\label{s7.2}
 As in the case $d=1$, by \eqref{santiago3} it is enough to show tightness of $\{\mc R_t^n; t \in [0,T]\}_{n \in \bb N}$, $\{\mc A_t^n; t \inT\}_{n \in \bb N}$ and $\{\mc Q_t^n; t \inT\}_{n \in \bb N}$. We have already proved tightness of $\{\mc R_t^n; t \in [0,T]\}_{n \in \bb N}$ in \eqref{concon}. In order to prove tightness of $\{\mc A_t^n; t \inT\}_{n \in \bb N}$, it was enough to use the bound
\[
\bb E_n\big[ \big| X_t^n(H) \big|^p\big] \leq C \|H\|_\infty^p,
\]
so our first objective will be to prove that bound. Notice that this bound does not follow from Corollary \ref{c2}, since in $d=2$ the bound on the relative entropy grows with $n$. It is enough to bound each of the terms on the right-hand side of \eqref{cerronavia}. The simplest is $X_0(P_{0,t} f)$. When $\eta^n(0)$ has law $\mu_0^n$, $X_0^n(H)$ is a sum of independent random variables, from where it is easy to show that for any $p \geq 1$ there exists a finite constant $C_p$ such that
\[
\bb E_n \big[ \big|X_0^n(H)\big|^p\big] \leq C_p \|H\|_\infty^p
\]
for any $n \in \bb N$ and any $H: \tdn \to \bb R$. In the general case, it is necessary to postulate this bound, at least for $p \in [1,2)$. In that case, the contractivity of $P_{0,t}$ shows that
\begin{equation}
\label{puentealto}
\bb E_n \big[ \big|X_0^n(P_{0,t} f)\big|^p\big] \leq C_p \|f\|_\infty^p.
\end{equation}
The martingale term is also easy to deal with: by Corollary \ref{c2}, 
\begin{equation}
\label{alhue}
\bb E_n \big[M_t^n(P_{\cdot,t} f)^2 \big] = \bb E_n \big[\big\<M_t^n(P_{\cdot,t} f)\big\> \big] \leq C t \|\nabla P_{\cdot,t} f\|_\infty^2.
\end{equation}
By interpolation, the analogous estimate holds for any $p \in [1,2)$. 

In order to simplify the notation, let $\overline{\hspace{-2pt}\smash{H}\vphantom{t}}_n(t)$ be the right-hand side of \eqref{entropy2}.
Using \eqref{lagranja} for 
\[
2 n \big( P_{s'\!\!,t} f\big(\tfrac{x+b}{n}\big) -P_{s'\!\!,t}  f\big(\tfrac{x}{n}\big) \big)F_b^n(x)
\]
with $s=0$ and $\gamma = (\|F\|_\infty \|\nabla P_{\cdot,t} f\|_\infty \sqrt{t})^{-1}$, we get the bound
\begin{equation}
\label{conchali}
\bb P_n \big( \big|\mc Q_t^n(P_{\cdot,t} f)\big| >\lambda\big) \leq \frac{C \varH_n(T) \sqrt{t} \|\nabla P_{\cdot,t} f\|_\infty}{\lambda n}.
\end{equation}
Notice that this bound is not enough to get a moment bound for $\mc Q_t^n(H)$, since we need an exponent larger than $1$ for $\lambda$ in the denominator. But the right-hand side of this estimate converges to $0$ as $n \to \infty$. Therefore, we could try to interpolate this estimate with an estimate diverging in $n$, but better in the exponent of $\lambda$. In fact, using Corollary \ref{c2} and estimate \eqref{c2'} for $\mc Q_t^n(H)$, we obtain the estimate
\[
\bb E_n \big[ \big| \mc Q_t^n(P_{\cdot,t} f) \big|^p\big] \leq C t^p \|\nabla P_{\cdot,t} f\|_\infty^p \varH_n(T)^{p/2},
\]
from where 
\begin{equation}
\label{recoleta}
\bb P_n \big( \big|\mc Q_t^n(P_{\cdot,t} f)\big| >\lambda\big)  \leq \frac{C \|\nabla P_{\cdot,t} f\|_\infty^p \varH_n(T)^{p/2}t^p}{\lambda^p}
\end{equation}
for any $p \in [1,2)$. Now we need to make some assumption about the growth of the relative entropy. Assume that $\varH_n(T) \leq C n^a$ for some finite constant $C$ and some $a>0$. We can assume that $a<1$, since this is already an hypothesis of the Boltzmann-Gibbs principle, see Theorem \ref{BG}. Using the simple interpolation bound $\min\{A,B\} \leq A^\theta B^{1-\theta}$, valid for any $\theta \in [0,1]$, \eqref{conchali} and \eqref{recoleta} give us the estimate
\[
\bb P_n \big( \big|\mc Q_t^n(P_{\cdot,t} f)\big| >\lambda\big) \leq \frac{C \|\nabla P_{\cdot,t} f\|_\infty^{\theta +(1-\theta)p}t^{\theta/2+(1-\theta)p}}{\lambda^{\theta +(1-\theta)p} n^{(1-a)\theta - (1-\theta)ap/2}}.
\]
The optimal choice is $\theta = \frac{ap}{2-2a+ap}$, which gives us the estimate
\begin{equation}
\label{sanramon}
\bb P_n \big( \big|\mc Q_t^n(P_{\cdot,t} f)\big| >\lambda\big) \leq \frac{C \|\nabla P_{\cdot,t} f\|_\infty^{p'}t^{p''}}{\lambda^{p'}}
\end{equation}
for $p' = \frac{2p-ap}{2-2a+ap} = p -\frac{ap(p-1)}{2-2a+ap}$ and $p'' = \frac{2-\frac{3}{2} ap}{2-2a+ap}$. At this point the dependence on $t$ is not important, but later on we will use it. Taking $p$ arbitrarily close to $2$, we can take $p'$ arbitrarily close to $2-a$, on which case $p''$ gets arbitrarily close to $2-\frac{3}{2}a$. We conclude that for any $p' < 2-a$ there exists a finite constant $C$ such that
\begin{equation}
\label{lapintana}
\bb E_n \big[ \big| \mc Q_t^n(P_{\cdot,t} f) \big|^{p'}\big] \leq C \|\nabla P_{\cdot,t} f\|_\infty^{p'}.
\end{equation}
\
Using Corollary \ref{c2} and \eqref{c2'} for $\mc A_t^n(P_{\cdot,t} f)$, we obtain the estimate
\begin{equation}
\bb E_n \big[ \big| \mc A_t^n(P_{\cdot,t} f) \big|^p \big] 
	\leq C t^p \varH_n(t)^{p/2} \big\|\big(\tfrac{d}{ds} + \Lambda_s^n\big)P_{\cdot,t} f \big\|_\infty^p
	\leq C t^p n^{(a-2\alpha)p/2}\|P_{\cdot,t} f\|_{\mc C^{2+\alpha}}^p
\label{pirque}
\end{equation}
for any $\alpha \in [0,2]$. Since we just need this expectation to be bounded, the optimal choice is $\alpha =a/2$.

By Proposition \ref{pC2}, $\|P_{\cdot,t} f \|_{\mc C^{2+a/2}} \leq C\|f\|_{\mc C^{2+a/2}}$, since $2+a/2 \notin \bb N$. 
Putting estimates \eqref{concon}, \eqref{puentealto}, \eqref{alhue}, \eqref{lapintana} and \eqref{pirque} into \eqref{cerronavia}, we conclude that for any $a <1$ and any $p <2-a$ there exists $C=C(\eps_0,F,T,p,a)$ such that
\begin{equation}
\label{melipilla}
\bb E_n\big[ \big| X_t^n(f) \big|^p\big] \leq C \|f\|_{\mc C^{2+a/2}}^p.
\end{equation}
In particular,
\[
\bb E_n\big[ \big| X_t^n(\Lambda_t^n \phi_m) \big|^p\big] \leq C |m|^{(4+a/2)p},
\]
and the proof of tightness of $\{\mc A_t^n; t \in [0,T]\}_{n \in \bb N}$ of Section \ref{s7.1} can be repeated here. The only difference is that we have to define $a_m = c_{\delta} (1+|m|^2)^{-(1+\delta)}$, since now the sum is over $m \in \bb Z^2$. Then we obtain the bound
\[
\bb P_n \big( \big\|\mc A_t^n -\mc A_s^n\big\|_{-k} > \lambda \big) 
		\leq \frac{C|t-s|^p}{\lambda^p} \sum_{m \in \bb Z^2} \frac{|m|^{(4+a/2)p}}{(1+|m|^2)^{p(k-1-\delta)/2}}.
\]
The sum is finite if $k> \frac{2}{p} +4 +\frac{a}{2} +1+\delta$, which gives the restriction $k> 5+\frac{a}{2}+\frac{2}{2-a}$. Therefore, $\{\mc A_t^n; t \in [0,T]\}_{n \in \bb N}$ is tight in $\mc C([0,T]; H_{-k}(\bb T^2))$ for $k > 5 +\frac{a}{2}+\frac{2}{2-a}$. 

In order to prove tightness of $\{\mc Q_t^n;  t \inT\}_{n \in \bb N}$, it is enough to observe that \eqref{sanramon} also holds for $\mc Q_t^n(\phi_m)-\mc Q_s^n(\phi_m)$. Therefore,
\[
\bb P_n \big( \big| \mc Q_t^n(\phi_m) -\mc Q_s^n(\phi_m)\big| > \lambda\big) 
		\leq \frac{C |m|^{p'}|t-s|^{p''}}{\lambda^{p'}}.
\]
Repeating the computations performed to prove tightness in Section \ref{s7.1}, we see that
\[
\bb P_n \big( \big\| \mc Q_t^n -\mc Q_s^n\big\|_{-k} > \lambda \big) 
		\leq \frac{C|t-s|^{p''}}{\lambda^{p'}} \sum_{m \in \bb Z^2} \frac{|m|^{p'}}{(1+|m|^2)^{p'(k-1-\delta)/2}}.
\]
The restriction $p'(k-1-\delta)> p'+2$ imposes the condition $k>2+\frac{2}{2-a}$, while the restriction $p''>1$ imposes the condition $a <\frac{2}{3}$. We conclude that $\{\mc Q_t^n; t \in [0,T]\}_{n \in \bb N}$ is tight in $\mc C([0,T]; H_{-k}(\bb T^2))$ for $k > 2+ \frac{2}{2-a}$. We have proved the following theorem:

\begin{theorem}
\label{t7} Assume that $d=2$. Let $a \in [0,\frac{2}{3})$ and let $k>5+\frac{a}{2}+\frac{2}{2-a}$. Let $\{\eta^n(0); n \in \bb N\}$ be a sequence of initial conditions such that:
\begin{itemize}
\item[i)] there exists $C$ finite such that $H_n(0) \leq C n^a$ for any $n \in \bb N$,
\item[ii)] $X_0^n \to X_0$ in law with respect to the topology of $H_{-k}(\bb T^2)$,
\item[iii)] for any $p <2-a$ there exists a finite constant $C_p$ such that for any test function $f$, $\bb E_n[|X_0^n(f)|^p] \leq C_p \|f\|_\infty^p$. 
\end{itemize}
Then, $\{X_t^n; t \in [0,T]\}_{n \in \bb N}$ converges in law to the solution of \eqref{SHE} with initial condition $X_0$ with respect to the $J_1$-Skorohod topology of $\mc D([0,T]; H_{-k}(\bb T^2))$. 
\end{theorem}

\begin{remark}
The interested reader may verify that the proof presented in this section proves Theorem \ref{t6} in dimension $d=1$ under the condition $H_n(0) \leq C n^a$ for $a < \frac{1}{3}$ and some $k$ large enough..
\end{remark}

\subsection{The case \texorpdfstring{$d=3$}{d=3}} 
\label{s7.3}
In $d=3$, \eqref{conchali} becomes
\[
\bb P_n \big( \big|\mc Q_t^n(P_{\cdot,t} f)\big| >\lambda\big) \leq \frac{C \varH_n(T) \sqrt{t} \|\nabla P_{\cdot,t} f\|_\infty}{\lambda n^{3/2}}
\]
and \eqref{recoleta} stays the same. However, $\varH_n(T)$ is at least of order $\mc O(n)$ and therefore we can only assume that $\varH_n(T) \leq C n^a$ for $a \in [1,\frac{3}{2})$. In that case, one can verify that the interpolation bound gives exponents 
\[
p' = \frac{(3-a)p}{3-a(2-p)}, \quad p'' = \frac{6p-3ap}{6-2a(2-p)}.
\]
When $p$ goes to $2$, $p'$ goes to $\frac{2(3-a)}{3}$ and $p''$ goes to $2-a$. In fact, no matter how we choose $a$ and $p$, $p''$ is always smaller than $1$. Therefore, our proof of tightness for $\{\mc Q_t^n; t \in [0,T]\}_{n \in \bb N}$ does not work. However, for any $a \in [1,\frac{3}{2})$, it is possible to find $p$ close enough to $2$ such that $p'>1$, and the estimate \eqref{melipilla} holds. Therefore, we have that for any $a \in [1,\frac{3}{2})$, there exists $p'>1$ such that 
\[
\bb E_n\big[ \big| X_t^n(f) \big|^{p'}\big] \leq C \|f\|_{\mc C^{2+a/2}}^{p'}.
\]
Repeating the computations of Section \ref{s7.2}, we can check that this implies tightness for the integral process $\{\mc A_t^n; t \in [0,T]\}_{n \in \bb N}$ in $\mc C([0,T], H_{-k}(\bb T^3))$ for $k>\frac{11}{2}+\frac{a}{2} + \frac{9}{2(3-a)}$. Actually, by Kolmogorov-Centsov criterion $\{\mc A_t^n; t \in [0,T]\}_{n \in \bb N}$ is tight in $\mc C^\alpha([0,T], H_{-k}(\bb T^3))$ for $\alpha <1-\frac{3}{2(3-a)}$. Since $X_t^n = \partial_t \mc A_t^n$ in the distributional sense, we conclude that the sequence $\{X_t^n; t \in [0,T]\}_{n \in \bb N}$ is tight in the path space $\mc C^{-1+\alpha}([0,T]; H_{-k}(\bb T^3))$ for $k > \frac{11}{2}+\frac{a}{2} + \frac{9}{2(3-a)}$. The exact values of the constants $a, k$ are not very important, since they are far from being optimal; we just wanted to point out that they exist and that they can be explicitly estimated.

\section{Discussion and concluding remarks}

\subsection{The gradient condition and generalizations}

Our derivation of the large-scale limit of the density fluctuations stated in Theorem \ref{t3} holds in $d<4$ for basically the same class of models for which the entropy method of \cite{GuoPapVar} works. These models satisfy the so-called {\em gradient condition}, see Remark 4.2.4 in \cite{KipLan} and Section II.2.4 in \cite{Spo}. In general, these systems have a hydrodynamic limit governed by a partial differential equation of parabolic type. Just to see how these systems could be defined, let us present an example. Fix $K \in \bb N$ and consider the operator $L_n^a$ given by
\[
L_n^a f(\eta)
		= n^2 \sum_{\subind}\Big(a_b^0+ \sum_{k=1}^K a_b^k \Big( \sum_{\ell=0}^k\prod_{i=1}^{k-\ell} \eta_{x-ib}\prod_{j=1}^\ell \eta_{x+(j+1)b}\Big) \Big) \nabla_{x,x+b} f(\eta).
\]
Here the coefficients $a_b^k$ can be smooth, of the form $a_b^k\big(t,\frac{x}{n}\big)$, but they do not depend on $\eta$. Define
\[
a_b(u) := \sum_{k=1}^K a_b^k (k+1) u^k.
\]
Under the condition $a_b \geq 0$, $L^a_n$ is the generator of a particle system in $\Omega_n$.
The process generated by this operator satisfies the gradient condition, see Remark 2.3 of \cite{FarLanMou}. The corresponding hydrodynamics equation is given by
\[
\partial_t u = \nabla \cdot \big(a(u) \otimes \nabla u\big),
\]
where $a \otimes \nabla u$ is the vector with coordinates $a_b(u) \partial_b u$.
Under the additional ellipticity condition $a_b \geq \eps_2>0$ for any $b, x$ and $t$, our methods apply without important modifications. The key point is that any local function can be decomposed into a finite sum of functions of the form $\omega_{x+A}$. For this dynamics, we need to use the main lemma for sets $A$ with up to $K+1$ points. It is for that reason that we proved the main lemma in that generality.

One can include a more general weakly asymmetric term in the dynamics as well. For each $b \in \mc B$, let $g^b:\Omega \to \bb [0,\infty)$ be a fixed local function. Let $g_{x}^b$ be the translation of $g^b$ by $x$. Let $F$ be a smooth vector field and define $F_b = F \cdot b$. Then consider the operator
\[
L_n^g := n \sum_{\subind} F_b \big(\tfrac{x}{n}\big)  g_x^b(\eta)\big(\eta_x- \eta_{x+b}\big) \nabla_{x,x+b} f(\eta).
\]
The operator $L_n^g$ is not a generator, but since $\eps_2>0$, for $n$ large enough $L_n^a+L_n^g$ is indeed a generator. For $\rho \in [0,1]$, define $G_b(\rho) = \int g^b(\eta) (\eta_0-\eta_b)^2 d \nu_\rho$, where $\nu_\rho$ is the Bernoulli product measure of density $\rho$. The hydrodynamic equation is modified to
\[
\partial_t u =  \nabla \cdot \big(a(u) \otimes \nabla u\big) - \nabla \cdot \big( F \otimes  G(u)\big),
\]
where $F \otimes G(u)$ is the vector of coordinates $F_b G_b(u)$.
One can even introduce a reaction term into the hydrodynamic equation. Let $c: \Omega \to [0,\infty)$ be a local function and let $c_x$ be the translation of $c$ by $x$. For $x \in \tdn$ and $\eta \in \Omega_n$, let $\eta^x \in \Omega_n$ be the configuration obtained from $\eta$ by modifying its occupation variable at $x$ from $\eta_x$ to $1-\eta_x$. Then consider the operator
\[
L_n^c = \sum_{x \in \tdn} c_x(\eta) \big( f(\eta^x) -f(\eta)\big).
\]
This process is a generator, and the process generated by $L_n^a + L_n^g + L_n^c$ has the hydrodynamic equation
\[
\partial_t u = \nabla \cdot\big(a(u) \otimes \nabla u\big) - \nabla \cdot \big( F \otimes  G(u)\big) + H(u),
\]
where $H(\rho) := \int c(\eta)(1-2\eta_0) d \nu_\rho$. For all of these modifications, Theorem \ref{t3} can be proved with minor changes in the proofs. The limiting equation of the density fluctuations is
\[
\begin{split}
\partial_t X_t
		&= \nabla \cdot\Big( a(u) \otimes \nabla X_t + \big(a'(u) \otimes \nabla u -F \otimes G'(u)\big) X_t + \\
		&\quad \quad \quad \quad  +\sqrtt{2u(1-u) a(u)} \otimes \dot{\mc W}^1_t\Big) + H'(u) X_t + \sqrtt{C(u)} \dot{\mc W}^2_t,
\end{split}
\]
where $C(\rho) : = \int c(\eta) d \nu_\rho$, $\dot{\mc W}^1$ is a vectorial space-time white noise and $\dot{\mc W}^2$ is a scalar space-time white noise independent of $\dot{\mc W}^1$.

\subsection{The critical dimension \texorpdfstring{$d=4$}{d=4}}. 

Theorem \ref{t3} was stated to hold in dimension $d<4$, which in principle is the same as $d \leq 3$. We did this in order to emphasize that the critical dimension for our method is $d=4$. From the proofs it is clear that the condition on dimension is $d-2 < d/2$. In the literature, one possible way to understand this criticality is to replace the Laplacian in the hydrodynamic limit by a {\em fractional Laplacian}. This is achieved by introducing long-range jumps into the definition of the model, see \cite{LohSlaWal} for example. Let $\lambda$ be the positive measure in $\bb T^d$ given by
\[
\lambda(dx) = \big| \sin^2(\pi x_1)+\dots+\sin^2(\pi x_d)\big|^{-\frac{d+\alpha}{2}} dx
\]
with $\alpha \in (0,2)$ and define, for $x \neq 0 \in \tdn$,
\[
\lambda_n(x) = \lambda \Big(\prod_{i=1}^d \big(x_i -\tfrac{1}{2n},x_i+\tfrac{1}{2n}\big]\Big).
\]
Define $r_n(x,y) = \mu_n(y-x)$ and consider the operator $L^\alpha_n$ given by
\[
L^\alpha_n f(\eta) = \sum_{x \neq y \in \tdn} r_n(x,y) \nabla_{x,y} f(\eta).
\]
For $\alpha >1$ and $n$ large enough, the operator $L_n^\alpha + L_n^g$ is the generator of an interacting particle system. The hydrodynamic limit of this model can be obtained as in \cite{Jar}. It is possible to check that the condition under which the corresponding analog of Theorem \ref{t3} holds for this model is $d < 2\alpha$, so Theorem \ref{t3} holds for $d \leq 3$ if $\alpha >3/2$.

\subsection{Sharpness of Theorem \ref{t2}}

We claim that the estimate \eqref{entropy1} should be optimal for product reference measures. Let us give an heuristic argument to explain why. Let us consider the simple, symmetric exclusion process in contact with reservoirs, like in \cite{BerD-SGabJ-LLan} or \cite{FarLanMou}. It is well known that when the density/temperature of the reservoirs are different, the system has a non-equilibrium stationary state (NESS) that develops long-range correlations. Although a complete prove is only available in dimension $d=1$, macroscopic fluctuation theory predicts that the density fluctuations of the NESS is Gaussian, with a covariance kernel of the form $\chi(x) \delta(x-y) - \mc G(x,y)$, where $\mc G(x,y)$ is a kernel which behaves like the Green function near the diagonal $x=y$. This Gaussian measure is absolutely continuous with respect to white noise only in dimension $d=1$. Moreover, if we average this Gaussian measure with an approximation of the identity of size $1/n$, we obtain a smooth process which is absolutely continuous with respect to the smoothening of the white noise with the same approximation of the identity, but the relative entropy between these two measures grows like $\log n$ in dimension $d=2$ and like $n^{d-2}$ in dimension $d \geq 3$. If we observe our particle systems inside a box of size $\eps n$ centered at some macroscopic point $x$ and we look at its law inside a box of size $\delta n$ around $x$ with $\delta \ll \eps \ll 1$, we can argue that due to finite speed of propagation of the heat equation (in the sense of variance), the system should be close to another system on which we put boundary conditions at the boundary of the box of size $\eps n$ with matching densities. 
In that case, after times of order $\eps^2$ the law inside the smaller box should be close to the NESS of the system with boundaries, which would imply that our entropy bound can not be improved without taking into consideration long-range correlations created by local currents. This heuristics also provides a possible way to improve Theorem \ref{t2}. If we introduce some first-order correction into the reference measures $\mu_t^n$, there is room to get extra cancellation that could improve the comparison of the main lemma. These first-order corrections were actually computed in \cite{FunUchYau} in the context of the hydrodynamic limit of non-gradient systems. It would be interesting to pursue this line of research.

\appendix

\section{Entropy methods for Markov chains}
\label{sA}
In this section we prove Yau's entropy inequality and we also establish a new variational bound for exponential moments of observables of Markov chains. Since these bounds could be of independent interest and hold for any Markov chain, we present the proofs in a general context. Afterwards we will implement these estimates for the exclusion processes considered in this article.

\subsection{Yau's entropy inequality}
\label{sA.1}

Let $\{x_t; t \geq 0\}$ be a Markov chain on a finite state space $\Omega$. We will denote by $\bb P$ the law of this chain and by $\bb E$ the expectation with respect to $\bb P$. If we need to specify the initial law $\mu$ of the chain $\{x_t; t \geq 0\}$, we will write $\bb P^\mu$ and $\bb E^\mu$.
\
Let $L$ be the generator of this chain. The action of $L$ over functions $f: \Omega \to \bb R$ can be written as
\[
L f(x) = \sum_{y \in \Omega} r(x,y) \big( f(y) - f(x) \big) \text{ for any } x \in \Omega.
\]
Let $\Gamma$ be the {\em carr\'e du champ} associated to $L$, that is, $\Gamma$ is the quadratic operator given by
\[
\Gamma f(x) = \sum_{y \in \Omega} r(x,y) \big( f(y) -f(x) \big)^2 
\]
for any $f: \Omega \to \bb R$ and any $x \in \Omega$. We say that a measure $\nu$ in $\Omega$ is a {\em reference measure} if $\nu(x) >0$ for any $x \in \Omega$. Fix a reference measure $\nu$ and fix $T >0$. Let $\{\mu_t; t \in [0,T]\}$ be a family of reference measures in $\Omega$, differentiable with respect to $t$. Let $\psi_t: \Omega_n \to [0,\infty)$ be the Radon-Nikodym derivative of $\mu_t$ with respect to $\nu$, that is, $\psi_t(x) =\frac{\mu_t(x)}{\nu(x)}$ for any $t \in [0,T]$ and any $x \in \Omega$. Let $L_t^\ast$ be the adjoint of $L$ with respect to $\mu_t$. In general, $\mu_t$ will {\em not} be an invariant measure of $\{x_t; t \geq 0\}$ and therefore $L^\ast_t$ will not be a Markovian operator. The action of $L_t^\ast$ over a function $g: \Omega \to \bb R$ is given by
\begin{equation}
\label{ecA1}
L_t^\ast g(x) = \sum_{y \in \Omega} \Big\{ r(y,x) g(y) \frac{\mu(y)}{\mu(x)} -r(x,y) g(x)\Big\}.
\end{equation}
\
Let $f_t: \Omega \to [0,\infty)$ be the density of the law of $x_t$ with respect to $\mu_t$, that is,
\[
f_t(x) := \frac{\bb P(x_t =x)}{\mu_t(x)} \text{for any } x \in \Omega \text{ and any } t \in [0,T].
\]
Let $H(t)$ be the relative entropy of the law of $x_t$ with respect to $\mu_t$, that is,
\[
H(t) := \int f_t \log f_t d \mu_t \text{ for any } t \in [0,T].
\]
We have the following estimate:

\begin{lemma}[Yau's inequality]
\label{lA1}
For any $t \in [0,T]$, 
\[
H'(t) \leq - \int \Gamma \sqrtt{f_t} d \mu_t + \int \big( L_t^\ast \mathds{1} - \tfrac{d}{dt} \log \psi_t \big) d \mu_t.
\]
\end{lemma}

\begin{proof}
Although this estimate is classical, see \cite{Yau} and Chapter 6 of \cite{KipLan}, we were not able to find a reference with exactly the version stated here for finite-state Markov chains; compare with Lemma 6.1.4 of \cite{KipLan}. Part of the proof will be used to prove Lemma \ref{lA2} below. In the case of diffusions, this estimate becomes an identity; see Lemma 1 in \cite{Yau}.

Let $L^\ast$ be the adjoint of $L$ with respect to the reference measure $\nu$. The forward Fokker-Planck equation tells us that 
\[
\tfrac{d}{dt} \big( f_t \psi_t \big) = L^\ast \big( f_t \psi_t\big)
\]
for any $t \in [0,T]$, from where
\[
\tfrac{d}{dt} f_t = \frac{1}{\psi_t} \Big( L^\ast \big( f_t \psi_t \big) - f_t \tfrac{d}{dt} \psi_t \Big).
\]
Therefore, rewriting $H(t)$ as $H(t) = \int f_t \log f_t \psi_t d \nu$, we see that
\[
\begin{split}
H'(t)
		&= \int (1+\log f_t) \big( L^\ast \big(f_t \psi_t\big) - f_t \tfrac{d}{dt} \psi_t \big) d \nu\\
		&\quad \quad + \int f_t \log f_t \tfrac{d}{dt} \psi_t d \nu\\
		&=\int f_t L \log f_t d \mu_t - \int f_t \tfrac{d}{dt} \log \psi_t d \mu_t.
\end{split}
\]
Using the inequality $a(\log b - \log a) \leq 2 \sqrt{a} ( \sqrt{b} -\sqrt{a})$, we obtain the estimate
\[
\begin{split}
f_t(x) L \log f_t(x)
		&= \sum_{y \in \Omega} r(x,y) f_t(x) \big( \log f_t(y) - \log f_t(x) \big) \\
		&\leq \sum_{y \in \Omega} 2r(x,y) \sqrtt{f_t(x)}\big( \sqrtt{f_t(y)} -\sqrtt{f_t(x)} \big) = 2 \sqrtt{f_t(x)} L \sqrtt{f_t}(x),
\end{split}
\]
valid for any $x \in \Omega$. Using now the identity $2 \sqrt{a} (\sqrt{b} -\sqrt{a} ) = - (\sqrt{b}-\sqrt{a})^2 + b-a$, we see that
\[
2r(x,y) \sqrtt{\!f_t(x)}\big( \sqrtt{\!f_t(y)} -\sqrtt{\!f_t(x)} \big) = -r(x,y) \big( \sqrtt{\!f_t(y)} -\sqrtt{\!f_t(x)} \big)^2 + r(x,y) \big( f_t(y) -f_t(x)\big).
\]
\[
\begin{split}
2r(x,y) \sqrtt{f_t(x)}\big( \sqrtt{f_t(y)} -\sqrtt{f_t(x)} \big) 
		&= -r(x,y) \big( \sqrtt{f_t(y)} -\sqrtt{f_t(x)} \big)^2\\
		&\hspace{20pt}+ r(x,y) \big( f_t(y) -f_t(x)\big).
\end{split}
\]
Therefore, $2 \sqrtt{f_t} L \sqrtt{f_t} = - \Gamma \sqrtt{f_t} + L f_t$. We conclude that
\[
H'(t) \leq - \int \Gamma \sqrtt{f_t} d \mu_t + \int \Big( L f_t - f_t \tfrac{d}{dt} \log \psi_t\Big) d \mu_t.
\]
Since $\int L f_t d \mu_t = \int L_t^\ast \mathds{1} f_t d \mu_t$, the lemma is proved.
\end{proof}

\begin{remark}
The equation $L_t^\ast \mathds{1} - \tfrac{d}{dt} \log \psi_t =0$ is exactly the forward Fokker-Planck equation for $\psi_t$, and therefore the expression $L_t^\ast \mathds{1} - \tfrac{d}{dt} \log \psi_t$ can be interpreted as a measure of how close the family $\{\mu_t; t \geq 0\}$ is to be the marginal laws of the process $\{x_t; t \in [0,T]\}$ with initial measure $\mu_0$.
\end{remark}

\subsection{Variational estimates for exponential moments of observables of Markov chains}
\label{sA.2}
Let $V: [0,T] \times \Omega \to \bb R$ be a bounded function. Our aim is to derive a variational estimate for the exponential expectation
\[
\bb E \Big[ \exp\Big\{ \int_0^T V_t(x_t) dt \Big\} \Big].
\]
The integral $\int_0^T V_t(x_t) dt$ is what we call an {\em observable} of the Markov chain $\{x_t; t \geq 0\}$.
We will need an additional condition: we will assume that $x_0$ has law $\mu_0$; this is equivalent to  assume that $f_0=1$. We start recalling Feynman-Kac's formula: we have that
\[
\bb E^{\mu_0} \Big[ \exp\Big\{ \int_0^T V_t(x_t) dt \Big\} \Big] = \int g_0 d \mu_0,
\]
where $g: [0,T] \times \Omega \to \bb R$ is the solution of the final-value problem
\[
\left\{
\begin{array}{r@{\;=\;}l}
\tfrac{d}{dt} g_t + L g_t & - V_t g_t; t \in [0,T]\\
g_T & 1.
\end{array}
\right.
\]
Notice that Feynman-Kac's formula holds for general initial conditions; however the method below requires that we start the process from $\mu_0$. 
For $t \in [0,T]$ define $\phi(t) = \int g_t^2 d \mu_t$. Then,
\[
\phi'(t) = \int 2 g_t \big( -V_t g_t -L g_t \big) d \mu_t + \int g_t^2 \tfrac{d}{dt} \log \psi_t d \mu_t.
\]
Using the identity $2a(b-a) = -(a-b)^2 + b^2 -a^2$ we see that $-2 g_t L g_t = \Gamma g_t  - L g_t^2$ and therefore
\[
\begin{split}
\phi'(t) 
	&= \int \Gamma g_t d \mu_t - \int \big( 2 V_t g_t^2 + Lg_t^2 - g_t^2 \tfrac{d}{dt} \log \psi_t \big) d \mu_t \\
	&= \int \Gamma g_t d \mu_t - \int \big( 2 V_t + L^\ast_t \mathds{1} - \tfrac{d}{dt} \log \psi_t \big) g_t^2 d \mu_t.
\end{split}
\]
Notice that if $g(x) \geq 0$ for any $x \in \Omega$ and $g(x^\ast)=0$, then $Lg(x^\ast) + V_t(x^\ast) g(x^\ast) \leq 0$. Therefore, by the maximum principle, $g_t(x) \geq 0$ for any $t \in [0,T]$ and any $x \in \Omega$, since $g_T \geq 0$.
Let us define $\lambda: [0,T] \to \bb R$ as
\[
\lambda_t := \sup_{f} \Big\{ - \int \Gamma \sqrtt{f} d \mu_t + \int \Big(2V_t + L_t^\ast \mathds{1} - \tfrac{d}{dt} \log \psi_t\Big) f d \mu_t\Big\}
\]
for any $t \in [0,T]$, where the supremum runs over all densities $f$ with respect to $\mu_t$. Taking as a test function in this supremum $f = g_t^2 (\int g_t^2 d\mu_t)^{-1}$, we see that
\[
\phi'(t) \geq - \lambda_t \phi(t) \text{ for any } t \in [0,T].
\]
Using an integrating factor, we can integrate this estimate between $t=0$ and $t=T$ to conclude that
\[
\phi(T) \exp\Big\{ \int_0^T \lambda_t dt \Big\} \geq \phi(0).
\]
Since $\phi(T) =1$, we have proved that $\phi(0) \leq \exp\big\{\int_0^T \lambda_t dt \big\}$. In particular, 
\[
\int g_0 d \mu_0 \leq \Big(\int g_0^2 d\mu_0 \Big)^{1/2} \leq \exp \Big\{\frac{1}{2} \int_0^T \lambda_t dt\Big\}.
\]
Therefore, we have proved the following result:
\begin{lemma}
\label{lA2}
For any $V: [0,T] \times \Omega \to \bb R$,
\[
\begin{split}
\log \bb E^{\mu_0} \Big[ \exp\Big\{ \int_0^T V_t(x_t) dt \Big\} \Big] 
		&\leq \int_0^T \sup_f \Big\{-\int\Gamma \sqrtt{f} d \mu_t + \int V f d \mu_t\\
		 &\quad \quad \quad\quad \quad + \frac{1}{2} \int \big( L_t^\ast \mathds{1} - \tfrac{d}{dt} \log \psi_t \big) d \mu_t \Big\} dt,
\end{split}
\]
where the supremum runs over all densities $f$ with respect to $\mu_t$.
\end{lemma}

\begin{remark}
Notice that the correction term appearing in the last integral of this variational formula is exactly the same function $J_t = L_t^\ast \mathds{1} -\tfrac{d}{dt} \log \psi_t$ that appears in Yau's entropy inequality. In applications, it is difficult to obtain more information about the density $f_t$, and it is necessary to derive an efficient way to estimate integrals of the form $\int J_t f d \mu_t$ in terms of $\int \Gamma \sqrtt{f} d \mu_t$ and $\int f \log f d \mu_t$ for arbitrary densities $f$. In counterpart, one has the freedom to choose $\mu_t$ in any convenient way, and $J_t$ can be explicitly computed in terms of $\mu_t$ and $L$. This type of estimates is exactly what one needs in order to use Lemma \ref{lA2} in an efficient way.  
\end{remark}

\begin{remark}
Although estimates similar to Lemma \ref{lA2} are common in the literature, see Lemma A.1.7.2 in \cite{KipLan} for example, it seems that the exact formula presented here is new in the literature. In particular, the sort of duality between Lemmas \ref{lA1} and \ref{lA2} that we use in this article seems to be new, and could be of interest in other situations.
\end{remark}

\subsection{Estimates for the exclusion process}
\label{sA.3}
In this article, we will apply the estimates of the previous subsections to the process $\eta^n(\cdot)$, using as reference measures the measures $\{\mu_t^n; t \geq 0\}$ defined in \eqref{pichidangui}. We need to compute $J_t^n = L_{n,t}^\ast \mathds{1} - \tfrac{d}{dt} \log \psi_t^n$, where $L_{n,t}^\ast$ is the adjoint of $L_n$ with respect to $\mu_t^n$ and $\psi_t^n$ is the Radon-Nikodym derivative of $\mu_t^n$ with respect to $\nu_{1/2}^n$. In order to avoid overcharged notation, we will write $u_x$ instead of $u\big(t,\tfrac{x}{n} \big)$ and we will assume that $n \geq 2 \|F\|_\infty$. Using \eqref{ecA1} we see that
\[
\begin{split}
L_{n,t}^\ast \mathds{1} 
		&= n^2 \sum_{\subind} \Big\{ r_n(x,x+b)\Big( \eta_{x+b}(1-\eta_x) \frac{\mu_t^n(\eta^{x,x+b})}{\mu_t^n(\eta)} - \eta_x(1-\eta_{x+b})\Big)\\
		&\quad \quad \quad + r_n(x+b,x) \Big( \eta_x(1-\eta_{x+b}) \frac{\mu_t^n(\eta^{x,x+b})}{\mu_t^n(\eta)} -\eta_{x+b}(1-\eta_x)\Big) \Big\}\\
		&= n^2 \sum_{\subind} \Big\{ r_n(x,x+b)\Big( \eta_{x+b}(1-\eta_x) \frac{u_x(1-u_{x+b})}{u_{x+b}(1-u_x)} - \eta_x(1-\eta_{x+b})\Big)\\
		&\quad \quad \quad + r_n(x+b,x) \Big( \eta_x(1-\eta_{x+b}) \frac{u_{x+b}(1-u_x)}{u_x(1-u_{x+b})} -\eta_{x+b}(1-\eta_x)\Big) \Big\}.
\end{split}
\]
This last expression can be factorized to obtain the identity
\begin{equation}
\label{ecA3.1}
\begin{split}
L_{n,t}^\ast \mathds{1}
		&= n^2 \sum_{\subind} \Big( r_n(x,x+b)u_x(1-u_{x+b}) -  r_n(x+b,x)u_{x+b}(1-u_x)\Big) \times\\ 
		&\quad \quad  \quad \quad \quad\times\Big(  \frac{\eta_{x+b}(1-\eta_x)}{u_{x+b}(1-u_x)} -  \frac{\eta_x(1-\eta_{x+b})}{u_x(1-u_{x+b})} \Big).
\end{split}
\end{equation}
Recall the definition $\omega_x = \frac{\eta_x - u_x}{u_x(1-u_x)}$. Notice that any function of $\eta_x$ and $\eta_y$ can be written as a linear combination of $1$, $\omega_x$, $\omega_y$ and $\omega_x\omega_y$. Therefore,
\[
\frac{\eta_{x+b}(1-\eta_x)}{u_{x+b}(1-u_x)} -  \frac{\eta_x(1-\eta_{x+b})}{u_x(1-u_{x+b})}
		= a+ b\omega_x + c \omega_{x+b} + d \omega_x \omega_{x+b}
\] 
for some real constants $a,b,c,d$ to be determined. Taking the expectation of this identity with respect to $\mu_t^n$, we see that $a=0$. Evaluating this identity at $\eta_x=1$ and $\eta_{x+b}=u_{x+b}$\footnote{This is equivalent to take expectations with respect to $\Bern(1)\otimes \Bern(u_x)$.} we see that $b=-1$. Taking $\eta_x=u_x$, $\eta_y=1$, we obtain $c=1$. Evaluating at $\eta_x=\eta_y=1$, we obtain the relation
\[
\frac{d}{u_x u_{x+b}} - \frac{1}{u_x} + \frac{1}{u_{x+b}} = 0,
\]
from where $d=u_{x+b}-u_x$. Therefore,
\begin{equation}
\label{ecA3.2}
\frac{\eta_{x+b}(1-\eta_x)}{u_{x+b}(1-u_x)} -  \frac{\eta_x(1-\eta_{x+b})}{u_x(1-u_{x+b})}
		= \omega_{x+b} -\omega_x +(u_{x+b}-u_x) \omega_x \omega_{x+b}.
\end{equation}
In the other hand,
\begin{multline}
\label{ecA3.3}
 r_n(x,x+b)u_x(1-u_{x+b}) -  r_n(x+b,x)u_{x+b}(1-u_x)=\\
 		= u_x -u_{x+b} + \frac{1}{n} F_b^n(x) \big( u_x + u_{x+b} -2 u_x u_{x+b}\big).
\end{multline}
Let us introduce the discrete operator $\mc L^n$ by defining
\begin{equation}
\label{conguillio}
\begin{split}
\mc L^n f_x 
		&:= \sum_{b \in \mc B} n^2(f_{x+b}+f_{x-b}-2f_x)\\
		&\quad \quad -\sum_{b \in \mc B} n\Big( F_b^n(x) \big(f_x+f_{x+b}-2f_x f_{x+b}\big) -\\
		&\quad \quad \quad \quad - F_b^n(x-b)\big(f_{x-b}+f_x -2 f_{x-b} f_x \big)\Big).
\end{split}
\end{equation}
for any $f: \tdn \to \bb R$ and any $x \in \tdn$.
Putting \eqref{ecA3.2} and \eqref{ecA3.3} back into \eqref{ecA3.1}, after a summation by parts we obtain that
\[
L_{n,t}^\ast \mathds{1} 
		= \sum_{x \in \tdn} \omega_x \mc L^n u_x + \sum_{\subind} \omega_x \omega_{x+b} G_{x,b}^n,
\]
where 
\begin{equation}
\label{funG}
G_{x,b}^n:= n(u_{x+b}-u_x)F_b^n(x)(u_x+u_{x+b} -2u_x u_{x+b})
		-n^2(u_{x+b}-u_x)^2.
\end{equation}
Since $\mu_t^n$ is a product measure, we have that
\[
\begin{split}
\tfrac{d}{dt} \log \psi_t^n 
		&= \tfrac{d}{dt} \sum_{x \in \tdn} \big( \eta_x \log 2 u_x + (1-\eta_x) \log 2(1-u_x)\big)\\
		&= \sum_{x \in \tdn} \Big( \frac{\eta_x}{u_x} -\frac{1-\eta_x}{1-u_x}\Big) \frac{d}{dt} u_x. 
\end{split}
\]
Notice that $\frac{\eta_x}{u_x} -\frac{1-\eta_x}{1-u_x} = \omega_x$. Therefore, we conclude that
\begin{equation}
\label{ecA3.4}
L_{n,t}^\ast \mathds{1} -\tfrac{d}{dt} \log \psi_t^n = \sum_{x \in \tdn} \omega_x \big(\mc L^n -\tfrac{d}{dt} \big) u_x + \sum_{\subind} \omega_x \omega_{x+b} G_{x,b}^n.
\end{equation}
Recall that $u_x$ satisfies \eqref{echid}. Notice that $\mc L^n$ is a discrete approximation of order $\mc O(n^{-2})$ of $\mc L u:= \Delta u - 2 \nabla \cdot(u(1-u)F)$, that is, for any $f$ of class $\mc C^4$,$\sup_{x}|(\mc L^n -\mc L)f_x| \leq Cn^{-2}$ for some constant $C = C(u)$. Since we are assuming that $u_0$ and $F$ are of class $\mc C^\infty$, $u$ is also of class $\mc C^\infty$, see Proposition \ref{pC1}. We conclude that the first sum in \eqref{ecA3.4} is of the form 
\[
\frac{1}{n^2} \sum_{x \in \tdn} \omega_x R_x^n(t)
\]
for some error term $R_x^n(t)$ satisfying $|R_x^n(t)| \leq \|u_t\|_{\mc C^4}$.

\section{On the regularity of solutions of parabolic equations}
\label{sB}
In this section we collect some classical results about linear and quasilinear parabolic equations needed along the article. 

Let $u_0: \bb T^d \to [0,1]$ be of class $\mc C^\infty$ and let $\{u(t,x); t \geq 0, x \in \bb T^d\}$ the solution of the hydrodynamic equation \eqref{echid} with initial condition $u_0$. The following result is a direct application of Theorem V.6.1 of \cite{LadSolUra}.

\begin{proposition}
\label{pC1}
If the vector field $F$ and the initial condition $u_0$ are of class $\mc C^\infty$, then the solution $\{u(t,x); t \geq 0, x \in \bb T^d\}$ of \eqref{echid} is of class $\mc C^\infty$.
\end{proposition}

Knowing that the solution of \eqref{echid} is smooth, we can discuss about the regularity of solutions of the backwards Fokker-Planck equation defined in \eqref{semigrupo}. The following result follows from Theorem IV.5.1 in \cite{LadSolUra}:

\begin{proposition}
\label{pC2}
Let $f$ be of class $\mc C^\infty$. Then the solution $\{P_{s,t} f; 0\leq s \leq t\}$ of the backwards Fokker-Planck equation \eqref{semigrupo} is of class $\mc C^\infty$. Moreover, for any $T >0$ $\ell \in [0,\infty)$ and $\delta>0$ there exists $C=C(\ell,\delta,T,F,u_0)$ such that
\[
\big\|P_{s,t} f\big\|_{\mc C^\ell} \leq C \|f\|_{\mc C^{\ell+\delta}}
\]
for any $f \in \mc C^\infty(\bb T^d)$ and any $0 \leq s \leq t \leq T$. If $\ell \notin \bb N$, then we can take $\delta =0$.
\end{proposition}

The following lemma says that solutions of the hydrodynamic equation \eqref{echid} do not touch $0$ and $1$:

\begin{lemma}
\label{lC3}
Assume that there exists $\eps_0>0$ such that $\eps_0 \leq u_0(x) \leq 1-\eps_0$ for any $x \in \bb T^d$. For any $T>0$ there exists $\eps_1 = \eps_1(T,F)$ such that $\eps_1 \leq u(t,x) \leq 1- \eps_1$ for any $x \in \bb T^d$ and any $t \in [0,T]$.
\end{lemma}

\begin{proof}
For each $\theta \in (0,1)$, define $v^+:[0,T] \times \bb T^d \to [0,1]$ as
\[
v^+(t,x) := \frac{\theta}{\theta +(1-\theta)e^{-c t}}
\]
for any $(t,x) \in [0,T] \times \bb T^d$.
Then,
\[
\partial_t v^+(t,x) = -\frac{c\theta(1-\theta)e^{-ct}}{\big(\theta + (1-\theta)e^{-ct}\big)^2},
\]
\[
\big(\Delta v^+ - 2\nabla (v^+(1-v^+) F) \big)(t,x) = \frac{2\theta(1-\theta)e^{-ct}}{\big(\theta + (1-\theta)e^{-ct} \big)^2} \nabla \cdot F
\]
and $v^+$ is a supersolution of \eqref{echid} if $c \geq 2 \|\nabla \cdot F\|_\infty$.  By the maximum principle, taking $\theta = 1 -\eps_0$ we conclude that
\[
u(t,x) \leq \frac{\theta}{\theta +(1-\theta)e^{-c t}}
\]
for any $x \in \bb T^d$ and any $t \geq 0$. Similarly,
\[
v^-(t,x) : = \frac{\theta}{\theta +(1-\theta)e^{c t}}
\]
is a subsolution of \eqref{echid}. Taking $\theta =1-\eps_0$ we prove the lemma for 
\[
\eps_1 = \frac{\eps_0}{\eps_0+(1-\eps_0)e^{2 \|\nabla \cdot F \|_\infty T}}.
\]
\end{proof}

\section{Functional spaces and topology}
\label{sD}
In this section we define what we understand by solutions of \eqref{SHE}. In order to do that, we need to define various functional spaces on which trajectories of distribution-valued stochastic processes live. We also take the opportunity to collect all results needed in this article related to the topology of these spaces.

\subsection{Sobolev spaces}
\label{sD1}
For each $m \in \bb Z^d$, let $\phi_m: \bb T^d \to \bb C$ be defined as $\phi(x) = e^{2\pi i x \cdot m}$ for any $x \in \bb T^d$. The family of functions $\{\phi_m; m \in \bb Z^d\}$ is an orthonormal basis of $L^2(\bb T^d)$. For $f: \bb T^d \to \bb R$ bounded, let $\widehat{f}: \bb Z^d \to \bb C$ be given by 
\[
\widehat{f}(m) := \int  f(x) \overline{\phi(x)} dx, 
\]
that is, $\widehat{f}(m)$ is the Fourier coefficient of $f$ of order $m$. For $f \in \mc C^\infty(\bb T^d)$ and $k \in \bb R$, define
\begin{equation}
\label{ecA1.1}
\|f\|_k : = \Big( \sum_{m \in \bb Z^d} \big| \widehat{f}(m) \big|^2 (1+|m|^2)^{k} \Big)^{1/2},
\end{equation}
where $|m|:= (m_1^2+\dots+m_d^2)^{1/2}$. Notice that $\|f\|_k$ is finite for any $k \in \bb R$, since $f$ is infinitely differentiable. The formula \eqref{ecA1.1} defines a norm in $\mc C^\infty(\bb T^d)$. The Sobolev space$ H_{k}(\bb T^d)$ is defined as the closure of $\mc C^\infty(\bb T^d)$ with respect to the norm $\|\cdot\|_k$. By Parseval's identity, $H_0(\bb T^d) = L^2(\bb T^d)$. Notice that if $\ell \leq k$, then $\|f\|_\ell \leq \|f\|_k$ and therefore $H_\ell(\bb T^d) \subseteq H_k (\bb T^d)$. 

The spaces $H_{k}(\bb T^d)$ and $H_{-k}(\bb T^d)$ are dual in the following sense. Let $\<\cdot,\cdot\>$ be the inner product in $L^2(\bb T^d)$. By Plancherel's theorem, the restriction of $\<\cdot,\cdot\>$ to $\mc C^\infty(\bb T^d) \times \mc C^\infty(\bb T^d)$ can be continuously extended to a bilinear form in $H_{-k}(\bb T^d) \times H_k(\bb T^d)$. This allows us to identify $H_{-k}(\bb T^d)$ with the space of linear functionals defined in $\mc C^\infty(\bb T^d)$, continuous with respect to the norm $\|\cdot \|_{-k}$. This fact will allow us to define random variables with values in $H_{-k}(\bb T^d)$ via duality.

The following characterization of compact sets in $H_{k}(\bb T^d)$ will be useful:

\begin{proposition}
\label{pD1.1}
A set $K \subseteq H_{k}(\bb T^d)$ is relatively compact if there exists a sequence $\{\lambda_m; m \in \bb Z^d\}$ of positive integers such that
\begin{itemize}
\item[i)] 
$
\sum_{m \in \bb Z^d} \lambda_m < +\infty,
$
\item[ii)] $\big|\widehat{f}(m)\big|^2 (1+|m|^2)^k \leq \lambda_m$ for any $f \in K$ and any $m \in \bb Z^d$.
\end{itemize}
\end{proposition}

\subsection{Holder spaces}
\label{sD2}

Let $E_1,E_2$ be two Banach spaces and let $\alpha \in (0,1)$. We say that $f: E_1 \to E_2$ is (globally) H\"older-continuous of order $\alpha$ if there exists finite a constant $K_\alpha$ such that
\[
\|f(y)- f(x) \| \leq K_\alpha \|y-x\|^\alpha
\]
for any $x,y \in E_1$. In that case we say that $f \in \mc C^\alpha$. The cases considered in this article will be $E_1 = \bb T^d$, $E_2 = \bb R$ and $E_1 = [0,T]$, $T >0$, $E_2 = H_{-k}(\bb T^d)$, $k >0$. Since in these cases $E_1$ is a manifold, the definition of the space $\mc C^\alpha$ can be generalized to $\alpha \geq 0$. Let $I_\ell := \{1,\dots,d\}^\ell$. If $f$ is $\ell$-times continuously differentiable, we say that $f$ is of class $\mc C^\ell$. For ${\bf i} \in I_\ell$ and $f$ of class $\mc C^\ell$, define $\partial^\ell_{\bf i} f = \partial_{i_1} \dots \partial_{i_\ell} f$. For $\ell \in \bb N_0$, let us define
\[
\|f\|_{\mc C^\ell} := \sum_{k=0}^\ell \sum_{{\bf i} \in I_k} \|\partial^k_{\bf i} f\|_\infty.
\]
Here $\|\partial_{\bf i}^k f\|_\infty = \sup_{x \in E_1} \| \partial_{\bf i}^k f(x)\|$, where the second norm is the norm in $E_2$. For $\alpha \in (\ell,\ell+1)$, $\ell \in \bb N$ we say that $f \in \mc C^\alpha$ if $f \in \mc C^\ell$ and each derivative $\partial^\ell_{\bf i} f$ belongs to $\mc C^{\alpha-\ell}$. In that case we define
\[
\|f\|_{\mc C^\alpha} = \|f\|_{\mc C^\ell} + \sum_{{\bf i} \in I_\ell} \sup_{x \neq y} \frac{\| \partial_{\bf i}^\ell f(y) - \partial_{\bf i}^\ell f(x)\|}{\|y-x\|^{\alpha-\ell}}.
\]
In the case $E_1=[0,T]$, we can also extend these definitions to $\alpha \in (-1,0)$: let $f,g: [0,T] \to H_{-k}(\bb T^d)$ be given. We say that $f = \tfrac{d}{dt} g$ if
\[
\int_0^T \<f_t, h_t \> dt = - \int_0^T \< g_t, \tfrac{d}{dt} h\> dt
\]
for any $h: [0,T] \to H_{k}(\bb T^d)$ of class $\mc C^1$. Then we say that $f \in \mc C^\alpha$, $\alpha \in (-1,0)$ if $f = \tfrac{d}{dt} g$ for some $g \in \mc C^{\alpha+1}$. Then we define $\|f\|_{\mc C^\alpha} = \|g - g_0\|_{\mc C^{\alpha+1}}$. This is a particular case of what it known in the literature as {\em Besov spaces}, which are nowadays very popular in the context of nonlinear stochastic partial differential equations.

\subsection{The space \texorpdfstring{$\mc D$}{D} and tightness}
\label{sD3}

Let $(E,d)$ be a complete and separable metric space and let $T>0$. We denote by $\mc D([0,T];E)$ the space of \cadlag trajectories from $[0,T]$ to $E$. We equip $\mc D([0,T]; E)$ with the $J_1$-Skorohod topology, see Chapter VI.1 of \cite{JacShi}. We say that a family $\{X_t;t \in [0,T]\}$ of random variables with values in $E$ is a {\em stochastic process} if in addition the trajectories $t \mapsto X_t$ belong to $\mc D([0,T]; E)$. In that case, the law of the process $\{X_t; t \geq 0\}$ is a measure $Q$ in $\mc D([0,T]; E)$. We say that a sequence of stochastic processes $\{Y_t^n; t \in [0,T]\}_{n \in \bb N}$ with values in $E$ is {\em tight} if the sequence $\{Q^n; t \in [0,T]\}_{n \in \bb N}$ of laws of  $\{Y_t^n; t \in [0,T]\}_{n \in \bb N}$ is relatively compact with respect to the weak topology. 

The following proposition, known as {\em Aldous' criterion}, gives a way to prove tightness of a sequence of stochastic processes  $\{Y_t^n; t \in [0,T]\}_{n \in \bb N}$.

\begin{proposition}[Aldous' criterion]
\label{pD3.1}
The sequence $\{Y_t^n; t \inT\}_{n \in \bb N}$ is tight with respect to the $J_1$-Skorohod topology in $\mc D([0,T];E)$ if
\begin{itemize}
\item[i)] for any $\eps >0$ and any $t \in [0,T]$ there exists a compact set $K = K(\eps,t) \subseteq E$ such that 
\[
\limsup_{n \to \infty} \bb P_n\big(Y_t^n \notin K\big) \leq \eps,
\]
\item[ii)] for any $\eps >0$, 
\[
\lim_{\delta \to 0} \limsup_{n \to \infty} \sup_{\gamma \leq \delta} \sup_{\tau \in \mc T_T} \bb P_n \big( d\big(Y_{\tau+\gamma}^n , Y_\tau^n(x)\big) > \eps \big) =0,
\]
\end{itemize}
where $\mc T_T$ is the set of stopping times in $[0,T]$.\footnote{Here we use the convention $X_{\tau+\gamma}^n = Y_t^n$ if $\tau+\gamma >T$.} 
\end{proposition}

Let $\mc C([0,T]; E)$ be the space of continuous trajectories from $[0,T]$ to $E$. We say that a sequence of stochastic processes $\{Y_t^n; t \inT \}_{n \in \bb N}$ with values in $E$ is {\em $\mc C$-tight} if it is tight and in addition every limit point of $\{Y_t^n; t \inT\}_{n \in \bb N}$ has trajectories in $\mc C([0,T]; E)$ with probability 1. A simple criterion for $\mc C$-tightness is the following:

\begin{proposition}
\label{pD3.2}
The sequence $\{Y_t^n; t \inT\}_{n \in \bb N}$ of stochastic processes with values in $E$ is $\mc C$-tight if
\begin{itemize}
\item[i)] $\{Y_t^n; t \inT\}_{n \in \bb N}$ is tight,

\item[ii)] defining $\Delta_T^n:= \sup_{0 \leq t \leq T} d(Y_{t-}^n, Y_t^n)$, then
\[
\limsup_{n \to \infty} \bb P_n (\Delta^n_T \geq \eps) =0
\]
for any $\eps >0$. 
\end{itemize}
\end{proposition}

If the sequence $\{Y_t^n; t \inT\}_{n \in \bb N}$ has trajectories in $\mc C^\alpha = \mc C^\alpha([0,T]; E)$ for some $\alpha \geq 0$, then the following tightness criterion, known in the literature as {\em Kolmogorov-Centsov criterion}, is very effective:

\begin{proposition}
\label{pD3.3}
Assume that there exist constants $C, a, b >0$ such that
\begin{equation}
\label{KolCen}
\bb P_n\big( d\big(Y_s^n, Y_t^n\big) > \lambda \big) \leq \frac{C |t-s|^{1+b}}{\lambda^a}
\end{equation}
for any $0 \leq s \leq t \leq T$, any $n \in \bb N$ and any $\lambda >0$. Assume as well that for any $\eps >0$ there exists compact set $K_\eps \subseteq E$ such that 
\[
\bb P_n( Y_0^n \notin K_\eps) \leq \eps \text{ for any } n \in \bb N.
\]
Then,  $\{Y_t^n; t \inT\}_{n \in \bb N}$ is tight in $\mc C^\beta$ for any $\beta < \frac{b}{a}$.
\end{proposition}

\begin{remark}
From \eqref{KolCen} it is possible to conclude that the trajectories of  $\{Y_t^n; t \inT \}$ are in $\mc C^\beta([0,T]; E)$ with probability 1.
\end{remark}

\subsection{The white noise}
\label{sD.4}
A random variable $\xi$ with values in $H_{-k}(\bb T^d)$ for some $k>0$ is called a  standard {\em white noise} if for any $\ell \in \bb N$ and any $f_1,\dots,f_\ell \in \mc C^\infty(\bb T^d)$, the vector $(\xi(f_1),\dots,\xi(f_\ell))$ is a Gaussian vector of mean 0 and covariances
\[
\bb E\big[ \xi(f_i) \xi(f_j) \big] = \<f_i,f_j\>.
\]
Let us assume that there exists a white noise $\xi$ with values in $H_{-k}(\bb T^d)$ for some $k>0$, defined in some probability space $(\mc X,\bb P, \mc F)$. Let $f \in L^2(\bb T^d)$ and let $\{f^n; n \in \bb N\}$ be a sequence of functions in $\mc C^\infty(\bb T^d)$, convergent to $f$ in $L^2(\bb T^d)$. Then, the real-valued random variables $\{\xi(f^n); n \in \bb N\}$ converge in $L^2(\bb T^d)$ to a random variable that we call $\xi(f)$. Notice that $\xi(f)$ is well defined up to a set of zero measure {\em that may depend on $f$}. 

If $\{g^n; n \in \bb N\}$ is an orthonormal basis of $L^2(\bb T^d)$, then the random variables $\{\xi(g^n); n \in \bb N\}$ are i.i.d.~with common law $\mc N(0,1)$. In this case, $\xi(f)$ admits the representation
\[
\xi(f) = \sum_{n \in \bb N} \xi(g^n) \<f,g^n\>.
\]
By the two-series theorem, this series converges if and only if $\sum_{n \in \bb N} \<f,g^n\>^2 <+\infty$. On the other hand, by Riesz's representation theorem, since $\sum_{n \in \bb N} \xi(g^n)^2 =+\infty$ with probability one, $\xi$ can not be represented as a random variable with values in $L^2(\bb T^d)$. This fact is the main reason for the necessity of the introduction of the Sobolev spaces $H_{-k}(\bb T^d)$. At least formally, this discussion leads to the representation
\begin{equation}
\label{ecD4.1}
\xi = \sum_{n \in \bb N} \xi^n g^n,
\end{equation}
where $\{\xi^n; n \in \bb N\}$ is i.d.d.~with common law $\mc N(0,1)$. It will be convenient to use the basis $\{\phi_m; m \in \bb Z^d\}$ in order to construct $\xi$. However, since the functions $\phi_m$ are complex valued, \eqref{ecD4.1} can not be used directly. Let $\{\xi_m^{i,0}; m \in \bb Z^d, i=1,2\}$ be an i.i.d.~sequence of random variables with common law $\mc N(0,1)$. Let us define $\{\xi_m^i; m \in \bb Z^d, i= 1,2\}$ as
\[
\xi_m^1 = \frac{\xi_m^{1,0} +\xi_{-m}^{1,0}}{\sqrt 2}; \quad \xi_m^2 = \frac{\xi_m^{2,0} -\xi_{-m}^{2,0}}{\sqrt 2}.
\]
The sequence $\{\xi^i_m; m \in \bb Z^d, i=1,2\}$ is also i.i.d.~with common law $\mc N(0,1)$, except for the relations $\xi_m^1 = \xi_{-m}^1$, $\xi_m^2= - \xi_{-m}^2$ for any $m \in \bb Z^d$. Then, at least formally the random variable
\begin{equation}
\label{ecD4.3}
\xi := \sum_{m \in \bb Z^d} \Big( \frac{\xi_m^1 + i \xi_{-m}^2}{\sqrt 2} \Big) \phi_m
\end{equation}
is a white noise. The point of this formula is that $\xi(f)$ is real for real-valued functions $f$. For any $k >0$,
\begin{equation}
\label{ecD4.2}
\|\xi\|_{-k}^2 = \sum_{m \in \bb Z^d}\bigg(\frac{\big(\xi^1_m\big)^2 + \big( \xi_m^2\big)^2}{2}\bigg) (1+|m|^2)^{-k}.
\end{equation}
Thereofre, $\xi$ is a random variable with values in $H_{-k}(\bb T^d)$ if and only if this sum is convergent with probability 1. By the three-series theorem, this is the case if and only if
\[
\sum_{m \in \bb Z^d} (1+|m|^2)^{-k} <+\infty.
\]
We summarize this discussion in the following proposition:
\begin{proposition}
\label{pD4.1}
A white noise $\xi$ belongs to $H_{-k}(\bb T^d)$ if and only if $k> d/2$.
\end{proposition}

Recall the definition of $X_t^n$ given in \eqref{losandes}. We see that $X_t^n$ is a linear combination of Dirac $\delta$ functions. Let $\delta_x$ be the Dirac $\delta$ function centered at $x \in \bb T^d$. Then,
\[
\widehat{\delta_x}(m) = e^{-2\pi i x \cdot m}
\]
and 
\[
\|\delta_x\|_{-k}^2 = \sum_{m \in \bb Z^d} (1+|m|^2)^{-k}.
\]
In particular $\delta_x \in H_{-k}(\bb T^d)$ if and only if $k >d/2$ and  the process $\{X_t^n; t \in [0,T]\}$ has trajectories in $\mc D([0,T]; H_{-k}(\bb T^d))$ for any $ k >d/2$.

A space-time white noise is a random variable $\dot{\mc W}$ such that $\dot{ \mc W} (f)$ has a Gaussian law of mean zero and variance
\[
\int_0^\infty \int f(t,x)^2 dx dt
\]
for any $f: [0,\infty) \times \bb T^d \to \bb R$ of compact support and of class $\mc C^\infty$. Although it is possible to construct $\dot{\mc W}$ using \eqref{ecD4.1}, it is more convenient to define $\dot{\mc W}$ as the derivative of another process. We say that a process $\{\mc W_t;t \geq 0\}$ with trajectories in $\mc C([0,T]; H_{-k}(\bb T^d))$ is a {\em cylindrical Wiener process} if for any $f \in \mc C^\infty(\bb T^d)$, the process $\{\mc W_t(f); t \geq 0\}$ is a Brownian motion of variance $\|f\|^2_{L^2(\bb T^d)} t$. The process $\{\mc W_t; t \geq 0\}$ can be constructed using a formula similar to \eqref{ecD4.3}. Let $\{B_m^i(t); t \geq 0, m \in \bb Z^d, i=1,2\}$ be a family of standard, independent Brownian motions, except for the relations $B_m^1(\cdot) = B_{-m}^1(\cdot)$, $B_m^2(\cdot) = - B_{-m}^2(\cdot)$ for any $m \in \bb Z^d$. Then,
\[
\mc W_t := \sum_{m \in \bb Z^d} \frac{1}{\sqrt 2} \big( B_m^1(t) + i B_m^2(t) \big) \phi_m
\] 
is at least formally a cylindrical Wiener process. Using Proposition \ref{pD3.3} we can verify that $\{\mc W_t; t \in [0,T]\}$ has trajectories in $\mc C^\beta([0,T];H_{-k}(\bb T^d))$ with probability 1 for any $T>0$, any $\beta <1/2$ and any $k>d/2$. The process $\dot{\mc W}$ is then defined as the derivative of $\{\mc W_t; t \geq 0\}$ in the It\^o sense.

\subsection{The stochastic heat equation}
\label{sD.6}

In this section we define in a rigorous way what do we understand by a solution of the stochastic heat equation \eqref{SHE}. In what follows we fix $T >0$ and we assume that all processes are defined in a probability space $(\mc X, \bb P, \mc F)$ and they are adapted to a common filtration $\{\mc F_t; t \in [0,T]\}$. Let $\{\mc W_t^i; t \in [0,T], i=1,\dots,d\}$ be a family of independent cylindrical Wiener processes. Recall the definition of the operator $\bb L_t$ given in \eqref{concepcion}. We say that a process $\{X_t; t \in [0,T]\}$ with values in $H_{-k}(\bb T^d)$ for some $k \in \bb R$ is a {\em strong solution} of \eqref{SHE} if for any $f \in \mc C^{\infty}([0,T] \times \bb T^d)$,
\[
X_t(f_t) = X_0(f_0) + \int_0^t X_s\big( (\partial_s + \bb L_s) f_s \big) ds 
		+\sum_{i=1}^d \int_0^t d \mc W_s^i \Big( \sqrtt{u_s(1-u_s)} \frac{\partial f}{\partial x_i}\Big),
\] 
where the integral with respect to $\mc W_s^i$ is taken in the It\^o sense. We say that a process $\{M_t; t \in [0,T]\}$ with values in $H_{-k}(\bb T^d)$ is a martingale if for any $f \in \mc C^\infty(\bb T^d)$ the real-valued process $\{M_t(f); t \in [0,T]\}$ is a martingale. Notice that by duality, the relation
\begin{equation}
\label{ecD5.1}
M_t(f) := \sum_{i=1}^d \int_0^t d \mc W_s^i \Big( \sqrtt{u_s(1-u_s)} \frac{\partial f}{\partial x_i} \Big)
\end{equation}
defines a martingale. Since $\mc W_t^i$ belongs to $H_{-k}(\bb T^d)$ for $k>d/2$, it can be verified that $M_t \in H_{-k}(\bb T^d)$ for $k>1+d/2$. Notice that \eqref{ecD5.1} can also be used for test functions that depend on time. For any $f$, $\{M_t(f); t \in [0,T]\}$ has continuous trajectories and that 
\begin{equation}
\label{portillo}
\<M_t(f)\> = \int_0^t \int 2 u(s,x) (1-u(s,x)) \|\nabla f(x) \|^2 dxds.
\end{equation}
Thanks to L\'evy's characterization theorem, see Theorem II.4.4 of \cite{JacShi}, this relation characterizes the law of $\{M_t(f); t \in [0,T]\}$. Based in this observation, we say that $\{X_t; t \in [0,T]\}$ is a {\em martingale solution} of \eqref{SHE} if for any $f \in \mc C^\infty([0,T] \times \bb T^d)$ the process $\{M_t(f); t \in [0,T]\}$ defined as
\[
M_t(f) = X_t(f_t) -X_0(f_0) - \int_0^t X_s\big( (\partial_s + \bb L_s) f_s \big)ds 
\]
is a continuous martingale of quadratic variation given by \eqref{portillo}. This notion of solution is in principle weaker than the notion of strong solution of \eqref{SHE}, since it does not make explicit reference to the white noise $\dot{\mc W}$. However, since there exists $\eps_1>0$ such that $\eps_1 \leq u(t,x) \leq 1-\eps_1$ for any $(t,x) \in [0,T] \times \bb T^d$, using the martingale representation theorem it is possible to construct $\{\mc W_t; t \in [0,T]\}$ based on $\{M_t; t \in [0,T]\}$, from where both notions are equivalent.

Let $\{P_{s,t}; 0 \leq s \leq t \leq T\}$ be the semigroup defined in \eqref{semigrupo}. We say that $\{X_t; t \in [0,T]\}$ is a {\em mild solution} of \eqref{SHE} if for any $t \in [0,T]$ and any $f \in \mc C^\infty(\bb T^d)$,
\[
X_t(f) = X_0(P_{0,t} f) + \int_0^t d M_s(P_{s,t} f),
\]
where $\{M_t; t \in [0,T]\}$ is the martingale defined in \eqref{ecD5.1}. 

By definition, any martingale solution of \eqref{SHE} is also a mild solution. In this article, we will only need the following result:

\begin{proposition}
\label{pD4.2}
The law of a mild solution of \eqref{SHE} is uniquely determined by the law of $X_0$.
\end{proposition}
\begin{proof}
Let $h \in \mc C^\infty([0,T] \times \bb T^d)$. Since the quadratic variation of $\{M_t(h); t \geq 0\}$ is deterministic, by L\'evy's characterization theorem, for any $0 \leq s < t \leq T$, $M_t(h) - M_s(h)$ is independent of $\mc F_s$ and has a Gaussian law of mean zero and variance
\[
\int_s^t \int 2 u(s',x)(1-u(s',x)) \|\nabla h(s',x)\|^2 dx ds'.
\]
This characterizes the law of $X_t(f)$ as the sum of the independent random variables $X_0(P_{0,t} f)$ and $\int_0^t d M_s(P_{s,t} f)$ and by duality the law of $X_t$. Using the relation
\begin{equation}
\label{pajaritos}
X_t(f) = X_s(P_{s,t} f) + \int_s^t d M_{s'} (P_{s',t} f),
\end{equation}
it also characterizes the joint law of $X_s$ and $X_t$, since the two terms on the right-hand side of \eqref{pajaritos} are independent. Recursively, this procedure characterizes all finite-dimensional laws of $\{X_t; t \in [0,T]\}$ which proves the lemma.
\end{proof}

\section{Some computations involving the generator \texorpdfstring{$L_n$}{Ln}}
\label{sF1}
In this section we compute $(\partial_s+L_n)X_s^n(f)$.
In order to simplify the notation, we consider $f: \tdn \to \bb R$ and we compute
\[
L_n \sum_{x \in \tdn} (\eta_x -u_x) f_x.
\]
In order to simplify the expression for $r_n$, we assume that $n \geq 2\|F\|_\infty$. We have that
\begin{equation}
\label{ecF1} 
\begin{split}
L_n \sum_{x \in \tdn} (\eta_x - u_x) f_x 
		&= n^2 \sum_{\subind} \big( r_n(x,x+b) \eta_x (1-\eta_{x+b}) -\\
		&\hspace{65pt}-r_n(x+b,x) \eta_{x+b} (1-\eta_x)\big)(f_{x+b} -f_x)\\
		&=n^2 \sum_{\subind} (\eta_x-\eta_{x+b})(f_{x+b}-f_x) \\
		&\quad \quad \quad 	+n \sum_{\subind} F_b^n(x) (\eta_x +\eta_{x+b} -2 \eta_x \eta_{x+b} ) (f_{x+b} -f_x).
\end{split}
\end{equation}
The first sum on the right-hand side of this identity is equal to $\sum_{x \in \tdn} \eta_x \Delta_x^n f$, where
\[
\Delta_x^n f := n^2 \sum_{b \in \mc B} (f_{x+b}+f_{x-b}-2f_x).
\]
Performing a summation by parts, we see that $\sum_{x \in \tdn} \eta_x \Delta_x^n f$ is equal to 
\begin{equation}
\label{ecF2}
\sum_{x \in \tdn} (\eta_x - u_x) \Delta_x^n f + \sum_{x \in \tdn} f_x \Delta_x^n u_x.
\end{equation}
In order to compute the second sum on the right-hand side of \eqref{ecF1}, it is convenient to write $\eta_x + \eta_{x+b} -2 \eta_x \eta_{x+b}$ in terms of the centered variables $\overline{\eta}_x = \eta_x -u_x$:
\[
\begin{split}
\eta_x + \eta_{x+b} -2 \eta_x \eta_{x+b} 
		&= \overline{\eta}_x+ \overline{\eta}_{x+b} + u_x +u_{x+b} -\\
		&\hspace{30pt} -2 \overline{\eta}_x \overline{\eta}_{x+b} -2 u_x \overline{\eta}_{x+b} -2u_{x+b} \overline{\eta}_x -2 u_x u_{x+b}\\
		&= -2 \overline{\eta}_x \overline{\eta}_{x+b} +(1-2u_x) \overline{\eta}_{x+b} + (1-2 u_{x+b}) \overline{\eta}_x +\\
		&\hspace{30pt}+u_x +u_{x+b} -2 u_x u_{x+b}.
\end{split}
\]
Therefore, the second sum on the right-hand side of \eqref{ecF1} is equal to the sum of the three terms
\[
\sum_{\subind} n(f_{x+b}-f_x) F_b^n(x) (u_x+u_{x+b} -2 u_x u_{x+b}),
\]
\[
 \sum_{\subind} (\eta_x -u_x) \big( (1-2u_{x+b})F_b^n(x) n(f_{x+b}-f_x)
		+(1-2u_{x-b})F_b^n(x-b)n(f_x-f_{x-b})\big),
\]
and 
\[
-\sum_{\subind} 2 (\eta_x -u_x) (\eta_{x+b}-u_{x+b}) F_b^n(x) (u_{x+b} -u_x).
\]
Notice that 
\[
\sum_{x \in \tdn} u_x \Delta_x^n f + \sum_{\subind} n(f_{x+b}-f_x) F_b^n(x) (u_x+u_{x+b} -2 u_x u_{x+b}) = \sum_{x \in \tdn} f_x \mc L^n u_x,
\]
where $\mc L^n$ is the discrete approximation of the operator $u \mapsto \Delta u -2\nabla \cdot(u(1-u)F)$ defined in \eqref{conguillio}.

\section{Integration by parts formula}

\label{sG}

In this section we prove an estimate known in the literature as the {\em integration by parts formula}, see Lemma 7.2.1 in \cite{KipLan}. Differently to the usual setting, we need to derive this estimate using as reference measure the measures $\mu_t^n$, which are not invariant under the dynamics. This will introduce error terms that need to be carefully computed. For the reader's convenience, we repeat here some of the definitions introduced in Section \ref{s2}. Let $u: \tdn \to (0,1)$ be given and let $\mu$ be the measure
\[
\mu := \bigotimes_{x \in \tdn} \Bern(u_x).
\]
For each $x \in \tdn$, define
\[
\omega_x := \frac{\eta_x - u_x}{u_x(1-u_x)}.
\]
Let $f: \Omega_n \to [0,\infty)$ be a density with respect to $\mu$. Let $x,y \in \tdn$ and let $h: \Omega_n \to \bb R$ be such that $\nabla_{x,y} h$ is identically zero. Our objective will be to estimate $\int h(\omega_y - \omega_x) d \mu$ in terms of 
\[
\mc D_{x,y} \big(\sqrtt{f}; \mu \big) := \int \big( \nabla_{x,y} \sqrtt{f} \big)^2 d \mu.
\]
The first step is the following identity:

\begin{lemma}[Integration by parts]
\label{lG1}
Let $\mu$, $f$, $x,y$ and $h$ be as above. Then,
\[
\int h(\omega_y - \omega_x) f d \mu 
		= \int h s_{x,y} \nabla_{x,y} f d \mu - (u_y-u_x) \int h \omega_x \omega_y f d\mu,
\]
where
\[
s_{x,y}:= \frac{\eta_x(1-\eta_y)}{u_x(1-u_y)}.
\]
\end{lemma}
\begin{proof}
For any function $g: \Omega_n \to \bb R$, 
\[
\int g \nabla_{x,y} f d \mu = \int f \frac{\nabla_{x,y} (g \mu)}{\mu} d \mu.
\]
Since $\nabla_{x,y} h =0$,
\[
\int h g \nabla_{x,y} f d \mu = \int h f \frac{\nabla_{x,y} (g \mu)}{\mu} d \mu
\]
and therefore we only need to choose a proper function $g$. Taking $g = s_{x,y}$, we see that
\[
\frac{\nabla_{x,y}(g \mu)}{\mu} = \frac{\eta_y(1-\eta_x)}{u_y(1-u_x)} - \frac{\eta_x(1-\eta_y)}{u_x(1-u_y)}.
\]
By \eqref{ecA3.2}, 
\[
\frac{\nabla_{x,y}(g \mu)}{\mu} = \omega_y -\omega_x +(u_y-u_x) \omega_x \omega_y,
\]
which proves the lemma.
\end{proof}

The simple form of this lemma accounts for the choice of the variables $\omega_x$ as main variables in place of $\eta_x$. 

We will combine Lemma \ref{lG1} with the following estimate:

\begin{lemma}
\label{lG2}
Under the hypothesis of Lemma \ref{lG1}, for any $\beta >0$,
\[
\int h s_{x,y} \nabla_{x,y} f d \mu 
		\leq \beta \mc D_{x,y} \big( \sqrtt{f}; \mu \big) + \frac{4}{\eps_0 \beta} \int h^2 f d\mu,
\]
where $\eps_0>0$ is such that $\eps_0 \leq u_0 \leq 1-\eps_0$ for any $x \in \tdn$. 
\end{lemma}

\begin{proof}
Notice that for any $\alpha >0$, 
\[
\begin{split}
\nabla_{x,y} f = \nabla_{x,y} \sqrtt{f} \big( \sqrtt{f(\eta^{x,y})} + \sqrtt{f} \big)
		&\leq \frac{\alpha}{2} \big( \nabla_{x,y} \sqrtt{f} \big)^2 + \frac{1}{2 \alpha} \big( \sqrtt{f(\eta^{x,y})} + \sqrtt{f} \big)^2 \\
		& \leq \frac{\alpha}{2} \big( \nabla_{x,y} \sqrtt{f} \big)^2 + \frac{1}{\alpha} \big( f (\eta^{x,y})+f\big).
\end{split}
\]
Taking $\alpha = \frac{2 \beta}{|h s_{x,y}|}$, we see that
\begin{equation}
\label{ecG1}
\int h s_{x,y} \nabla_{x,y} f d \mu 
		\leq \beta \mc D_{x,y} \big( \sqrtt{f} ; \mu \big) + \frac{1}{2\beta} \int h^2 s_{x,y}^2 \big( f(\eta^{x,y}) -f \big) d \mu.
\end{equation}
The integral on the right-hand side of this estimate is equal to
\[
\frac{1}{2 \beta } \int h^2 \Big( s_{x,y}^2 + s_{y,x}^2 \frac{u_y(1-u_x)}{u_x (1-u_y)} \Big) f d\mu \leq \frac{4}{\eps_0^2 \beta} \int h^2 f d\mu,
\]
where we used the bound $u_x(1-u_x) \geq \frac{\eps_0}{2}$ to get the last estimate. Putting this estimate back into \eqref{ecG1}, the lemma is proved.
\end{proof}

Putting Lemmas \ref{lG1} and \ref{lG2} together and choosing $\beta = \delta n^2$, we obtain the following estimate:

\begin{lemma}
\label{lG3} 
Let $x,y \in \tdn$. Let $f: \Omega_n \to [0,\infty)$ be a density with respect to $\mu$ and let $h: \Omega_n \to \bb R$ be such that $\nabla_{x,y} h =0$. Then, for any $\delta >0$,
\[
\begin{split}
\int h(\omega_y - \omega_x) f d \mu 
		&\leq \delta n^2 \mc D_{x,y}\big( \sqrtt{f}; \mu\big) + \frac{4}{\delta \eps_0 n^2} \int h^2 f d\mu\\
		&\quad \quad \quad -(u_y-u_x) \int h \omega_x \omega_y f d \mu.
\end{split}
\]
\end{lemma}

\begin{remark}
In the case on which $u$ is constant, in all three previous lemmas the factor $u_y-u_x$ vanishes identically. This would have rendered the proof of Theorems \ref{t2} and \ref{t3} considerably easier, see \cite{JarMen}.
\end{remark}

\section{Entropy and concentration inequalities}

\label{sH}

In this section we collect two classes of classical inequalities which complement each other very well in our context. First we discuss how to estimate integrals in terms of entropy and exponential moments, and then we discuss how to estimate exponential moments using concentration inequalities.

\subsection{Entropy inequalities}
\label{sH1}

In this section we present some classical inequalities involving the entropy. For completeness, we present the proofs of these inequalities. We start with the variational formula for the entropy. In order to avoid integrability issues, we will only work on a finite set $\Omega$.

\begin{proposition}[Variational formula of the entropy]
\label{pH1.1}
Let $\mu$ be a measure on a finite space $\Omega$. Let $f$ be a density with respect to $\mu$. Then,
\begin{equation}
\label{ecH1.1}
H(f;\mu) := \int f \log f d \mu = \sup_{g: \Omega \to \bb R} \Big\{ \int f g d \mu - \log \int e^g d \mu\Big\}.
\end{equation}
\end{proposition}

\begin{proof}
The Legendre transform of $y \mapsto e^{y-1}$ is $x \mapsto x \log x$. Therefore,
\[
x \log x = \sup_{\theta \in \bb R} \big\{ \theta x - e^{\theta -1}\big\}.
\]
Using this formula for $f \log f$, we see that
\begin{equation}
\label{ecH1.2}
\int f \log f d \mu = \int \sup_{\theta \in \bb R} \big\{ \theta f - e^{\theta -1} \big\} d \mu = \sup_{g: \Omega \to \bb R} \Big\{ \int fg d \mu - \int e^{g-1} d \mu\Big\}.
\end{equation}
This formula is not \eqref{ecH1.1}, but it looks very similar. Let $g: \Omega \to \bb R$ be fixed. Notice that
\[
\sup_{\lambda \in \bb R} \Big\{ \int f(g+\lambda) d \mu - \int e^{g+\lambda -1} d \mu \Big\} = \int f g d \mu - \log \int e^g d \mu.
\]
Putting this estimate back into \eqref{ecH1.2}, the proposition is proved.
\end{proof}

The main application of this variational formula is to derive the following estimates:

\begin{proposition}
\label{lH1.2}
Let $\mu$ be a measure on a finite space $\Omega$ and let $f$ be a density with respect to $\mu$. Then,
\begin{itemize}

\item[i)] for any $\gamma >0$ and any $g: \Omega \to \bb R$,
\begin{equation}
\label{temuco}
\int f g d \mu \leq \frac{1}{\gamma} \Big( H(f; \mu) + \log \int e^{\gamma g} d \mu\Big),
\end{equation}

\item[ii)] for any $A \subseteq \Omega$, 

\begin{equation}
\label{temuco2}
\int_{A} f d \mu \leq \frac{H(f;\mu) + \log 2}{\log \mu(A)^{-1}}.
\end{equation}

\end{itemize}

\end{proposition}

\begin{proof}
In order to prove \eqref{temuco}, it is enough to take $\gamma g$ as a test function in \eqref{ecH1.1}. In order to prove \eqref{temuco2}, it is enough to choose $g = \mathds{1}_A$ and $\gamma = \log (1+\frac{1}{\mu(A)})$ in \eqref{temuco}.
\end{proof}

Proposition \eqref{lH1.2} is very useful in the case on which $g$ is a sum of random variables which are independent with respect to $\mu$. In our context, functions like $V(G)$ or $V^\ell(G)$ defined in \eqref{vina}, \eqref{curacavi} resp., are sums of {\em local} random variables, which are independent only if they are apart enough. Let us recall the definition of $\ell$-dependent random variables given in Section \ref{s2}. We say that a set $B \subseteq \tdn$ is $\ell$-sparse if $|y-x| \geq \ell$ for any $x \neq y \in B$. We say that a family $\{\xi_x; x \in \tdn\}$ is $\ell$-dependent if the random variables $\{\xi_x; x \in B\}$ are independent for any $\ell$-sparse set $B$. We have the following lemma.

\begin{lemma}
\label{lH1.3} 
For any $\ell \leq n/2$ there exists a partition $\{B;  i \in \mc I_\ell\}$ of $\tdn$ in at most $(d+1) \ell^d$ $\ell$-sparse sets.
\end{lemma}

\begin{proof}
Let us identify $\tdn$ with the set $\{0,1,\dots,n-1\}^d$. Recall that $\Lambda_\ell := \{0,\dots,\ell\}^d$. Let $n=a\ell+b$ be the division with rest of $n$ by $\ell$. Notice that $a \geq 2$. Let $z \in \tdn$. Considering the division with rest of each coordinate of $z$ by $\ell$, we can write $z = \ell x +y$ in a unique way, where $y \in \Lambda_\ell$ and $x \in \Lambda_a$. For $x \in \Lambda_a$, let $\mc H(x)$ be the cardinality of the set $\{i \in \{1,\dots,d\}; x_i =a\}$. For $y \in \Lambda_\ell$ and $j = 0,1,\dots,d$, let
\[
B_y^j := \big\{ z= \ell x + y; \mc H(x) =j \big\}.
\]
It can be verified that $\{B_y^j; y \in \Lambda_\ell, j = 0,1,\dots,d\}$ is a partition of $\tdn$ into $\ell$-sparse sets, which proves the lemma.
\end{proof}

\begin{remark}
The condition $\ell \< n/2$ is necessary, because for $n =2\ell-1$, any $\ell$-sparse set has a unique element.
\end{remark}

For $\ell$-dependent families $\{\xi_x; x \in \tdn\}$, \eqref{temuco} has the following form:

\begin{lemma}
\label{lH1.4} 
Let $\mu$ be a measure on a finite set $\Omega$. Let $f$ be a density with respect to $\mu$. Let $\{\xi_x; x \in \tdn\}$ be $\ell$-dependent with respect to $\mu$, with $\ell <n/2$. Then, for any $\gamma >0$,
\[
\int \sum_{x \in \tdn} \xi_x f d\mu
		\leq \frac{d+1}{\gamma} \Big( H(f; \mu) + \frac{1}{\ell^d} \sum_{x \in \tdn} \log \int e^{\gamma \ell^d \xi_x} d \mu \Big)
\]
and
\[
\Big|\int \sum_{x \in \tdn} \xi_x f d\mu \Big|
		\leq \frac{d+1}{\gamma} \Big( H(f; \mu) + \frac{1}{\ell^d} \sum_{x \in \tdn} \Big|\log \int e^{\gamma \ell^d \xi_x} d \mu\Big| \Big).
\]
\end{lemma}

\begin{proof}
Let $\{B_i; i \in \mc I_\ell\}$ be the partition obtained from Lemma \ref{lH1.3}. Using \eqref{temuco} with $\widetilde{\gamma} = \gamma {\ell^d}$, we see that
\[
\begin{split}
\int \sum_{x \in \tdn} \xi_x f d \mu 
		&=\sum_{i \in \mc I_\ell} \int \sum_{x \in B_i} \xi_x f d \mu 
		\leq \sum_{i \in \mc I_\ell} \frac{1}{\gamma \ell^d} \Big( H(f; \mu) + \log \int \exp\Big\{ \gamma \ell^d \!\!\!\sum_{x \in B_i} \xi_x \Big\} d \mu \Big)\\
		& \leq \frac{d+1}{\gamma} \Big( H(f ; \mu) + \frac{1}{\ell^d} \sum_{x \in \tdn} \log \int e^{\gamma \ell^d \xi_x } d \mu \Big),
\end{split}
\]
which proves the first inequality. The second inequality is proved in the same way.
\end{proof}

\begin{remark}
Although this lemma is fairly simple, it is convenient to include it as a reference, since it is used uncountable times along the article. 
\end{remark}

\subsection{Tail and moment bounds}
\label{sH3}
Estimate \eqref{temuco2} will provide a way to obtain tail estimates for various random variables of interest. The following propositions will be very useful to transform these tail estimates into moment estimates.

\begin{proposition}
\label{lH3.2}
Let $\xi$ be a non-negative random variable and let $f$ be a non-decreasing function of class $\mc C^1$. Then,
\[
E[f(\xi)] \leq f(0) + \int_0^\infty f'(\lambda) P(\xi >\lambda) d\lambda.
\]
\end{proposition}

\begin{proof}
It follows from the integration-by-parts theorem for Stieltjes integrals.
\end{proof}
\
We will use this proposition to obtain moment estimates from tail estimates:

\begin{proposition}
\label{lH1.5} Let $K>0$, $p>1$ be given and let $\xi$ be a random variable such that
\[
P\big( |\xi| > \lambda\big) \leq \frac{K}{\lambda^p}
\]
for any $\lambda >0$. Then, for any $0<q<p$ there exists $c =c(p,q)$ such that
\[
E[|\xi|^q] \leq c K^{q/p}.
\]
\end{proposition}
\begin{proof}
By Proposition \ref{lH3.2}, 
\[
\begin{split}
E[|\xi|^q] 
		&= q \int_0^\infty \lambda^{q-1} P\big( |\xi| > \lambda \big) d \lambda \leq q \int_0^{K^{1/q}} \lambda^{q-1} d\lambda + q \int_{K^{1/q}}^\infty K \lambda^{q-p-1} d \lambda \\
		&\leq \frac{p}{p-q} K^{q/p},
\end{split}
\]
as we wanted to show.
\end{proof}

\subsection{Concentration inequalities}

\label{sH2}

Recall expressions \eqref{curacavi} for $V^\ell(G)$ and \eqref{chimbarongo} for $W^{\ell,b'}(G)$. If one wants to use Lemma \ref{lH1.4} to estimate integrals of the form $\int V^\ell(G) f d \mu$, we need to know how to estimate exponential moments of products and squares of sums of bounded random variables. This is what is accomplished by what is known in the literature as {\em concentration inequalities}, see \cite{BouLugMas}. 

We say  that a real-valued random variables $\xi$ is subgaussian of order $\sigma^2$ if
\[
\log E[e^{\theta \xi}] \leq \tfrac{1}{2} \sigma^2 \theta^2 \text{ for any } \theta \in \bb R.
\]
For our purposes, the most important property of subgaussian random variables is the following:

\begin{proposition}
\label{lH2.1} Let $\xi$ be subgaussian of order $\sigma^2$. Then,
\[
E[e^{\gamma \xi^2}] \leq 3 \text{ for any } \gamma \leq (4 \sigma^2)^{-1}.
\]
\end{proposition}
\begin{proof}
This is a simple application of Chernoff's method. For any $\lambda >0$, 
\[
\log P(\xi > \lambda) \leq \log E[e^{\theta \xi}] - \theta \lambda \leq \tfrac{1}{2} \sigma^2 \theta^2 - \theta \lambda.
\]
Taking $\theta = \frac{\lambda}{\sigma^2}$, we see that $P(\xi > \lambda) \leq e^{-\lambda^2/2\sigma^2}$. Repeating the computation for $P(\xi <-\lambda)$, we conclude that
\[
P\big( |\xi| > \lambda\big) \leq 2 e^{-\lambda^2/2 \sigma^2}.
\]
Using Proposition \ref{lH3.2}, we see that
\[
\begin{split}
E[e^{\gamma \xi^2}] 
		&\leq 1 + \int_0^\infty 4 \gamma \lambda \exp\Big\{ - \lambda^2 \Big(\frac{1}{2\sigma^2} -\gamma\Big) \Big\} d \lambda\\
		&\leq 1- \frac{2 \gamma \exp\big\{ -\lambda^2 \big( \frac{1}{2\sigma^2} -\gamma\big) \big\}}{\frac{1}{2\sigma^2} - \gamma}\bigg|_{x=0}^\infty\\
		&\leq 1+ \frac{2\gamma}{\frac{1}{2\sigma^2} -\gamma} = \frac{1+2 \gamma \sigma^2}{1-2\gamma \sigma^2}.
\end{split}
\]
This expression is increasing in $\gamma$. For $\gamma = (4 \sigma^2)^{-1}$, the right-hand side of this estimate is equal to $3$, which proves the proposition.
\end{proof}

\begin{remark}
The numeric constant $3$ is not relevant. In principle, it can be proved that $E[e^{\gamma \xi}]$ converges to $1$ linearly in $\gamma$, but this type of estimate would not improve any of our results. Notice that the estimate explodes for $\gamma = \frac{1}{2} \sigma^2$. This is not an accident, since the square of a Gaussian random variable does not have finite exponential moments of all orders.
\end{remark}

We will also need to estimate exponential moments of products of subgaussian random variables:

\begin{lemma}
\label{lH2.2}
Let $\xi_i$ be subgaussian random variables of order $\sigma^2_i$, $i=1,2$. Then, for any $\gamma \leq (4 \sigma_1 \sigma_2)^{-1}$,
\[
E[e^{\gamma \xi_1 \xi_2}] \leq 3.
\]
\end{lemma}
\begin{proof}
This estimate follows from Cauchy-Schwartz inequality: first we notice that for any $\beta >0$,
\[
\begin{split}
E[e^{\gamma \xi_1 \xi_2}] 
		\leq E \Big[ \exp\Big\{ \frac{\beta \gamma \xi_1^2}{2} + \frac{\gamma \xi_2^2}{2 \beta} \Big\} \Big]
		&\leq E[e^{\beta \gamma \xi_1^2}]^{1/2} E [ e^{\gamma \xi_2^2/\beta}]^{1/2}\\
		&\leq \bigg( \frac{1+2 \gamma \sigma^2_1\beta}{1-2\gamma \sigma_1^2\beta} \cdot\frac{1+2 \gamma \sigma^2_2/\beta}{1-2\gamma \sigma_2^2/\beta}\bigg)^{1/2}.
\end{split}\]
Taking $\beta = \sigma_2/\sigma_1$, we see that
\[
E[e^{\gamma \xi_1 \xi_2}] \leq \frac{1+2 \gamma \sigma_1 \sigma_2}{1-2\gamma \sigma_1 \sigma_2} \leq 3
\]
if $\gamma \leq (4 \sigma_1 \sigma_2)^{-1}$, as we wanted to prove.
\end{proof}

Now we explain how to prove that sums of bounded random variables are subgaussian. It is clear that any mean-zero, bounded random variable is subgaussian, but the real question is how to estimate $\sigma^2$ in an efficient way. We start with Hoeffding's lemma:

\begin{lemma}[Hoeffding]
\label{lH2.3}
Let $\xi$ be a random variable with values on the interval $[0,1]$. Let $\rho = E[\xi]$. Then,
\[
\log E[e^{\theta (\xi -\rho)} ] \leq \tfrac{1}{8} \theta^2.
\]
\end{lemma}
Proofs of this lemma are easy to find; even Wikipedia has a reasonably well explained proof. Therefore, we omit the proof. Notice that in particular, $\xi -\rho$ is subgaussian of order $\frac{1}{4}$. As mentioned above, the importance of this lemma is that the order is independent of the law of $\xi$. Here and below, $\mu$ is a measure satisfying the hypothesis of Lemma \ref{l1}. Since $\eta_x \in [0,1]$, this lemma has the following consequence:

\begin{lemma}
\label{lH2.4} 
For any finite set $A$ and any $x \in \tdn$, $\omega_{x+A}$ is subgaussian of order $C(A,\eps_0)$. Moreover, if the cardinality of $A$ is equal to $\ell$, we can choose $C(A,\eps_0) = (2/\eps_0)^{-2 \ell}$.
\end{lemma}
\begin{proof}
It is enough to observe that $u_x(1-u_x) \geq \eps_0/2$ for any $x \in \tdn$. 
\end{proof}

Holder's inequality can be used to estimate exponential moments of sums. For later use we will state this as a lemma.

\begin{lemma}
\label{Holder}
Let $\{\zeta_i; i \in \mc I\}$ be a family of random variables such that $E[e^{\zeta_i}]<+\infty$ for any $i \in \mc I$. Let $k$ be the cardinality of $\mc I$. For any $m \in \bb N$ such that $m \geq k$,
\[
\log \int e^{\sum_{i \in \mc I} \zeta_i} d \mu \leq \sum_{i \in \mc I} \frac{1}{m} \log \int e^{m \zeta_i} d \mu.
\]
\end{lemma}
\begin{proof}
Defining $\zeta_i = 0$ if $i \notin \mc I$, we can assume that the cardinality of $\mc I$ is $m$. By Holder's inequality,
\[
\int \prod_{i \in \mc I} e^{\zeta_i} d \mu \leq \prod_{i \in \mc I} \bigg( \int e^{m \zeta_i} d \mu \bigg)^{1/m}.
\]
Taking logarithms, the lemma is proved.
\end{proof}

The previous lemma shows in particular that sums of subgaussian random variables are also subgaussian. When some independence is present, the following lemma explains how to obtain a better bound on the order of the sum.

\begin{lemma}
\label{lH2.5}
Let $\{\xi_x; x \in \tdn\}$ be $\ell$-dependent. Assume that for any $x \in \tdn$, $\xi_x$ is subgaussian of order $\sigma^2_x$. Then for any $f: \tdn \to \bb R$, 
\[
\sum_{x \in \tdn} f_x \xi_x \text{ is subgaussian of order }
(d+1) \ell^d \sum_{x \in \tdn} \sigma^2_x f_x.
\]
\end{lemma}

\begin{proof}
Let $\{B_i; i \in \mc I_\ell\}$ be the partition of $\tdn$ obtained from Lemma \ref{lH1.3}. Then,
\[
\begin{split}
\log E\Big[ \exp\Big\{ \theta \sum_{x \in \tdn} f_x \xi_x\Big\} \big] 
		&= \log E \Big[ \exp\Big\{ \theta \sum_{i \in \mc I_\ell} \sum_{x \in B_i} f_x \xi_x\Big\} \Big]\\
		&\leq \frac{1}{(d+1)\ell^d} \sum_{i \in \mc I_\ell} \log \int \exp\Big\{ (d+1) \ell^d \theta \sum_{x \in B_i} f_x \xi_x\Big\} d \mu\\
		&\leq \frac{1}{(d+1) \ell^d} \sum_{i \in \mc I_\ell} \sum_{x \in B_i} \log \int \exp\big\{ (d+1) \ell^d \theta f_x \xi_x \big\} d \mu\\
		& \leq (d+1) \ell^d \theta^2 \sum_{x \in \tdn} \tfrac{1}{2}\sigma^2_x f_x^2,
\end{split}
\]
as we wanted to prove. Here we used Lemma \ref{lH2.5} in the second line and independence in the third line.
\end{proof}

A very useful generalization of Lemma \ref{lH2.2} is known in the literature as the {\em Hanson-Wright inequality}. We will state this generalization in the precise form needed in this article.

\begin{lemma}[Hanson-Wright]
\label{lH2.6} 
Let $\{(\xi_x, \widetilde{\xi}_x); x \in \tdn\}$ $\ell$-dependent random variables, and let $F: \tdn \times \tdn \to \bb R$ be such that $F_{x,y} =0$ whenever $|x-y|<\ell$. 
%Assume that the random variables $\{\widetilde{\xi}_x; x \in \tdn\}$ are independent. 
Assume
%as well 
that the random variables $\xi_x$, $\widetilde{\xi}_x$ are subgaussian of order $\sigma^2_x$, $\widetilde{\sigma}^2_x$ for any $x \in \tdn$. Then, there exists $C = C(\ell)$ such that
\begin{equation}
\label{ecH2.1}
\int \exp\Big\{ \gamma \!\!\!\!\! \sum_{x,y \in \tdn} \!\!\!\xi_x \widetilde{\xi}_y F_{x,y} \Big\} d \mu \leq 3
\end{equation}
for any 
\[
\gamma \leq \bigg( C(\ell) \!\!\!\!\!\sum_{x,y \in \tdn} \!\!\!\sigma^2_x \widetilde{\sigma}^2_y F_{x,y}\bigg)^{-1/2}.
\]
\end{lemma}

\begin{remark}
We were not able to find a reference with a version of Hanson-Wright inequality for $\ell$-dependent random variables that we could directly use, nor a simple way to derive it from its classical version for independent random variables. Therefore, we need to present a complete proof. Our proof is an adaptation of the proof in \cite{RudVer}.
\end{remark}

\begin{proof}
We will make repeated use of the inequality $e^{E[X]} \leq E[e^X]$. The idea is to use {\em decoupling}. Let $\{B_i; i \in \mc I_\ell\}$ be the partition obtained from Lemma \ref{lH1.3}. Let ${\bf B}$, $\widetilde{\bf B}$ two random sets obtained in the following way. First we choose an index $i \in \mc I_\ell$ uniformly at random. Then, we choose ${\bf B}$ uniformly at random among all subsets of $B_i$. Then we choose $\widetilde{\bf B}$ uniformly at random among the subsets of $\{ y \in \tdn; |y-x| \geq \ell \text{ for any } x \in {\bf B}\}$. Let ${\bf P}$ be the law of $({\bf B},{\bf \widetilde{B}})$ and let $\bf E$ the expectation with respect to $\bf P$. Let us define
\[
q_{x,y} = {\bf P} \big( x \in {\bf B},  y \in {\bf \widetilde{B}}\big)
\]
Recall that whenever $F_{x,y} \neq 0$, $|x-y| \geq \ell$ and in particular
\[
q_{x,y} \geq \frac{1}{(d+1) \ell^d} \cdot \frac{1}{2}\cdot \frac{1}{2^d}\cdot \frac{1}{2} = \frac{1}{(d+1)2^{d+2} \ell^d}.
\]
This estimate need some explanation. The first term is smaller than the probability of choosing the right index $i$. The second term is the probability of choosing $x$ in $\bf B$. The third term is smaller than the probability of not choosing any of the points in $B_i$ which are distance $\ell$ or less from $y$, and the fourth term is the probability of choosing $y$, once it is able, in $\bf \widetilde{B}$.
Notice that
\[
\begin{split}
\sum_{x,y \in \tdn} \xi_x \widetilde{\xi}_y F_{x,y} 
		&= {\bf E} \Big[ \sum_{x,y \in \tdn} \xi_x \widetilde{\xi}_y F_{x,y} \frac{\mathds{1}\big( x \in {\bf B}, y \in {\bf \widetilde{B}}\big)}{q_{x,y}}\Big]
		= {\bf E} \Big[ \sum_{\substack{x \in {\bf B}\\y \in {\bf \widetilde{B}}}} \xi_x \widetilde{\xi}_y \frac{F_{x,y}}{q_{x,y}} \Big].
\end{split}
\]
Therefore, the left-hand side of \eqref{ecH2.1} is bounded by
\[
{\bf E} \Big[ \int \exp\Big\{ \gamma \sum_{\substack{x \in {\bf B}\\y \in {\bf \widetilde{B}}}}
			 \xi_x \widetilde{\xi}_y \frac{F_{x,y}}{q_{x,y}} \Big\} d \mu \Big]
\]
The main point of this bound is that now the families $\{\xi_x; x \in {\bf B}\}$ and $\{ \widetilde{\xi}_y; y \in {\bf \widetilde{B}}\}$ are independent. Therefore, conditioning on $\widetilde{\xi}_y$ we can use Lemma \ref{lH2.5} to show that this expectation is bounded by
\[
{\bf E} \Big[ \int \exp\Big\{ \sum_{x \in {\bf B}} \frac{\gamma^2 \sigma_x^2}{2} \Big( \sum_{y \in {\bf \widetilde{B}}} \widetilde{\xi}_y \frac{F_{x,y}}{q_{x,y}}\Big)^2 \Big\} d \mu \Big].
\]
Let $\{p(x); x \in {\bf B}\}$ be a measure to be chosen in a few lines. Rewriting the expectation above as
\[
{\bf E} \Big[ \int \exp\Big\{\frac{\gamma^2}{2} \sum_{x \in {\bf B}} \sigma_x^2 \frac{p(x)}{p(x)} \Big( \sum_{y \in {\bf \widetilde{B}}} \widetilde{\xi}_y \frac{F_{x,y}}{q_{x,y}}\Big)^2 \Big\} d \mu \Big].
\]
we see that it is bounded by
\[
{\bf E} \Big[ \int \sum_{x \in {\bf B}} p(x) \exp\Big\{ \frac{\gamma^2\sigma_x^2}{2p(x)} \Big( \sum_{y \in {\bf \widetilde{B}}} \widetilde{\xi}_y \frac{F_{x,y}}{q_{x,y}}\Big)^2 \Big\} d \mu \Big].
\]
By Proposition \ref{lH2.1} and Lemma \ref{lH2.5} applied to the variables $\{\xi_y; y \in {\bf \widetilde{B}}\}$, the integral
\[
 \int \exp\Big\{ \frac{\gamma^2\sigma_x^2}{2p(x)} \Big( \sum_{y \in {\bf \widetilde{B}}} \widetilde{\xi}_y \frac{F_{x,y}}{q_{x,y}}\Big)^2 \Big\} d \mu
\]
is bounded by $3$ if
\[
\frac{\gamma^2 \sigma^2_x}{2 p(x)} \leq \bigg( 4(d+1)\ell^d \sum_{y \in {\bf \widetilde{B}}} \frac{\widetilde{\sigma}^2_y F_{x,y}^2}{q_{x,y}^2}\bigg)^{-1}.
\]
We want this bound to be satisfied for any $x \in {\bf B}$. The optimal choice of $p(x)$ in order to maximize the value of $\gamma$ is
\[
p(x) =2(d+1) \ell^d \gamma ^2 \sigma_x^2 \sum_{y \in {\bf \widetilde{B}}} \frac{\widetilde{\sigma}^2_y F_{x,y}^2}{q_{x,y}^2}, \text{ with } \gamma \leq \bigg( 2(d+1)\ell^d \sum_{\substack{x \in {\bf B}\\y \in {\bf \widetilde{B}}}} \frac{\sigma^2_x \widetilde{\sigma}^2_y F_{x,y}^2}{q_{x,y}^2}\bigg)^{-1/2}.
\]
This restriction holds regardless of the chosen set $\bf B$ if we take
\[
\gamma \leq \bigg( 2 (d+1)\ell^d\!\!\!\!\! \sum_{x,y \in \tdn} \!\!\! \frac{\sigma_x^2 \widetilde{\sigma}_y^2 F_{x,y}^2}{q_{x,y}^2} \bigg)^{-1/2}.
\]
Since $q_{x,y}^{-2} \leq (d+1)^2 2^{2(d+2)} \ell^{2d}$, the lemma is proved with $C(\ell) = (d+1)^3 2^{2d+5}\ell^{3d}$.
\end{proof}

\begin{remark}
In our context, the value of the constant $C(\ell)$ will not be relevant.
\end{remark}

\section{Proof of the flow lemma}
\label{sI}

In this section we prove Lemma \ref{flow}. Instead of proving this lemma directly, we will prove the following lemma, which is simpler. Let us recall the definition of flow given in Section \ref{s2} after \eqref{curacavi}. Recall that $\Lambda_\ell$ denotes the cube $\{0,1,\dots,\ell-1\}^d$ and that $p_\ell$ denotes the uniform measure on $\Lambda_\ell$. 

\begin{lemma}
\label{lI1} There exists a finite constant $C=C(d)$ such that for any $\ell \in \bb N$ there exists a flow $\psi^\ell$ connecting $p_\ell$ to $p_{\ell-1}$ with support contained in $\Lambda_\ell$, such that
\[
\big| \psi^\ell(x;b) \big| \leq C \ell^{-d}
\]
for any $x \in \Lambda_\ell$ and any $b \in \mc B$.
\end{lemma}

\begin{proof}
For $k=0,1,\dots,d$, let $\Lambda_\ell^k$ be the set of sites in $\Lambda_\ell$ with exactly $k$ coordinates equal to $\ell-1$. The idea is to successively transport mass from $\Lambda_\ell^k$ to $\Lambda_\ell^{k-1}$ in a uniform way. In that case, the final result has to be uniform by construction.
Notice that flows have an Abelian structure: if $\phi_{1,2}$ connects $p_1$ to $p_2$ and $\phi_{2,3}$ connects $p_2$ to $p_3$, then $\phi_{1,2} + \phi_{2,3}$ connects $p_1$ to $p_3$. We start moving the mass at $\Lambda_\ell^d =\{(\ell-1,\dots,\ell-1)\}$ along the $d$ segments of length $\ell-1$ that form $\Lambda_\ell^{d-1}$. This is accomplished by defining 
\[
\psi^\ell_d(x-kb;b) = -\frac{\ell-k}{\ell-1} \cdot \frac{1}{d \ell^d}
\]
for $x = (\ell-1,\dots,\ell-1)$, $k =1,\dots,\ell-1$ and $b \in \mc B$, and $\psi^\ell_d(y;b)=0$ otherwise. The factor $\frac{1}{\ell^d}$ is the mass of $\Lambda_\ell^d$ with respect to $p_\ell$, the factor $\frac{1}{d}$ divides the mass uniformly among each segment of $\Lambda_\ell^{\ell-1}$, and the factor $\frac{\ell-k}{\ell-1}$ dsitributes the mass uniformly on each segment. The flow $\psi^\ell_d$ defined in this way connects $p_\ell$ to a measure $p_\ell^{d-1}$ supported in $\Lambda_\ell^0 \cup \dots \cup \Lambda_\ell^{d-1}$, equal to $p_\ell$ in $\Lambda_\ell^0 \cup \dots \cup \Lambda_\ell^{d-2}$ and equidistributed in $\Lambda_\ell^{d-1}$. 

The idea is to iterate this construction. Let $p_\ell^k$ the measure in $\Lambda_\ell^0 \cup \dots \cup \Lambda_\ell^k$ defined by the conditions
\begin{itemize}
\item $p_\ell^k$ is equal to $p_\ell$ in $\Lambda_\ell^0 \cup \dots \cup \Lambda_\ell^{k-1}$,

\item $p_\ell^k$ is equidistributed in $\Lambda_\ell^k$. 

\end{itemize}

For $x \in \Lambda_\ell^k$, $j=1,\dots,\ell-1$ and $b \in \mc B$ such that $x\cdot b =\ell-1$ (there are exactly $k$ of such indices $b$), let us define
\[
\psi^\ell_k(x-jb;b) = - \frac{\ell-j}{\ell-1} \cdot \frac{p_\ell^k(x)}{k}.
\]
Otherwise we take $\psi^\ell_k(y;b) =0$. The flow $\psi^\ell_k$ defined in this way connects $p_\ell^k$ to $p_\ell^{k-1}$. In this way we have constructed a sequence $\{\psi^\ell_k; k=d,d-1,\dots,1\}$ of flows connecting $p_\ell^k$ to $p_\ell^{k-1}$. We conclude that the flow
\[
\psi^\ell := \sum_{k=1}^d \psi_k^\ell
\]
connects $p_\ell$ to $p_\ell^0=p_{\ell-1}$. Since the supports of the flows $\psi_k^\ell$ are disjoint,
\[
\|\psi^\ell\|_\infty = \sup_{1 \leq k \leq d} \|\psi_k^\ell\|_\infty.
\]
The constants $a_k:= p_\ell^k(x)$ for $x \in \Lambda_\ell^k$ are not difficult to estimate. The cardinality of $\Lambda_\ell^k$ is equal to $\binom{d}{k} (\ell-1)^{d-k}$. Therefore,
\[
\frac{1}{\ell^d} \sum_{i=0}^{k-1}  \binom{d}{k} (\ell-1)^{d-i} + \binom{d}{k} (\ell-1)^{d-k} a_k =1 = \frac{1}{\ell^d} \sum_{i=1}^d \binom{d}{i} (\ell-1)^{d-i}.
\]
Therefore, 
\[
a_k = \binom{d}{k}^{-1} \sum_{i=k}^d \binom{d}{i} (\ell-1)^{-(i-k)} \leq \binom{d}{k}^{-1} \sum_{i=k}^d \binom{d}{i}.
\]
Notice that this last expression does not depend on $\ell$. Since $\|\psi_k^\ell\|_\infty \leq \frac{a_k}{k}$, we conclude that
\[
\|\psi^\ell\|_\infty \leq \sup_{1 \leq k \leq d} \frac{1}{k} \binom{d}{k}^{-1} \sum_{i=k}^d \binom{d}{i},
\]
as we wanted to show.
\end{proof}

Now that we know how to connect $p_\ell$ to $p_{\ell-1}$, it is enough to add these flows to connect $p_\ell$ to $p_1$, which is equal to the point mass at $0$. Recall the definition of  $g_d(\cdot)$ given in Theorem \ref{t2}. We have the following result:

\begin{lemma}
\label{lI2}
There exists a finite constant $C=C(d)$ such that for any $\ell \in \bb N$ there exists a flow $\widetilde{\psi}^\ell$ supported in $\Lambda_\ell$, connecting the point mass at the origin to the uniform measure $p_\ell$ in $\Lambda_\ell$ and such that
\[
\sum_{\substack{x \in \Lambda_\ell \\ b \in \mc B}} \widetilde{\psi}^\ell(x;b)^2 \leq C g_d(\ell); 
\quad \sum_{\substack{x \in \Lambda_\ell \\ b \in \mc B}} \big|\widetilde{\psi}^\ell(x;b)\big| \leq C \ell.
\]
\end{lemma}

\begin{proof}
In $d=1$ it is enough to define
\[
\widetilde{\psi}^\ell(x;e_1) =
\left\{
\begin{array}{c@{\;;\;}l}
\frac{\ell-x}{\ell} & x=0,1,\dots,\ell-1\\
0& \text{otherwise.}
\end{array}
\right.
\]
For $d \geq 2$, we define
\[
\widetilde{\psi}^\ell = - \sum_{k=1}^\ell \psi^k
\]
For $x,b$ such that $x+b \in \Lambda_\ell^k \setminus \Lambda_\ell^{k-1}$, $\psi^i(x;b)=0$ if $i \leq k-1$. Therefore,
\[
\big| \widetilde{\psi}^\ell(x;b) \big| = \Big| \sum_{i=k}^\ell \psi^i(x;b)\Big| \leq C(d) \sum_{i=k}^\ell \frac{1}{i^d} \leq \frac{C(d)}{k^{d-1}}
\]
The number of couples $(x;b)$ such that $x+b \in \Lambda_\ell^k \setminus \Lambda_\ell^{k-1}$ is bounded by $C(d) k^{d-1}$. Therefore,
\[
\sum_{\substack{x \in \Lambda_\ell \\ b \in \mc B}} \widetilde{\psi}^\ell(x;b)^2 
		\leq C(d) \sum_{k=1}^\ell \frac{k^{d-1}}{k^{2(d-1)}} = C(d) \sum_{k=1}^\ell \frac{1}{k^{d-1}} \leq C(d) g_d(\ell)
\]
and
\[
\sum_{\substack{x \in \Lambda_\ell \\ b \in \mc B}} \big|\widetilde{\psi}^\ell(x;b)\big| \leq C(d) \sum_{k=1}^\ell \frac{k^{d-1}}{k^{d-1}} = C(d) \ell,
\]
as we wanted to show.
\end{proof}

Starting from Lemma \ref{lI2}, it is not difficult to prove Lemma \ref{flow}: since $q_\ell = p_\ell \ast p_\ell$, the flow $\phi^\ell$ defined as
\[
\phi^\ell(x;b) = \sum_{z \in \bb Z^d} \widetilde{\psi}^\ell(x-z;b) p_\ell(z)  
\]
for any $x \in \bb Z^d$ and any $b \in \mc B$ connects the point mass at $0$ with $q_\ell$. It also has support on $\Lambda_{2\ell-1}$. Moreover,
\[
\begin{split}
\sum_{\substack{x \in \bb Z^d \\ b \in \mc B}} \phi^\ell(x;b)^2 
		&= \sum_{\substack{x \in \bb Z^d \\ b \in \mc B}}\Big( \sum_{z \in \bb Z^d} \widetilde{\psi}^\ell(x-z;b) p_\ell(z)\Big)^2\\
		&\leq\sum_{\substack{ x,z \in \bb Z^d\\b \in \mc B}} \widetilde{\psi}^\ell(x-z;b)^2 p_\ell(z) = 	
		\sum_{\substack{x \in \bb Z^d \\ b \in \mc B}} \widetilde{\psi}^\ell(x;b)^2,
\end{split}
\]
and similarly
\[
\sum_{\substack{x \in \bb Z^d \\ b \in \mc B}} \big|\phi^\ell(x;b)\big| \leq \sum_{\substack{x \in \bb Z^d \\ b \in \mc B}} \big|\widetilde{\psi}^\ell(x;b)\big|,
\]
which proves Lemma \ref{flow}.

\section*{Acknowledgements}

\noindent
M.J.~would like to thank the warm hospitality of Leiden University, where this work was initiated, and of the Isaac Newton Institute of Mathematical Sciences, where this work was finished.
M.J.~and O.M.~were partially supported by NWO Gravitation Grant 024.002.003-NETWORKS. M.J.~acknowledges CNPq for its support through the Grant 305075/2017-9 and ERC for its support through the European Unions Horizon 2020 research and innovative programme (Grant Agreement No. 715734).

\bibliographystyle{plain}

\end{document}